\documentclass[12pt,twoside,a4paper]{article}

\usepackage{fullpage}
\usepackage{enumerate}
\usepackage{subfigure}
\usepackage{xcolor}

\usepackage[utf8]{inputenc}
\usepackage[T1]{fontenc}
\usepackage[english]{babel}
\usepackage{enumitem}
\usepackage{amsmath,amsfonts,amssymb,graphicx,bbm,stmaryrd}
\usepackage{stackrel}
\usepackage{verbatim}

\usepackage{tabularx}

\usepackage{hyperref}
\definecolor{colorlinks}{RGB}{0, 24, 168}
\definecolor{colorcites}{RGB}{124, 10, 2}
\hypersetup{
	colorlinks=true,
	linkcolor=colorlinks,
	citecolor=colorcites,
	urlcolor=colorlinks,
	pdfborder={0 0 0}
}

\usepackage{amsthm}
\usepackage{constants}
\newconstantfamily{small}{symbol=c}

\theoremstyle{plain}
\newtheorem{theorem}{Theorem}

\newtheorem{proposition}{Proposition}[section]
\newtheorem{lemma}[proposition]{Lemma}
\newtheorem{corollary}[proposition]{Corollary}

\newtheorem{question}{Question}
\newtheorem{claim}{Claim}

\newtheorem*{remark}{Remark}

\theoremstyle{definition}
\newtheorem*{definition}{Definition}

\usepackage{mathtools}

\usepackage{dsfont}

\usepackage{mathtools,bbm}
\usepackage{stackrel}

\newcommand{\calC}{\mathcal{C}}
\newcommand{\calD}{\mathcal{D}}
\newcommand{\calE}{\mathcal{E}}
\newcommand{\calF}{\mathcal{F}}
\newcommand{\calG}{\mathcal{G}}

\newcommand{\bbE}{\mathbb{E}}

\newcommand{\bbN}{\mathbb{N}}

\newcommand{\bbP}{\mathbb{P}}

\newcommand{\bbR}{\mathbb{R}}

\newcommand{\bbT}{\mathbb{T}}

\newcommand{\bbZ}{\mathbb{Z}}

\newcommand{\Z}{\mathbb{Z}}

\newcommand{\dist}{\operatorname{dist}}

\newcommand{\R}{\mathbb{R}}

\newcommand{\eps}{\varepsilon}

\renewcommand{\int}{\mathrm{in}}

\title{On the transition between the disordered and antiferroelectric phases of the $6$-vertex model}
\author{Alexander Glazman\thanks{Universit{\"a}t Innsbruck, Innsbruck, Austria. \url{alexander.glazman@uibk.ac.at}} \and Ron Peled\thanks{School of Mathematical Sciences, Tel Aviv University, Israel. \url{peledron@tauex.tau.ac.il}}}

\date{October 6, 2023}

\begin{document}

\maketitle

\begin{abstract}
The symmetric six-vertex model with parameters~$a,b,c>0$ is expected to exhibit different behavior in the regimes $a+b<c$ (antiferroelectric), $|a-b|<c\leq a+b$ (disordered) and $|a-b|>c$ (ferroelectric). In this work, we study the way in which the transition between the regimes $a+b=c$ and $a+b<c$ manifests.

When~$a+b<c$, we show that the associated height function is localized and its extremal periodic Gibbs states can be parametrized by the integers in such a way that, in the~$n$-th state, the heights~$n$ and~$n+1$ percolate while the connected components of their complement have diameters  with exponentially decaying tails.
When~$a+b=c$, the height function is delocalized.

The proofs rely on the Baxter--Kelland--Wu coupling between the six-vertex and the random-cluster models and on recent results for the latter.
An interpolation between free and wired boundary conditions is introduced by modifying cluster weights.
Using triangular lattice contours ($\mathbb{T}$-circuits), we describe another coupling for height functions that in particular leads to a novel proof of the delocalization at~$a=b=c$.

Finally, we highlight a spin representation of the six-vertex model and obtain a coupling of it to the Ashkin--Teller model on~$\bbZ^2$ at its self-dual line~$\sinh 2J = e^{-2U}$.
When~$J<U$, we show that each of the two Ising configurations exhibits exponential decay of correlations while their product is ferromagnetically ordered.
\end{abstract}

\tableofcontents

\section{Introduction}

The six-vertex model is a classical model in statistical mechanics, which was initially introduced by Pauling~\cite{Pau35} in 1935 to study the structure of ice in three dimensions. A two-dimensional version as well as ferroelectric (Slater~\cite{Sla41}) and antiferroelectric (Rys~\cite{Rys63}) variants were later introduced, see~\cite[Chapter 8]{Bax82}, \cite{LieWu72}, and~\cite{Res10} for introductory texts. In this work we discuss the two-dimensional six-vertex model, whose configurations are orientations of edges of the square grid (graphically indicated by arrows on the edges) which satisfy the ice rule: at each vertex, there are exactly two outgoing and two incoming arrows, yielding six possible local configurations. The local possibilities are assigned nonnegative weights~$a_1$, $a_2$, $b_1$, $b_2$, $c_1$, $c_2$ as in Figure~\ref{fig:6v-arrow-config} and, in finite domains, the probability of each arrow configuration is proportional to the product of local weights (see Section~\ref{sec:six vertex results}).

\begin{figure}
	\begin{center}
		\includegraphics[width=\textwidth]{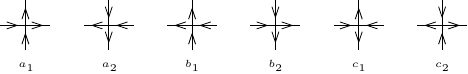}
	\end{center}
	\caption{The six possible arrow orientations at a vertex which satisfy the ice-rule (exactly two outgoing and two incoming arrows) and their associated weights.}
	\label{fig:6v-arrow-config}
\end{figure}

The model is typically studied under the assumption that the weights are invariant to reversal of all arrows, that is, $a_1 = a_2$, $b_1 = b_2$, $c_1 = c_2$ (zero external electric field), with the three weights termed $a,b,c$. Following Yang--Yang~\cite{YanYan66}, Lieb~\cite{Lie67c, Lie67a, Lie67b, lieb1967exact} and Sutherland~\cite{sutherland1967exact} who found an expression for the free energy using the Bethe ansatz~\cite{Bet31}, it is predicted that the behaviour of the model is governed by the value of
\begin{equation*}
  \Delta:= \tfrac{a^2+b^2 - c^2}{2ab}
\end{equation*}
with the following distinguished regimes:
\begin{itemize}
  \item $\Delta > 1$ (equivalently, $c<|a-b|$): the \emph{ferroelectric phase}, closely related to the stochastic six-vertex model; see~\cite{BorCorGor16} and references therein.
  \item $-1\le \Delta<1$ (equivalently, $|a-b|<c\le a+b$): the \emph{disordered phase}. The case $\Delta=0$, termed the \emph{free fermion point}, enjoys additional integrability properties; see, e.g.,~\cite{Ken00,Dub11}. The uniform model $a=b=c$ is called \emph{square ice} and has $\Delta = 1/2$.
  \item $\Delta < -1$ (equivalently, $c>a+b$): the \emph{antiferroelectric phase}; see~\cite{DumGagHar16b, DumGagHar16} for recent rigorous confirmation of some of the predictions regarding this case and exposition of the Bethe ansatz.
\end{itemize}
The degenerate case~$c=0$ is known as the corner percolation model~\cite{Pet08}.

The goal of this work is to discuss the behavior of natural probabilistic observables in the different regimes. Our focus here is on the antiferroelectric phase ($a+b<c$) and on the boundary ($a+b=c$) between the antiferroelectric and disordered phases (see~\cite{BorCorGor16, CorGhoSheTsa20, aggarwal2018current} for recent progress on the ferroelectric phase). Our main results are:
\begin{itemize}
    \item Height function of the six-vertex model: The variance of the height function with flat boundary conditions is bounded when $a+b<c$ and grows as a logarithm of the distance to the boundary of the domain when~$a+b=c$.

        The translation-invariant, under parity-preserving translations, extremal Gibbs states are characterized. When $a+b=c$ no such states exist. When $a+b<c$, the states are parameterized by the integers, with the $n$-th state having unique infinite connected components (in diagonal connectivity) for heights~$n$ and~$n+1$, while all other heights together form connected components whose diameters have exponentially decaying tails.
    \item Spin representation of the six-vertex model: A spin representation introduced by Rys~\cite{Rys63}, of a mixed Ashkin--Teller type, is analyzed. Constant boundary conditions for each of the two Ising spins induce order in both spins when $a+b<c$ while leading to a disordered state when $a+b=c$.
    \item Gibbs states of the six-vertex model: The Gibbs states which exhibit infinitely many disjoint oriented circuits of alternating vertical and horizontal edges surrounding the origin (flat states) are classified. There are exactly two such extremal states when~$a+b<c$, and in each the orientation of each edge has a non-uniform distribution. There is a unique such state when~$a+b=c$.
    \item Self-dual Ashkin--Teller model on $\bbZ^2$: We introduce a coupling of the six-vertex model in the regime $a=b$ (the F-model), the symmetric Ashkin--Teller model on the self-dual curve $\sinh 2J = e^{-2U}$ and an associated graphical representation. It is then proved that when $J<U$ each of the two Ising configurations exhibits exponential decay of correlations while their product is ferromagnetically ordered. This is in contrast to the behavior for $J=U$ (the critical 4-state Potts model) and is in agreement with predictions in the physics literature.
\end{itemize}

The basic tool underlying the results is the Baxter--Kelland--Wu coupling~\cite{BaxKelWu76}. In the regime $a+b\le c$ it gives a probabilistic coupling of the six-vertex model with the critical FK model with parameter $q = 4\Delta^2$ (the coupling extends as a complex measure to other values of the parameters). This allows the transfer of known results on the FK model to the six-vertex setting, with the most relevant fact being that the phase transition in the FK model is of first order when $q>4$ and of second order when $q=4$. Our results on the fluctuations of the height function follow from this correspondence in a relatively direct manner. However, the characterization of Gibbs measures for the height function when $a+b<c$ requires additional tools. The challenge is to show that the model with flat boundary conditions (with values $0$ and $1$) converges in the thermodynamic limit irrespective of the sequence of domains used to exhaust $\Z^2$. This is achieved via a new technique involving $\bbT$-circuits, triangular lattice contours suitably embedded in $\bbZ^2$, combined with a careful analysis of the percolative properties of the level sets of the height function (see overview in Section~\ref{sec:overview for a=b=1}). Our understanding of the Gibbs states of the spin representation and of the six-vertex model itself are derived as a consequence, with the additional tool of an FKG inequality for the \emph{marginal} of the spin configuration on one of the sublattices which is established in the regime $a,b\le c$. Lastly, as mentioned above, our results for the Ashkin--Teller model are based on a coupling of it with the six-vertex model with $a=b$ and an associated graphical representation that we introduce; a coupling which extends previously discussed duality statements between the models. The graphical representation is proved to satisfy the FKG inequality in the regime $2a,2b\le c$ which is instrumental in deducing the exponential decay of correlations in each of the Ising configurations when $J<U$.

We add that some of the ideas that we develop in the current article have already
found further applications. In particular, the technique of $\bbT$-circuits that we introduce
(Definition 6.3) allows to give a short proof for the delocalisation of the height function in the
uniform case $a = b = c = 1$ (square ice), as we sketch in Section \ref{sec:open questions}. Unlike the original
proof \cite{ChaPelSheTas18}, our argument does not rely on Sheffield’s seminal work \cite{She05}. Also, our extension of the
Baxter--Kelland--Wu coupling to the six-vertex model with the same weights in the bulk
and on the boundary (Theorem \ref{thm:coupling}, part 1) was used in the recent short proof by Ray and Spinka
of the discontinuity of the phase transition in the Potts and random-cluster models
when q > 4, originally proven in  \cite{DumGagHar16} via the Bethe Ansatz.

\begin{figure}
	\begin{center}
		\includegraphics[width=\textwidth]{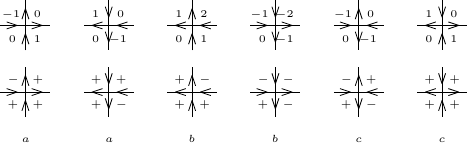}
	\end{center}
	\caption{The height and spin representation of the six-vertex model together with the weights assigned to each vertex. \emph{Top:} The height representation in the four faces around a vertex, normalized to equal $0$ at the lower left face. Going clockwise around the vertex, the height increases (decreases) by $1$ if the arrow points outward (inward).
\emph{Bottom:} The spin representation is derived from the height function by setting the spin state at each face to $+1$ ($-1$) if the height modulo four equals $0,1$ ($2,3$). The spin states may also be derived directly from the edge orientations, up to a global spin flip: The correspondence is indicated in the figure when the lower left face is even and has state $+1$. The case that the lower left face is odd and has state $+1$ is obtained from the case in the figure by switching the two configurations corresponding to each of the weights $a,b,c$. The case that the lower left face has state $-1$ is obtained from these cases by a global spin flip.
}
	\label{fig:6v-hom-config}
\end{figure}

\section{Results}\label{sec:results}
In this section we describe our results in detail.

We use the following definitions: The square lattice~$\bbZ^2$ is embedded in $\R^2$ so that its edges are parallel to the coordinate axes and its faces are centered at points with integer coordinates. Let $F(\bbZ^2)$ denote the set of faces of~$\bbZ^2$. A face centered at~$(i,j)$ is called \emph{even} if~$i+j$ is even, and it is called \emph{odd} if~$i+j$ is odd. For $u,v\in F(\bbZ^2)$ we write $|u-v|$ for the $\ell^1$ distance between the centers of $u$ and $v$. A finite subgraph~$\calD\subset\bbZ^2$ is called a \emph{domain} if there exists a simple cyclic path~$P$ in $\bbZ^2$ such that~$\calD$ coincides with the part of~$\bbZ^2$ surrounded by~$P$, including~$P$ itself. The path~$P$ is then termed the boundary of~$\calD$ and is denoted by~$\partial\calD$. For~$N\in \bbN$, let~$\Lambda_N$ be the domain whose vertices border the faces centered at all pairs of integers~$(i,j)$ that satisfy~$|i\pm j| \leq N-1$. As is common, we write \emph{cluster} for connected component. By \emph{parity-preserving translations} we mean translations by vectors in $\{(i,j)\in\Z^2\colon i+j\text{ is even}\}$. By \emph{all translations} we mean translations by vectors in $\Z^2$.

\subsection{Height function representation}
\label{sec:results-heights}

A function~$h\colon F(\bbZ^2)\to \bbZ$ is called a \emph{height function} if it satisfies that (see Figure~\ref{fig:height-function-0-1}):
\begin{itemize}
	\item for any two adjacent faces~$u,v$, $|h(u) - h(v)|=1$;
	\item for any face~$u$, the parity of~$h(u)$ is the same as the parity of~$u$.
\end{itemize}
Gradients of height functions are in a natural bijection with six-vertex configurations as described in Figure~\ref{fig:6v-hom-config} (see also Figure~\ref{fig:height-function-0-1}). Thus, a height function determines a unique six-vertex configuration and a six-vertex configuration determines a height function up to the global addition of an even integer.

\begin{figure}
	\begin{center}
		\includegraphics[scale=1.4, page=1]{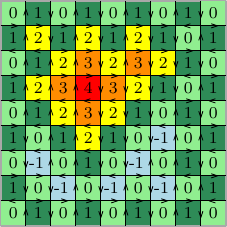}\quad\quad\quad
		\includegraphics[scale=1.4, page=2]{heights.pdf}
	\end{center}
	\caption{\emph{Left:} Height function with~$0,1$ boundary conditions. \emph{Right:} Ashkin--Teller-type spin representation of this height function ($++$ boundary conditions).}
	\label{fig:height-function-0-1}
\end{figure}

Given a height function $t$, the finite-volume height-function measure with parameters~$a,b,c>0$ on a domain~$\calD$ with boundary conditions~$t$ is supported on height functions that coincide with~$t$ at all faces outside of~$\calD$ and is defined by
\[
	{\sf HF}_{\calD,a,b,c}^{t}(h) = \tfrac{1}{Z_{{\sf hf},\calD,a,b,c}^t}\cdot a^{n_1(h)+n_2(h)}b^{n_3(h)+n_4(h)}c^{n_5(h)+n_6(h)},
\]
where~$Z_{{\sf hf},\calD,a,b,c}^t$ is a normalizing constant and~$n_i(h)$ is the number of vertices of~$\calD$ that are of type~$i$ according to the correspondence described in Figure~\ref{fig:6v-hom-config} (up to an additive constant).

For integer $n$ and domain~$\calD$, let~${\sf HF}_{\calD,a,b,c}^{n,n+1}$ be the height-function measure on~$\calD$ with boundary conditions given by a function that takes values in $\{n,n+1\}$ (each face according to its parity); see Figure~\ref{fig:height-function-0-1}. Note that if $f$ is sampled from ${\sf HF}_{\calD,a,b,c}^{0,1}$ then $f+2n$ is distributed as ${\sf HF}_{\calD,a,b,c}^{2n,2n+1}$ and $-f+2n$ is distributed as ${\sf HF}_{\calD,a,b,c}^{2n-1,2n}$.

In the next theorem, we show that the variance of the height at a fixed face is uniformly bounded when~$a+b<c$ and logarithmic in the distance to the boundary when~$a+b=c$.

\begin{theorem}[Fluctuations]
\label{thm:var}
Let~$\calD$ be a domain, let~$a,b,c > 0$ and let~$h^{0,1}_{\calD,a,b,c}$ be sampled from~${\sf HF}_{\calD,a,b,c}^{0,1}$. Then there exist~$c_1, C_1, C_2>0$, depending only on~$a,b,c$, such that for every face~$u$ of~$\calD$,
\begin{align}
\label{eq:var-c-greater-than-2}
\mathrm{Var}(h^{0,1}_{\calD,a,b,c}(u)) &\leq C_2,& &\text{ if } a+b<c,
\\
\label{eq:var-c-2}
c_1\log \mathrm{dist}(u,\bbZ^2\setminus\calD) \leq \mathrm{Var}(h^{0,1}_{\calD,a,b,c}(u))&\leq C_1\log  \mathrm{dist}(u,\bbZ^2\setminus\calD) ,&  &\text{ if } a+b=c.
\end{align}
\end{theorem}

We note that an analogue of~\eqref{eq:var-c-greater-than-2} is proven in~\cite{DumGagHar16} for periodic boundary conditions (i.e., when the height functions are defined on a torus). Analogues of~\eqref{eq:var-c-2} are known when: $c/a= c/b=\sqrt{2}$~\cite{Dub11,Ken00} (the free fermion point, by using its relation with the dimer model), $c/a= c/b \in (\sqrt{2}-\eps, \sqrt{2}+\eps)$~\cite{GiuMasTon17} (perturbation around the free fermion point, dimers with a small interaction), $a=b=c$~\cite{She05, ChaPelSheTas18, DumHarLasRauRay18} (uniform case, square ice). \footnote{See also very recent works where the delocalization \eqref{eq:var-c-2} was proven in the F-model ($a=b=1$): for $c\in [\sqrt{3},2]$ \cite{Lis20} and, more generally, for all $c\in [1,2]$ \cite{DumKarManOul20b}.}
Also, it was shown~\cite{Pel17} that high-dimensional versions of height functions have bounded variance.

A measure~${\sf HF}$ on height functions is called a \emph{Gibbs state for height functions} with parameters~$a,b,c>0$ if the following holds: Let $h$ be sampled from~${\sf HF}$. For any domain~$\calD$, conditioned on the values of $h$ on the faces outside of~$\calD$, the distribution of $h$ equals~${\sf HF}_{\calD,a,b,c}^{t}$, where $t$ is an arbitrary height function which agrees with $h$ outside of~$\calD$.
A Gibbs state is called \emph{extremal} if it has a trivial tail $\sigma$-algebra.

The next theorem characterizes the extremal Gibbs states for the height function which are invariant under parity-preserving translations.

\begin{theorem}[Gibbs states: height functions]\label{thm:heights-gibbs}
	1) Let~$a,b,c>0$ satisfy~$a+b<c$. For each integer $n$ and sequence of domains~$\{\calD_k\}$ increasing to $\bbZ^2$ the sequence of finite-volume measures ${\sf HF}_{\calD_k,a,b,c}^{n,n+1}$ converges to a Gibbs state~${\sf HF}_{a,b,c}^{n,n+1}$, which does not depend on~$\{\calD_k\}$. The limiting Gibbs states are extremal and invariant under parity-preserving translations, and each Gibbs state with these two properties equals ${\sf HF}_{a,b,c}^{n,n+1}$ for some integer $n$.
Moreover, the following properties are satisfied:
	\begin{itemize}
		\item Under~${\sf HF}_{a,b,c}^{n,n+1}$, clusters (in augmented connectivity) of heights different from~$n$ and~$n+1$ have diameters with exponential tail decay. Precisely, there exist $M,\alpha>0$ such that for all $N\in \bbN$,
		\begin{equation}\label{eq:exponential tail decay}
			 {\sf HF}_{a,b,c}^{n,n+1} (\exists \gamma\colon (0,0)\to \partial\Lambda_N \text{ s.t. } \forall k, \, h(\gamma_k)\notin\{n,n+1\} \text{ and } |\gamma_k - \gamma_{k+1}| \leq 2) \leq M e^{-\alpha N},
		\end{equation}
where we write $\gamma\colon (0,0)\to \partial\Lambda_N$ to mean a path in $F(\Z^2)$ starting at $(0,0)$ and ending at a face bordered by an edge of $\partial\Lambda_N$.
		\item ${\sf HF}_{a,b,c}^{n,n+1}$ is invariant under the operation $h(i,j) \mapsto 2n+1 - h(-i+1,j)$, whence
		\[
			{\sf HF}_{a,b,c}^{n,n+1}(h(0,0)+ h(1,0)) = 2n + 1.
		\]
		\item Each ${\sf HF}_{a,b,c}^{n,n+1}$ is positively associated and the stochastic ordering relation~${\sf HF}_{a,b,c}^{m,m+1} \prec {\sf HF}_{a,b,c}^{n,n+1}$ holds for~$m<n$.
	\end{itemize}
	2) Let~$a,b,c>0$ satisfy~$a+b=c$. Then there are no extremal Gibbs states for the height function which are invariant under parity-preserving translations.
\end{theorem}

It is a straightforward consequence that, in the case $a+b<c$,~${\sf HF}_{a,b,c}^{n,n+1}$-a.s., there exist infinitely many disjoint level lines separating the heights $n$ and~$n+1$ surrounding the origin.

\begin{remark}
	It follows from~\cite[Chapter 8]{She05} that all ergodic height-function Gibbs states are in fact extremal (see also~\cite[Section 5]{ChaPelSheTas18} for a short proof in the uniform case~$a=b=c=1$). With this fact, Theorem~\ref{thm:heights-gibbs} gives a complete description of translation-invariant height-function Gibbs states when~$a+b\leq c$.
\end{remark}

\subsection{Spin representation (mixed Ashkin--Teller model)}
\label{sec:spin-rep}

A function~$\sigma\colon F(\bbZ^2)\to \{-1,1\}$ is called a \emph{spin configuration satisfying the ice rule} if around each vertex there is a pair of diagonally-adjacent faces on which $\sigma$ agrees. The set of all such functions is denoted $\calE_\mathsf{spin}(\bbZ^2)$. Spin configurations satisfying the ice rule, regarded up to a global spin flip, are in a natural bijection with six-vertex configurations as described in Figure~\ref{fig:6v-hom-config} and its caption (see also Figure~\ref{fig:height-function-0-1}). In other words, each $\sigma\in \calE_\mathsf{spin}(\bbZ^2)$ determines a unique six-vertex configuration while a six-vertex configuration determines a pair of spin configurations satisfying the ice rule, related by a global spin flip.

There is a direct correspondence between height functions $h$ and spin configurations $\sigma$ satisfying the ice rule, which is consistent with the bijections of these objects with six-vertex configurations and is defined by setting $\sigma(u) = 1$ if $h(u)\equiv 0,1\pmod 4$ and $\sigma(u)=-1$ if $h(u)\equiv 2,3\pmod 4$ (see Figure~\ref{fig:6v-hom-config} and~Figure~\ref{fig:height-function-0-1}). A height function thus determines a unique $\sigma\in \calE_\mathsf{spin}(\bbZ^2)$ while each $\sigma\in \calE_\mathsf{spin}(\bbZ^2)$ determines a height function up to the global addition of an integer divisible by 4.

The finite-volume spin measure with parameters~$a,b,c>0$ on a domain~$\calD$ with boundary condition~$\tau\in\calE_\mathsf{spin}(\bbZ^2)$ is supported on $\sigma\in\calE_\mathsf{spin}(\bbZ^2)$ that coincide with~$\tau$ at all faces outside of~$\calD$ and is defined by
\begin{equation}\label{eq:spin representation measure def}
	{\sf Spin}_{\calD,a,b,c}^{\tau}(\sigma) = \tfrac{1}{Z_{{\sf spin},\calD,a,b,c}^\tau}\cdot a^{n_1(\sigma)+n_2(\sigma)}b^{n_3(\sigma)+n_4(\sigma)}c^{n_5(\sigma)+n_6(\sigma)},
\end{equation}
where~$Z_{{\sf spin},\calD,a,b,c}^\tau$ is a normalizing constant and~$n_i(\sigma)$ is the number of vertices of~$\calD$ that are of type~$i$ according to the correspondence described in Figure~\ref{fig:6v-hom-config} and its caption. In particular, if $t$ is a height function which maps to $\tau$ under the modulo $4$ mapping described above then the measure ${\sf Spin}_{\calD,a,b,c}^{\tau}$ is the push-forward of the measure ${\sf HF}_{\calD,a,b,c}^{t}$ by this mapping. For $\alpha,\beta\in\{-,+\}$, we use the notation ${\sf Spin}_{\calD,a,b,c}^{\alpha\beta}$ to denote the measure ${\sf Spin}_{\calD,a,b,c}^{\tau}$ in which $\tau$ is the configuration having sign $\alpha$ on all even faces and having sign $\beta$ on all odd faces.

A measure~${\sf Spin}$ on~$\calE_\mathsf{spin}(\bbZ^2)$ is called a \emph{Gibbs state for the spin representation} with parameters~$a,b,c>0$ if the following holds: Let $\sigma$ be sampled from~${\sf Spin}$. For any domain~$\calD$, conditioned on the values of $\sigma$ on the faces outside of~$\calD$, the distribution of $\sigma$ equals~${\sf Spin}_{\calD,a,b,c}^{\tau}$, where $\tau\in \calE_\mathsf{spin}(\bbZ^2)$ is an arbitrary configuration which agrees with $\sigma$ outside of~$\calD$.
Let~$\calG_{a,b,c}^{\sf spin}$ denote the set of all extremal (i.e., tail trivial) Gibbs states that are invariant under parity-preserving translations.

In the next theorem, we study the Gibbs states of the spin representation. In the regime of the antiferroelectric phase~$a+b<c$, we construct four distinct measures (the push-forwards of the height measures~${\sf HF}_{a,b,c}^{n,n+1}$ for different values of~$n$ by the modulo $4$ mapping) and show that, under these measures, the correlations of spins at faces of the same parity are uniformly positive. In the regime~$a+b=c$, we construct a measure~${\sf Spin}_{a,b,c}\in \calG_{a,b,c}^{\sf spin}$ and show that it may be characterized by the absence of certain infinite clusters. In discussing clusters of even (or odd) faces we consider two faces of the same parity adjacent if they share a vertex.

\begin{theorem}[Gibbs states: spin representation]\label{thm:spins-gibbs}\mbox{}

	1) Let~$a,b,c>0$ satisfy~$a+b<c$. For each $\alpha,\beta\in\{-,+\}$ and sequence of domains~$\{\calD_k\}$ increasing to $\bbZ^2$ the sequence of finite-volume measures~${\sf Spin}_{\calD_k,a,b,c}^{\alpha\beta}$ converges to a Gibbs state~${\sf Spin}_{a,b,c}^{\alpha\beta}\in\calG_{a,b,c}^{\sf spin}$, which does not depend on~$\{\calD_k\}$. The four limiting measures are distinct. Moreover, the measure~${\sf Spin}_{\calD_k,a,b,c}^{\alpha\beta}$ satisfies the following properties:
  \begin{itemize}
    \item Samples from ${\sf Spin}_{a,b,c}^{\alpha\beta}$ exhibit a unique infinite cluster of even faces with sign $\alpha$ and a unique infinite cluster of odd faces with sign $\beta$, almost surely.
    \item For each even (odd) face $u$, ${\sf Spin}_{a,b,c}^{\alpha\beta}(\sigma(u))$ does not depend on $u$ (by invariance), is non-zero and of the same sign as $\alpha$ (as~$\beta$). In addition, for any~$u$ and~$v$ of the same parity,
        \[
        {\sf Spin}_{a,b,c}^{\alpha\beta}(\sigma^\circ(u)\sigma^\circ(v))\ge {\sf Spin}_{a,b,c}^{\alpha\beta}(\sigma^\circ(u))^2 > 0.
        \]
 \end{itemize}
	
	2) Let~$a,b,c>0$ satisfy~$a+b=c$. There exists a Gibbs state~${\sf Spin}_{a,b,c}$ for the spin representation with the following properties:
\begin{itemize}
   \item For any sequence of domains~$\{\calD_k\}$ increasing to $\bbZ^2$ and any~$\tau\in \calE_\mathsf{spin}(\bbZ^2)$ which is either constant on all even faces or constant on all odd faces, the sequence of finite-volume measures~${\sf Spin}_{\calD_k,a,b,c}^{\tau}$ converges to~${\sf Spin}_{a,b,c}$.
   \item The measure~${\sf Spin}_{a,b,c}$ is invariant under all translations and is extremal (in particular, ${\sf Spin}_{a,b,c}\in \calG_{a,b,c}^{\sf spin}$). In addition, it is invariant under a global sign flip applied on either all even faces or all odd faces. Consequently, ${\sf Spin}_{a,b,c}(\sigma(u))=0$ and ${\sf Spin}_{a,b,c}(\sigma(u)\sigma(v))=0$ for two faces $u$ and $v$ of different parity.
   \item There exist $C,\alpha>0$ so that for two faces $u,v$ of the same parity,
       \begin{equation*}
         0\le {\sf Spin}_{a,b,c}(\sigma(u)\sigma(v))\le \frac{C}{|u-v|^\alpha}.
       \end{equation*}
   \item Samples from~${\sf Spin}_{a,b,c}$ exhibit no infinite cluster of faces having the same parity and the same spin, almost surely.

       In contrast, each other element of~$\calG_{a,b,c}^{\sf spin}$ exhibits, almost surely, at least one infinite cluster of each of the four types~--- even $+1$, even $-1$, odd $+1$, odd $-1$.
  \end{itemize}
\end{theorem}

The next theorem verifies the Fortuin--Kasteleyn--Ginibre (FKG) inequality (which  implies positive association~\cite{ForKasGin71}) for marginals of the spin representation in the regime $a,b\le c$ (all of the antiferroelectric phase and part of the disordered phase).
Denote by~$F^\bullet(\bbZ^2)$ and~$F^\circ(\bbZ^2)$ the set of even (resp. odd) faces of~$\bbZ^2$. Given~$\sigma\in\{-1,1\}^{F(\bbZ^2)}$ denote by~$\sigma^\bullet\in\{-1,1\}^{F^\bullet(\bbZ^2)}$ and by~$\sigma^\circ\in\{-1,1\}^{F^\circ(\bbZ^2)}$ the restrictions of~$\sigma$ to~$F^\bullet(\bbZ^2)$ and to~$F^\circ(\bbZ^2)$ respectfully. We endow $\{-1,1\}^{F^\bullet(\bbZ^2)}$ with the pointwise partial order: $\sigma^\bullet \geq \tau^\bullet$ if~$\sigma^\bullet(u) \geq \tau^\bullet(u)$ for all~$u\in F^\bullet(\bbZ^2)$.

\begin{theorem}[Positive association: spin representation]\label{thm:fkg}
	Let~$\calD$ be a domain and consider~$\tau\in \calE_{\sf spin}(\bbZ^2)$ that is equal to~$1$ at all odd faces outside of~$\calD$. Suppose that~$a,b,c > 0$ satisfy~$a,b\leq c$. Then the marginal of~${\sf Spin}_{\calD,a,b,c}^{\tau}$ on~$\sigma^\bullet$ satisfies the FKG lattice condition. In particular, for any increasing functions $f,g:\{-1,1\}^{F^\bullet(\bbZ^2)}\to\R$, one has
	\[
		{\sf Spin}_{\calD,a,b,c}^{\tau}(f(\sigma^\bullet)g(\sigma^\bullet)) \geq {\sf Spin}_{\calD,a,b,c}^{\tau}(f(\sigma^\bullet)) \cdot{\sf Spin}_{\calD,a,b,c}^{\tau}(g(\sigma^\bullet)).
	\]
\end{theorem}

\begin{remark}
	The restriction to constant boundary conditions for the spins at odd faces is essential~--- the positive association may fail if the boundary conditions assign mixed signs to odd faces; see Fig. \ref{fig:fkg-counter-ex} and also~\cite[Remark 2.11]{GlaMan18} for a related setting.
\end{remark}

\begin{figure}
	\begin{center}
		\includegraphics[scale=1.4]{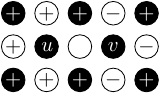}
		\caption{A counter-example for the positive association on even (black circles) sites under boundary conditions that have different spins at odd faces. Due to the ice-rule the probability to have $\sigma^{\bullet}(u) = \sigma^{\bullet}(v) = -1$ is zero. Since the other three spin configurations at $u,v$ all have strictly positive probabilities, the positive association does not hold.}
		\label{fig:fkg-counter-ex}
	\end{center}
\end{figure}

The FKG inequality established in Theorem~\ref{thm:fkg} should be put in analogy with a similar property established for other models: \cite[Proposition A.1]{Cha98} (XY model), \cite[Lemma 1]{ChaMcWin98} (Ashkin--Teller model), \cite[Proposition 8]{DumGlaPel17} and \cite[Theorem 2.6]{GlaMan18} (loop~$O(n)$ model). In particular, the result in~\cite{ChaMcWin98}, via the mapping between the spin representation and the standard Ashkin--Teller model described in Section~\ref{sec:proof-ashkin-teller} below (see also~\cite{HuaDenJacSal13}) allows to derive Theorem~\ref{thm:fkg} when $a/c = b/c \leq \frac{1}{\sqrt{2}}$ (but apparently not in the full range $a,b\le c$). Also, the proof of Theorem~\ref{thm:fkg} is closely related to that of~\cite[Theorem 2.6]{GlaMan18}.

The spin representation of the six-vertex model was considered already by Rys~\cite{Rys63}. In the terminology of~\cite{HuaDenJacSal13}, it can be called an \emph{infinite-coupling limit mixed Ashkin--Teller model}. The term `mixed' refers to the fact that the spin configurations~$\sigma^\bullet$ and~$\sigma^\circ$ are defined on two lattices that are dual to each other, while in the standard Ashkin--Teller model both spin configurations are defined on the same lattice (see Section~\ref{sec:ashkin-teller}). The term `infinite-coupling limit' refers, in our case, to the the ice rule constraint.

A follow-up work of Lis~\cite{Lis19} studies the case of two interacting Potts models (see also Owczarek and Baxter~\cite{OwcBax87} where a more general Temperley--Lieb interactions model is introduced) and proves, in particular, that they too satisfy an FKG inequality.

\subsection{Orientations of edges in the six-vertex model}\label{sec:six vertex results}

In this section, we state an immediate consequence of Theorem~\ref{thm:spins-gibbs} for the six-vertex model in its classical representation in terms of edge orientations. As stated in the introduction, a \emph{six-vertex configuration} is an orientation of the edges of $\bbZ^2$ that satisfies the ice-rule at every vertex (two incoming and two outgoing edges); see Figure~\ref{fig:6v-arrow-config}. Given a six-vertex configuration $\vec{\tau}$, the finite-volume six-vertex measure ${\sf SixV}_{E,a,b,c}^{\vec{\tau}}$, on a finite subset of edges $E\subset E(\Z^2)$ with boundary conditions~$\vec{\tau}$, is supported on six-vertex configurations that coincide with~$\vec{\tau}$ at all edges outside of~$E$ and is defined by:
\begin{equation*}
  {\sf SixV}_{E,a,b,c}^{\vec{\tau}}(\vec{\omega}) = \frac{1}{Z_{{\sf SixV},E,a,b,c}^{\vec{\tau}}}a^{n_1(\vec{\omega})+n_2(\vec{\omega})}b^{n_3(\vec{\omega})+n_4(\vec{\omega})}c^{n_5(\vec{\omega})+n_6(\vec{\omega})},
\end{equation*}
where~$Z_{{\sf SixV},E,a,b,c}^{\vec{\tau}}$ is a normalizing constant and $n_i(\vec{\omega})$ is the number of endpoints of edges in~$E$ at which the six-vertex configuration~$\vec{\omega}$ is of type~$i$ according to Figure~\ref{fig:6v-arrow-config}. Gibbs states for the six-vertex model are then defined in the standard way (as in the previous sections).

Our analysis classifies extremal (i.e., tail trivial) Gibbs states which are flat in an appropriate sense (these are expected to be the only Gibbs states for which the associated height function has zero slope but that is not proved here).

\begin{corollary}[Gibbs states: six-vertex model]\label{cor:6v-gibbs} A Gibbs state is termed `flat' if in a configuration sampled from that state, almost surely, there are infinitely many disjoint oriented circuits which surround the origin and consist of alternating vertical and horizontal edges.
\begin{enumerate}
  \item When~$a+b<c$, there are exactly two extremal flat Gibbs states ${\sf SixV}_{a,b,c}^{\circlearrowleft}, {\sf SixV}_{a,b,c}^{\circlearrowright}$. These states are invariant under parity-preserving translations and under ninety-degree rotations around the origin and differ from each other by a global edge-orientation flip. Moreover, if we denote by $A(e)$ the event that the edge $e\in E(\Z^2)$ is oriented so that the even face that it borders lies on its left, then ${\sf SixV}_{a,b,c}^{\circlearrowleft}(A(e))$ does not depend on $e$ (by invariance) and it holds that ${\sf SixV}_{a,b,c}^{\circlearrowleft}(A(e))>\frac{1}{2}$ and
      \begin{equation}\label{eq:arrows-correlation}
      	{\sf SixV}_{a,b,c}^{\circlearrowleft}(A(e) A(f))\ge {\sf SixV}_{a,b,c}^{\circlearrowleft}(A(e))^2
      \end{equation}
      for all edges $e,f\in E(\Z^2)$. (corresponding statements hold for ${\sf SixV}_{a,b,c}^{\circlearrowright}$ as it differs from ${\sf SixV}_{a,b,c}^{\circlearrowleft}$ by a global edge-orientation flip).

  \item When~$a+b=c$, the six-vertex model has a unique flat Gibbs state. This state is extremal and invariant under all translations.
\end{enumerate}
\end{corollary}

\subsection{Ashkin--Teller model}
\label{sec:ashkin-teller}

The Ashkin--Teller model was originally introduced~\cite{AshTel43} as a generalization of the Ising model to a four-state system. The definition in terms of two coupled Ising models that we provide below is due to Fan~\cite{Fan72}.

We consider the (symmetric) Ashkin--Teller model on the square grid. We will later describe a coupling of the Ashkin--Teller model with the spin representation of the six-vertex model (Proposition~\ref{prop:coupling-AT-6V}) and in anticipation of this it is convenient to define the Ashkin--Teller model on the set of even faces $F^\bullet(\bbZ^2)$ of $\bbZ^2$ with diagonal connectivity (this graph is isomorphic to $\bbZ^2$) rather than on $\bbZ^2$ itself. Accordingly, we let~$(\bbZ^2)^\bullet$ be the graph with vertex set $F^\bullet(\bbZ^2)$ and with edges between diagonally-adjacent faces. The Ashkin--Teller measure with parameters~$J,U\in\bbR$ on a subgraph~$\Omega$ of~$(\bbZ^2)^\bullet$ is supported on pairs of spin configurations~$(\tau,\tau')\in \{-1,1\}^{V(\Omega)}\times  \{-1,1\}^{V(\Omega)}$ and is defined by
\begin{equation}\label{eq:Ashkin-Teller measure def}
	{\sf AT}_{\Omega,J,U} (\tau,\tau') = \tfrac{1}{Z_{\Omega,J,U}}\cdot \exp \left[ \sum_{\{u,v\}\in E(\Omega)} J(\tau(u)\tau(v) + \tau'(u)\tau'(v)) + U\tau(u)\tau(v)\tau'(u)\tau'(v) \right],
\end{equation}
where~$Z_{\Omega,J,U}$ is a normalizing constant and the sum is taken over all edges in~$\Omega$.

Proposition~\ref{prop:coupling-AT-6V} shows that there is a coupling of the Ashkin--Teller measure with parameters $J,U$ and the spin representation measure~\eqref{eq:spin representation measure def} with parameters $a,b,c$, on suitable domains and with suitable boundary conditions, when the parameters satisfy the relations
\begin{equation}\label{eq:at-to-6v-parameters}
	\sinh 2J = e^{-2U}, \quad a=b=1, \quad c = \coth 2J
\end{equation}
so that the configurations $(\tau,\tau')$ and $\sigma$ satisfy at every even face the equality
\[
	\tau\tau' = \sigma.
\]
The first equality in~\eqref{eq:at-to-6v-parameters} describes the self-dual curve of parameters for the Ashkin--Teller model and was first found by Mittag and Stephen~\cite{MitSte71} (see Figure~\ref{fig:AT-phase-diagram}). The relation between the Ashkin--Teller and the eight-vertex model was noticed already by Fan~\cite{Fan72} and then made explicit by Wegner~\cite{Weg72} (see also~\cite[Section III]{Kot85}). In the particular case given by~\eqref{eq:at-to-6v-parameters}, this turns into a correspondence between the Ashkin--Teller and six-vertex models (see, e.g.,~\cite[Section 12.9]{Bax82}). This correspondence is upgraded here to a coupling of the models together with an FK--Ising representation that is introduced (Section~\ref{sec:fk-ising} and Section~\ref{sec:proof-ashkin-teller}), which facilitates the transfer of results between the models. Thus we obtain a coupling of the six-vertex model with $a=b=1$ and $c>1$ and the self-dual Ashkin--Teller model (see Section~\ref{sec:open questions} for the limiting $c=1$ case).

\begin{figure}
	\begin{center}
		\includegraphics[scale=1.4]{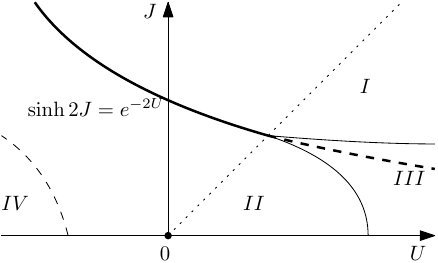}
		\caption{A sketch of the predicted phase diagram of the Ashkin--Teller model. The self-dual curve~$\sinh 2J= e^{-2U}$ is in bold: it is critical when~$J\ge U$ (solid) but expected to be non-critical when~$J<U$ (dashed) while the critical curve is expected to split at the point $J=U$ into two critical curves dual to each other (no prediction for their exact location). Phase $I$: $\tau,\tau',\tau\tau'$ are ferromagnetically ordered. Phase~$II$: $\tau,\tau',\tau\tau'$ are disordered. Phase~$III$: $\tau\tau'$ is ferromagnetically ordered, while $\tau,\tau'$ are disordered. Phase~$IV$: $\tau\tau'$ is antiferromagnetically ordered, while $\tau,\tau'$ are disordered. The line~$J=U$ (dotted) corresponds to the $4$-state Potts model. The line~$U=0$ corresponds to two independent Ising models. Only the regime $J\ge0$ is drawn as the regime $J\le 0$ is equivalent to it, by flipping the signs of the Ising models on one partite class of the square grid.}
		\label{fig:AT-phase-diagram}
	\end{center}
\end{figure}

In the next theorem, we show that on the self-dual curve when~$J<U$ (see Figure~\ref{fig:AT-phase-diagram}), correlations in~$\tau$ and~$\tau'$ decay exponentially (disordered regime) while correlations in the product~$\tau\tau'$ are uniformly positive (ordered regime), in agreement with the predicted~\cite{Kno75,DitBanGreKad80} phase diagram of the Ashkin--Teller model (see also~\cite[Section~5]{HuaDenJacSal13} for a recent survey with explicit computations). For integer $k\ge 0$, let~$\Omega_k$ be the induced subgraph of~$(\bbZ^2)^\bullet$ on the faces of $\bbZ^2$ whose centers are at~$(i,j)\in [-k-1,k+1]^2$. Define~${\sf AT}_{\Omega_k,J,U}^{\mathrm{free},+}$ to be the Ashkin--Teller measure conditioned on the event that~$\tau=\tau'$ on the internal boundary of~$\Omega_k$.
\begin{theorem}\label{thm:ashkin-teller}
	 Let~$J,U>0$ be such that~$\sinh 2J= e^{-2U}$ and~$J<U$. Then, the sequence of measures~${\sf AT}_{\Omega_k,J,U}^{\mathrm{free},+}$ converges (weakly) to a measure ${\sf AT}_{J,U}^{\mathrm{free},+}$ that is translation invariant and extremal. Further, there exist~$C,c,\alpha>0$ such that for any two vertices~$u,v$ of~$(\bbZ^2)^\bullet$,
	 \begin{align}
	 	{\sf AT}_{J,U}^{\mathrm{free},+}(\tau(u)\tau'(u)\tau(v)\tau'(v)) &\geq c,\label{eq:AT-corr-bounded}\\
	 	{\sf AT}_{J,U}^{\mathrm{free},+}(\tau(u)\tau(v)) = {\sf AT}_{J,U}^{\mathrm{free},+}(\tau'(u)\tau'(v)) &\leq C e^{-\alpha |u-v|}\label{eq:AT-corr-exp-decay} .
	 \end{align}
\end{theorem}
We briefly survey some of the rigorous results on the phase transition of the Ashkin--Teller model near its self-dual curve. It is natural to search for possible phase transitions when changing the parameters along the lines in which~$J/U$ is constant (this corresponds to changing the temperature when the term in the exponent of~\eqref{eq:Ashkin-Teller measure def} is multiplied by an inverse temperature parameter). 
When doing so with~$J>U>0$, under plus boundary conditions, correlations of~$\tau$, $\tau'$ and~$\tau\tau'$ can be shown to undergo a sharp phase transition at the same curve of parameters~$\gamma_c$: they decay exponentially fast in the distance when
~$(J,U)$ is below~$\gamma_c$ and~stay uniformly positive when~$(J,U)$ is above~$\gamma_c$\footnote{One needs to use a monotonic random-cluster representation developed in~\cite{PfiVel97,ChaMac98} and a general approach~\cite{DumRauTas17} allowing to show sharpness for monotonic measures using the OSSS inequality~\cite{OdoSakSchSer05}. See a recent work~\cite{AouDobGla23} for the details.}.
Provided that the transition under free boundary conditions also occurs at~$\gamma_c$, this implies that~$\gamma_c$ coincides with the self-dual curve~$\sinh 2J = e^{-2U}$.
Rigorous results on the critical behavior are available only at~$J=U=\tfrac14 \log 3$ (critical 4-state Potts model) where all correlations are known to have power-law decay~\cite{DumSidTas17} and at~$J=\tfrac12 \log (1+\sqrt{2})$, $U=0$ (two independent critical Ising models) where correlations in~$\tau$ and~$\tau'$ decay as~$|u-v|^{-1/4}$~\cite{onsager1944crystal, mccoy2014two, Smi10,CheHonIzy15}. 
When $0<J<U$ and the parameters are varied on a line with~$J/U<\frac{1}{2}$ fixed, Pfister~\cite{Pfi82} (see also H\"aggstr\"om~\cite{Hag96}) proved that there exist three phases~--- a disordered phase and an ordered phase (for $\tau, \tau'$ and $\tau\tau'$) as well as an intermediate phase in which~$\tau$ and~$\tau'$ are disordered but~$\tau\tau'$ is ordered (see Figure~\ref{fig:AT-phase-diagram}). This behavior is predicted to persist, when $0<J<U$, for all values of the ratio~$J/U$ and Theorem~\ref{thm:ashkin-teller} supports this prediction as it shows that on the part of the self-dual curve~$\sinh 2J= e^{-2U}$ with~$0<J<U$ the model indeed has the properties of the intermediate phase. However, our results do not show that the intermediate phase extends beyond the self-dual curve.

The following corollary is a straightforward consequence of the positive association of the spin representation established in Theorem~\ref{thm:fkg} and the coupling between the spin representation and the Ashkin--Teller models described in Proposition~\ref{prop:coupling-AT-6V}.

\begin{corollary}\label{cor:fkg-AT}
	Let~$J>0$ and $U\in\R$ be such that~$\sinh 2J= e^{-2U}$. Then, for any~$k$, the marginal of the measure~${\sf AT}_{\Omega_k,J,U}^{\mathrm{free},+}$ on the product~$\tau\tau'$ of the spins satisfies the FKG lattice condition (in particular, it is positively associated).
\end{corollary}

\subsection{Monotonicity in the boundary coupling constant}
\label{sec:monotonicity-in-c-b}

The starting point for our analysis of the six-vertex model is the Baxter--Kelland--Wu~\cite{BaxKelWu76} coupling of it with the random-cluster model.
Originally, the coupling was stated for the six-vertex model on general planar graphs with no boundary condition.
In Section~\ref{sec:coupling} we describe an extended version of this coupling on domains on $\bbZ^{2}$ for the models under two different types of boundary conditions.
In particular, the wired boundary condition in the critical random-cluster model (discussed in the original work~\cite{BaxKelWu76}) corresponds to a six-vertex model, in which the parameters $a,b,c$ assigned to vertices on the boundary of the domain differ from the parameters in the bulk. 
We call attention to it as a useful extension of the model.
As we describe below, on a class of domains, the height-function measures are stochastically ordered with respect to these boundary parameters and these parameters enable a sort of continuous interpolation between different boundary conditions for the six-vertex model.

Let $\calD$ be a domain on $\bbZ^{2}$.
Denote by~$\partial_V\calD$ the set of vertices belonging to exactly one face of~$\calD$; see Figure~\ref{fig:even-domain-rcm}.
Given a height function $t$, the finite-volume height-function measure on~$\calD$ with boundary conditions~$t$ and parameters~$a,b,c,c_b>0$ is supported on height functions that coincide with~$t$ at all faces outside of~$\calD$ and is defined by
\begin{equation}\label{eq:def-hf-cb}
	{\sf HF}_{\calD,a,b,c}^{t,c_b}(h) = \tfrac{1}{Z_{{\sf hf},\calD,a,b,c}^{t,c_b}}\cdot a^{n_1'(h)+n_2'(h)}b^{n_3'(h)+n_4'(h)}c^{n_5'(h)+n_6'(h)}c_b^{n_b(h)},
\end{equation}
where~$Z_{{\sf hf},\calD,a,b,c}^{t,c_b}$ is a normalizing constant and~$n_i'(h)$ is the number of vertices of~$\calD\setminus \partial_V \calD$ that are of type~$i$ according to the correspondence described in Figure~\ref{fig:6v-hom-config} (up to an additive constant) and $n_b(h)$ is the number of vertices in~$\partial_V \calD$ which are of type 5 and 6 according to the figure (the boundary weights of vertices of types 1-4 are fixed to one). 

As in Section \ref{sec:results-heights}, we write ${\sf HF}_{\calD,a,b,c}^{n,n+1,c_b}$ when the boundary condition $t$ takes values in $\{n,n+1\}$.
Denote by~$\partial_{\mathrm{ext}}\calD$ (and call it the {\it external boundary} of~$\calD$) the set of faces in~$\bbZ^2\setminus \calD$ that are adjacent to faces in~$\calD$.
We say that a domain $\calD$ on $\bbZ^{2}$ is of \emph{fixed parity} if all faces in $\partial_{\mathrm{ext}}\calD$ have the same parity.
According to this parity, the domain is then called \emph{even} or \emph{odd}.

\begin{proposition}[Monotonicity in~$c_b$: heights]
\label{prop:monotone-c-b-heights}
	Let~$\calD$ be an even domain. Suppose $c \geq \max\{a,b\}$ and~$c_b, c_b'\in [0,\infty]$ satisfy~$c_b\leq c_b'$.
	Then,
	\begin{equation}
	\label{eq:compare-c-b-heights}
	{\sf HF}_{\calD\setminus\partial_V\calD,a,b,c}^{-1,0} \preceq {\sf HF}_{\calD,a,b,c}^{0,1; c_b} \preceq{\sf HF}_{\calD,a,b,c}^{0,1; c_b'} \preceq {\sf HF}_{\calD\setminus\partial_V\calD,a,b,c}^{0,1},
	\end{equation}
	where~$\calD\setminus\partial_V\calD$ denotes the graph obtained from~$\calD$ after removing all vertices~$\partial_V\calD$ (together with the edges incident to them).
\end{proposition}

\begin{remark}
	It follows from the proof that varying~$c_b$ from~$0$ to~$\infty$ allows to continuously interpolate between~$-1,0$ and~$0,1$ boundary conditions.
\end{remark}

This monotonicity with respect to $c_{b}$ follows from the well-known positive association for the height-function measure when~$c\geq \max\{a,b\}$ (see~\cite[Proposition 2.2]{BenHagMos00} and Proposition~\ref{prop:fkg-homo} below).
Similarly, the positive association stated in Theorem~\ref{thm:fkg} for the marginals of the spin representation on the even and the odd sublattices, implies that these marginals are stochastically ordered in~$c_b$. More precisely, let ${\sf Spin}_{\calD,c}^{+; c_b}$ be supported on the set of spin configurations on~$\bbZ^2$ that are equal to~$+1$ outside of~$\calD$ and defined by
\begin{equation}\label{eq:def-spins-c-b}
	{\sf Spin}_{\calD,c}^{+; c_b}(\sigma) = \tfrac{1}{Z_{{\sf spin},\calD,c}^{+; c_b}} \cdot a^{n_1'(\sigma)+n_2'(\sigma)}b^{n_3'(\sigma)+n_4'(\sigma)}c^{n_5'(\sigma)+n_6'(\sigma)}c_b^{n_b(\sigma)},
\end{equation}
where~$Z_{{\sf spin},\calD,c}^{+;c_b}$ is a normalizing constant; $n_{i}'(\sigma)$ and $n_{b}(\sigma)$ are the counterparts for the spins of the corresponding quantities in \eqref{eq:def-hf-cb} for the heights.

\begin{proposition}[Monotonicity in~$c_b$: spins]
\label{prop:monotone-c-b-spins}
	Let~$\calD$ be an odd domain. Take any~$c \geq \max\{a,b\}$ and~$c_b, c_b'\in [0,\infty]$, such that~$c_b\leq c_b'$.
	Then, for any increasing event~$A\subset \{1,-1\}^{F^\bullet(\calD)}$,
	\begin{equation}
	\label{eq:compare-c-b-spins}
	{\sf Spin}_{\calD,c}^{+; c_b}(A)\leq {\sf Spin}_{\calD,c}^{+; c_b'}(A).
	\end{equation}
\end{proposition}

\begin{figure}
	\begin{center}
		\begin{minipage}{0.6\textwidth}
			\includegraphics[width=\textwidth]{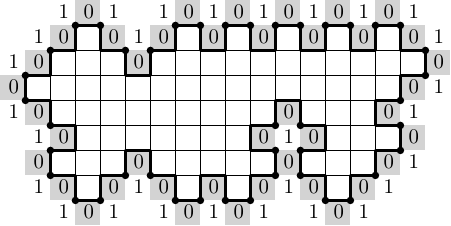}
		\end{minipage}
		\begin{minipage}{0.39\textwidth}
			\includegraphics[width=\textwidth]{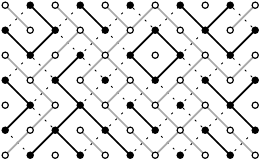}
		\end{minipage}
		\caption{\emph{Left:} An even domain~$\calD$ with edge-boundary~$\partial\calD$ in bold. The vertices belonging to~$\partial_V \calD$ are marked; they are assigned the coupling constant~$c_b$. The faces of the external boundary~$\partial_{\mathrm{ext}}\calD$ are shown in gray. The boundary conditions are~$0,1$. \emph{Right:} A sample of a random-cluster configuration on the square grid (rotated by~$45^\circ$). Open edges of the configuration are shown in black, closed edges are dashed, dual edges are in gray. Here~$k_b = 9$ (nine boundary clusters) and~$k_i=4$ (four non-boundary clusters).}
		\label{fig:even-domain-rcm}
	\end{center}
\end{figure}

\subsection{An FK model with a modified boundary-cluster weight}
\label{sec:rcm-results}

In Section~\ref{sec:coupling} we further introduce a second coupling of the six-vertex model with the random-cluster model, in which the six-vertex model has the same parameters on the boundary and in the bulk but the random-cluster model is modified by assigning a special weight~$q_b \neq q$ to boundary clusters. This modified random-cluster model appears natural, as varying $q_{b}$ between $1$ and $q$ allows a continuous interpolation between wired and free boundary conditions. Indeed, following our work, it was used in~\cite{RaySpi19} in a short proof of the discontinuity of the phase transition for~$q>4$. We describe the model and some of its basic properties in this section.

Let~$\Omega$ be a subgraph of a square lattice and let~$E(\Omega)$ denote the set of edges in~$\Omega$. Given a configuration~$\eta\in\{0,1\}^{E(\Omega)}$, we call an edge~$e\in E(\Omega)$ open if~$\eta(e) = 1$ and closed if~$\eta(e) = 0$. Thus, each configuration~$\eta$ can be viewed as a subset of~$E(\Omega)$ given by the set of open edges in~$\eta$; see Figure~\ref{fig:even-domain-rcm}.

Given~$q,q_b>0$ and~$p\in [0,1]$, the random-cluster measure~${\sf RC}_{\Omega,q,p}^{q_b}$ is supported on~$\eta\in\{0,1\}^{E(\Omega)}$ and is defined by
\begin{equation}
\label{eq:rcm-def}
{\sf RC}_{\Omega,q,p}^{q_b} (\eta) := \tfrac{1}{Z_{{\sf RC},\Omega,q,p}^{q_b}}\cdot q^{k_i(\eta)}q_b^{k_b(\eta)}p^{o(\eta)}(1-p)^{c(\eta)},
\end{equation}
where~$Z_{{\sf RC},\Omega,q,p}^{q_b}$ is a normalizing constant, $k_b(\eta)$ denotes the number of boundary clusters of~$\eta$ (i.e. connected components containing at least one boundary vertex), $k_i(\eta)$ denotes the number of the interior (i.e. non-boundary) clusters, $o(\eta)$ denotes the number of open edges in~$\eta$, and~$c(\eta)$ denotes the number of closed edges in~$\eta$.

In the classical definition of the random-cluster measure due to Fortuin and Kasteleyn~\cite{ForKas72} (see also~\cite[Section 1.2]{Gri06}), one does not distinguish between boundary clusters and interior clusters. The boundary conditions are defined only by merging (\emph{wiring}) certain boundary vertices, thus influencing the count of boundary clusters. If all boundary vertices are wired together, the boundary conditions are called \emph{wired}, and if there is no wiring, the boundary conditions are called \emph{free}. It is easy to see that~$q_b=1$ corresponds to the wired boundary conditions, $q_b=q$ corresponds to the free boundary conditions, and the measures~${\sf RC}_{\Omega,q,p}^{q_b}$ with different values of~$q_b\in [1,q]$ thus interpolate between wired and free boundary conditions (see Proposition~\ref{prop:rcm-input} below).

In~\cite{BefDum12} (see also~\cite{DumRauTas16,DumRauTas17} for alternative proofs), it was shown that when~$q\geq 1$, the model undergoes  
 the random-cluster model undergoes a phase transition at 
 \[
 	p_c(q):=\tfrac{\sqrt{q}}{\sqrt{q}+1}
\]
in terms of the correlation length~--- independently of the boundary conditions, the model exhibits exponential decay of the size of clusters when~$p<p_c(q)$, and the origin belongs to an infinite cluster with a positive probability when~$p>p_c(q)$. In particular, for all~$p\neq p_c(q)$, the infinite-volume limit does not depend on the boundary conditions.

We focus on the critical case~$p=p_c(q)$. Here it was shown that the free and the wired measures are the same~\cite{DumSidTas17} when~$q\in [1,4]$, and different~\cite{DumGagHar16,RaySpi19} when~$q>4$. This raises a natural question~--- when~$q>4$, what is the limit for each particular value of~$q_b\in [1,q]$? In the next theorem, we partially answer it.

\begin{theorem}\label{thm:q-b}
	i) Let~$q> 4$ and~$\lambda > 0$ be such that~$\sqrt{q} = e^\lambda+e^{-\lambda}$. Take any sequence~$\Omega_k$ of increasing domains. Then
	\begin{itemize}
		\item for all~$q_b\in [1, e^{-\lambda}\sqrt{q}]$, the limit of~${\sf RC}_{\Omega_k,q,p_c(q)}^{q_b}$ is the same and is equal to the wired random-cluster Gibbs measure;
		\item for all~$q_b\in [e^{\lambda}\sqrt{q}, q]$, the limit of~${\sf RC}_{\Omega_k,q,p_c(q)}^{q_b}$ is the same and is equal to the free random-cluster Gibbs measure.
	\end{itemize}
	
	ii) When~$q\in [1, 4]$, the infinite-volume limit of~${\sf RC}_{\Omega_k,q,p_c(q)}^{q_b}$ is the same, for any~$q_b\in [1,q]$, and is equal to the unique random-cluster Gibbs measure.
\end{theorem}

It is reasonable to expect that the limiting measure is wired for all~$q_b < \sqrt{q}$ and free for all~$q_b > \sqrt{q}$ (see Question \ref{ques:rc-qb} in Section~\ref{sec:open questions}).

\subsection{Overview of the proofs}\label{sec:overview of the proof}

The proofs of our main results on the height function (Theorem~\ref{thm:var}, Theorem~\ref{thm:heights-gibbs}), spin (Theorem~\ref{thm:spins-gibbs} and Theorem~\ref{thm:fkg}) and six-vertex representations (Corollary~\ref{cor:6v-gibbs}) are detailed below under the restriction that $a=b=1$. This special case is called the \emph{F-model} and was first considered by Rys~\cite{Rys63} (apparently named after Rys' advisor Fierz~\cite{Gaa79,simon2014statistical}). We chose to present the proofs in this restricted setting in order to further highlight the main ideas and to avoid encumbering the notation but they extend to the general case directly, as detailed in Section~\ref{sec:non-symmetric-case}. We note that the connection to the self-dual Ashkin--Teller model (Section~\ref{sec:ashkin-teller}) is available only in the restricted setting $a=b=1$.

In this section we provide an overview of the proofs in the restricted setting and then discuss the required modifications to handle the general case.

\subsubsection{Overview of the proofs for $a=b=1$}\label{sec:overview for a=b=1}
\paragraph{Baxter--Kelland--Wu correspondence:}
A crucial tool in our analysis is a correspondence introduced by Baxter--Kelland--Wu (BKW)~\cite{BaxKelWu76}, extending an earlier partition function relation by Temperley--Lieb \cite{LieTem71}. BKW described a correspondence of the random-cluster model on a finite planar graph $G$ with a six-vertex model on the medial graph of $G$. For domains in $\Z^2$, the random-cluster model corresponds to the F-model with the choice of parameters
\begin{equation*}
  q=(c^2-2)^2, \quad \quad p = p_c(q) := \tfrac{\sqrt{q}}{\sqrt{q} + 1}.
\end{equation*}
This choice leads to a critical random-cluster model (criticality is proven in~\cite{BefDum12} for $q\ge 1$. See also~\cite{DumRauTas16,DumRauTas17} for alternative proofs). The BKW correspondence allows to make use of recent results establishing the order of the phase transition in the random-cluster model with~$q\ge 1$: second order for~$q\in [1,4]$~\cite{DumSidTas17} and first order for~$q>4$~\cite{DumGagHar16,RaySpi19}. Our results apply in the regime $c\ge 2$, corresponding to $q\ge 4$, where the correspondence is given by a probabilistic coupling. For $c<2$, the correspondence involves a \emph{complex measure}, complicating the further transfer of results.

In the BKW correspondence it is important to take into account the effect of boundary conditions.
Theorem~\ref{thm:coupling} states the correspondence for $q\ge 4$, as probabilistic couplings of the random-cluster model and the height function representations, for two different sets of boundary conditions.

In the work of BKW, the random-cluster model is considered on a general finite planar graph under free or wired boundary condition, and boundary vertices in the corresponding six-vertex model have degree two.
On domains on $\bbZ^{2}$, we show that this six-vertex model can be defined in the standard way on full-plane arrow configurations satisfying the ice-rule if a modified parameter $c_{b}$ is introduced on the boundary (see~Section~\ref{sec:monotonicity-in-c-b}) and satisfies the relation
\[
c = c_b + \frac{1}{c_b}.
\]
We extend the correspondence to the six-vertex model with the standard boundary weight as long as the random-cluster model is modified so that connected components touching the boundary receive the cluster weight~$q_b$ satisfying the relation
\[
	q=q_b + \tfrac{q}{q_b}.
\]
The results on the modified random-cluster model are stated in Section~\ref{sec:rcm-results}.
When $q\geq 1$ and $q_{b} \in [1,q]$, we show that the modified random-cluster model is positively associated and thus the results known for the standard random-cluster model can be extended directly.
This is used in our proofs for the case $c=2$. 

\paragraph{Height representation:} When coupling the height representation and the random-cluster model, the faces on which the height function is defined correspond to the vertices and the dual vertices of the random-cluster model (even faces to vertices and odd faces to dual vertices) in such a way that the height is constant on every primal and dual cluster. The fluctuations of the height function (Theorem~\ref{thm:var}) are then analyzed for $c>2$ using the existence of an infinite cluster and for $c=2$ using Russo--Seymour--Welsh (RSW) techniques. We point out that for $c=2$ the random-cluster model is considered with the modified boundary weight $q_b$ but monotonicity of the model in the boundary weight (Proposition~\ref{prop:rcm-input} and Proposition~\ref{prop:fkg-rc}) allows to extend the known RSW estimates to this case. The case $c=2$ of Theorem~\ref{thm:heights-gibbs} may be deduced from Theorem~\ref{thm:var} and its proof.

In order to prove Theorem~\ref{thm:heights-gibbs} in the case $c>2$, a more detailed analysis of the height functions is performed. More precisely, when~$q>4$, it is known~\cite{DumGagHar16} that the critical random-cluster measures on finite domains with wired boundary conditions converge to an infinite-volume limit that exhibits an infinite cluster with finite `holes' whose diameters have exponentially decaying tail probability. Assigning heights to the primal and dual clusters according to the BKW coupling, this translates to the convergence of the height-function measure with parameter~$c>2$ on \emph{even domains} with $0,1$-boundary conditions (and a modified boundary weight) to an infinite-volume height-function measure~${\sf HF}_{\mathrm{even}, c}^{0,1}$ that exhibits an infinite cluster of diagonally-adjacent faces of height~$0$ (with holes whose diameter has exponentially decaying tail probability). Similarly, the measure~${\sf HF}_{\mathrm{odd},c}^{0,1}$ is obtained as the limit over \emph{odd domains} and exhibits an infinite cluster of height~$1$. However, it is a priori not clear whether these two measures are equal.

The argument proving that~${\sf HF}_{\mathrm{even}, c}^{0,1} ={\sf HF}_{\mathrm{odd},c}^{0,1}$ is one of the main novelties of the current article. Monotonicity of the heights implies that ${\sf HF}_{\mathrm{odd},c}^{0,1}$ stochastically dominates ${\sf HF}_{\mathrm{even}, c}^{0,1}$ and we derive the equality of the measures by showing stochastic domination in the opposite direction.

As the first step, we study the set of even faces $u$ where $h(u)\ge 2$ and the set where $h(u)\le 0$ when $h$ is sampled from~${\sf HF}_{\mathrm{odd},c}^{0,1}$. The standard connectivity structure on the even faces is to link $(i,j)$ with~$(i\pm 1, j\pm 1)$. An important trick in our argument is to augment this to a \emph{triangular lattice connectivity}, denoted $\mathbb{T}$-connectivity, in which $(i,j)$ is linked with~$(i\pm 1, j\pm 1)$ and also to~$(i\pm 2,j)$. As the $\mathbb{T}$-connectivity is still planar, standard percolation techniques (uniqueness and non-coexistence of infinite clusters) imply that, in this connectivity, there is at most one infinite cluster of $h(u)\ge 2$ and at most one infinite cluster of $h(u)\le 0$ and these may not coexist. From this we deduce that $h(u)\ge 2$ cannot have an infinite cluster in the $\mathbb{T}$-connectivity. The latter uses several ideas: symmetry between the set $h(u)\ge 2$ and the set $\tilde{h}(u)\le 0$ when $\tilde{h}$ is sampled from the (naturally-defined)~${\sf HF}_{\mathrm{odd},c}^{2,1}$ measure; monotonicity arguments; the non-coexistence statement. Consequently, as the $\mathbb{T}$-connectivity is dual to itself, there are infinitely many disjoint $\mathbb{T}$-circuits on which $h(u)\le 0$.

As a second step, we define a new height function $h'$ by
\[
	h'((i,j)):=1-h((i-1),j)).
\]
It is straightforward that $h'$ is distributed as~${\sf HF}_{\mathrm{even}, c}^{0,1}$. Let $\mathcal{C}$ be a $\mathbb{T}$-circuit on which $h(u)\le 0$ and set $\mathcal{C}':=\{(i,j)\colon (i-1,j)\in\mathcal{C}\}$ so that $h'(u)\ge 1$ on $\mathcal{C}'$. The crucial observation in the argument (indeed, the main use of the $\mathbb{T}$-connectivity) is the following: On any finite connected component of $\Z^2\setminus(\mathcal{C}\cup\mathcal{C}')$ the boundary values in $h'$ are necessarily larger or equal to those of $h$ (see Figure~\ref{fig:height-function-flip}). Thus, conditioned on $\mathcal{C}$, the distribution of $h'$ in this component stochastically dominates that of $h$. As there are infinitely many disjoint $\mathbb{T}$-circuits with $h(0)\le 0$, this implies that~${\sf HF}_{\mathrm{even}, c}^{0,1}$ stochastically dominates~${\sf HF}_{\mathrm{odd}, c}^{0,1}$, as we wanted to show.

A technical point for concluding Theorem~\ref{thm:heights-gibbs} is that the height-function measure on finite domains needs to be taken with a modified boundary weight and it is a priori possible that its effect remains in the infinite-volume limit. However, after proving the equality of the two infinite-volume measures, an extra monotonicity argument (using Proposition~\ref{prop:monotone-c-b-heights}) allows to adjust the boundary weight (Corollary~\ref{cor:therm-lim-heights-c-b}).

\paragraph{Spin and FK-Ising representations:}
As discussed in Section~\ref{sec:spin-rep}, it is useful to regard the spin representation $\sigma$ as consisting of two coupled Ising configurations $\sigma^\bullet$ and $\sigma^\circ$. Here $\sigma^\bullet$ ($\sigma^\circ$) is the restriction of $\sigma$ to the even (odd) faces (a mixed Ashkin--Teller representation). It is natural to condition on one of the Ising configurations, say $\sigma^\bullet$, and consider the other configuration $\sigma^\circ$ which is then exactly an Ising model on a graph obtained from the even faces by contracting some of the diagonal edges in a manner specified by $\sigma^\bullet$. This Ising model is ferromagnetic when $c\ge 1$ and then naturally has an associated \emph{FK-Ising} bond configuration which we denote by $\xi$. When conditioning on $\sigma^\bullet$ one can treat $\xi$ using the standard tools, in particular, its known monotonicity (FKG) properties. In our arguments, however, it is also important to consider the unconditional measure of $\xi$, i.e., when averaging over $\sigma^\bullet$, and this measure is later referred to as an \emph{FK-Ising representation of the Ashkin--Teller model} (related to a random-cluster representation introduced in~\cite{PfiVel97}). As it turns out, the averaged measure satisfies monotonicity (FKG) properties when $c\ge 2$ (see Proposition~\ref{prop:fkg-fk-ising}) but this is not used in proving the theorems of Section~\ref{sec:spin-rep}.

The monotonicity properties of the spin configuration $\sigma^\bullet$ stated in Theorem~\ref{thm:fkg} follow by checking the FKG lattice condition, but the calculation is non-trivial as the probability density of $\sigma^\bullet$ involves a sum over all possibilities for $\sigma^\circ$. This sum can be rewritten using a product of partition functions of Ising models on different domains and eventually the required inequality is proved using the monotonicity properties of the standard FK-Ising model.

For $c>2$, Theorem~\ref{thm:spins-gibbs} is derived from Theorem~\ref{thm:heights-gibbs} using the fact that the spin representation is obtained from the modulo $4$ of the height function. An additional ingredient is the equality (a standard relation for FK-Ising models)
\begin{equation*}
  {\sf Spin}_c^{++}(\sigma^\circ(u)) = {\sf FKIs}_c(u\xleftrightarrow{\xi} \infty)
\end{equation*}
which relates the expectation of the spin $\sigma(u)$ in the limiting measure to the probability that $u$ is connected to infinity in the FK-Ising representation of the Ashkin--Teller model. This already proves the non-negativity of the spin expectations and strict positivity is then deduced from the properties in Theorem~\ref{thm:heights-gibbs}.

For $c=2$, Theorem~\ref{thm:spins-gibbs} is derived by coupling the spin representation directly to the random-cluster model with modified boundary cluster weight $q_b$ and using the RSW estimates which are available there (again, using monotonicity in $q_b$). The description of infinite clusters in the extremal invariant Gibbs measures relies on percolation techniques and uses the monotonicity properties of Theorem~\ref{thm:fkg}.

\paragraph{Six-vertex representation:} The properties of flat Gibbs states of the six-vertex model stated in Corollary~\ref{cor:6v-gibbs} are derived from Theorem~\ref{thm:spins-gibbs}. Recall that a flat Gibbs state satisfies that sampled configurations contain infinitely many disjoint oriented circuits surrounding the origin and consisting of alternating vertical and horizontal edges. The connection to Theorem~\ref{thm:spins-gibbs} is enabled by noting that the spins on either side of such an oriented circuit must take a constant value.

\paragraph{Self-dual Ashkin--Teller model:} As discussed in Section~\ref{sec:ashkin-teller}, the self-dual Ashkin--Teller model is known to be in correspondence with the six-vertex model (see, e.g.,~\cite[Section 12.9]{Bax82}). This correspondence is upgraded here to a probabilistic coupling of the Ashkin--Teller model, the spin representation and the aforementioned FK-Ising representation of the Ashkin--Teller model. The correspondence maps the entire self-dual line to the six-vertex model with parameters $a=b=1$ and $c>1$ (the limiting case $c=1$ is discussed in Section~\ref{sec:open questions}), with the regime $J<U$ mapped to the regime $c>2$. The coupling is described in Proposition~\ref{prop:coupling-AT-6V} and we review its main features here.

The Ashkin--Teller configuration $(\tau,\tau')$ is defined on the lattice of even faces (with diagonal connectivity). Under the coupling, the spin configuration on the even faces $\sigma^\bullet$ equals the product of $\tau\tau'$, which already implies long-range order for the product when $J<U$ (equation~\eqref{eq:AT-corr-bounded} of Theorem~\ref{thm:ashkin-teller}) and also the FKG inequality of Corollary~\ref{cor:fkg-AT}, by using Theorem~\ref{thm:spins-gibbs} and Theorem~\ref{thm:fkg}. Recall now that the FK-Ising representation $\xi$ is a percolation configuration on the diagonal bonds between the odd faces, whence the dual percolation $\xi^*$ is on the bonds between even faces. Under the coupling, conditioned on $\sigma^\bullet$ (which equals $\tau\tau')$ and on $\xi^*$, the configuration $\tau$ is obtained by assigning a uniform spin value to each cluster of $\xi^*$, independently among the clusters; this also specifies $\tau'$.

The coupling directly links the decay of correlations in $\tau$ to the connection probabilities in $\xi^*$. Existence of an infinite cluster in $\xi$ is deduced from Theorem~\ref{thm:heights-gibbs}. It is additionally shown (Proposition~\ref{prop:fkg-fk-ising}) that $\xi^*$ satisfies the FKG inequality in the regime $c\ge 2$ (this should be compared with the fact that the spins $\sigma^\bullet$ are shown in Theorem~\ref{thm:fkg} to satisfy the inequality in the wider range $c\ge 1$).
By a general non-coexistence theorem (\cite[Theorem 1.5]{DumRauTas17}), all clusters in $\xi^{*}$ are finite. An exponential decay of their diameters (stated in~\eqref{eq:AT-corr-exp-decay})  is then derived from the BKW coupling and an exponential decay of dual clusters in the critical wired random-cluster measure for~$q>4$.

\subsubsection{Extension to the case~$a\neq b$}
\label{sec:non-symmetric-case}

We proceed to explain the modifications required to our arguments when the parameters of the six-vertex model satisfy $a\neq b$.

The main required modification is to the coupling between the six-vertex model and the random-cluster model that is described in Section~\ref{sec:coupling}. In the case $a=b=1$, Theorem~\ref{thm:coupling} of that section provides a coupling of the six-vertex measure~${\sf HF}_{\calD,c}^{0,1}$ with the random-cluster measure with modified boundary-cluster weight~${\sf RC}_{\calD^\bullet,q,p_c(q)}^{q_b}$, when the parameters satisfy
\[
		c=e^{\lambda/2} + e^{-\lambda/2}, \quad q=[e^{\lambda} + e^{-\lambda}]^2, \quad p_c(q) :=\tfrac{\sqrt{q}}{\sqrt{q}+1}, \quad q_b = e^{-\lambda}\sqrt{q},
\]
with the coupling supported on compatible configurations (i.e., assigning constant heights on primal and dual clusters).
The theorem additionally provides a coupling of the six-vertex measure with modified boundary-vertex weight~${\sf HF}_{\calD,c}^{0,1;c_b}$ and the standard random-cluster measure~${\sf RC}_{\calD^\bullet,q,p_c(q)}^1$ when
\[
	c=e^{\lambda/2} + e^{-\lambda/2}, \quad c_b = e^{\lambda/2}, \quad q=[e^{\lambda} + e^{-\lambda}]^2,\quad p_c(q) :=\tfrac{\sqrt{q}}{\sqrt{q}+1}.
\]
In the case $a\neq b$ there is still a coupling of the six-vertex model with a random-cluster model, as described by Baxter--Kelland--Wu~\cite{BaxKelWu76}. In this case, the relation between $q$ and the parameters $a,b,c$ is given by
\begin{equation}\label{eq:Delta and q}
\Delta= \tfrac{a^2+b^2 - c^2}{2ab},\quad \sqrt{q} = -2\Delta.
\end{equation}
The random-cluster model obtained assigns different probabilities $p_c^h, p_c^v$ to open horizontal and open vertical edges, respectively, according to the formulas
\begin{equation}\label{eq:horizontal and vertical probabilities}
p_c^h := \tfrac{b\sqrt{q}}{b\sqrt{q} + a}, \quad p_c^v := \tfrac{a\sqrt{q}}{a\sqrt{q} + b}.
\end{equation}
Thus one obtains, with the same proof, an analog of the coupling theorem, Theorem~\ref{thm:coupling}, with the same notion of compatible configurations and with the parameters related by~\eqref{eq:Delta and q} and~\eqref{eq:horizontal and vertical probabilities} and additionally by
\begin{equation}
  \cosh \lambda = -\Delta,\quad q_b = e^{-\lambda}\sqrt{q},\quad c_b = e^{\lambda/2}
\end{equation}
where the relations determine $\lambda$ only up to its sign but both choices lead to a valid coupling. Additionally, in the coupling in which the modified boundary-vertex weight $c_b$ is used, the weights assigned to boundary vertices of the first four types in Figure~\ref{fig:6v-arrow-config} are fixed to one.

Our arguments further use input on the critical properties of the random-cluster model as described in items (iv) and (v) of Proposition~\ref{prop:rcm-input}. To adapt these, we rely on the work of Duminil-Copin--Li--Manolescu~\cite{DumLiMan18} who extended many of the results known for the random-cluster model on the square lattice to the setting of the random-cluster model on isoradial graphs. In the language of that work, the above random-cluster model with edge probabilities~\eqref{eq:horizontal and vertical probabilities} is the critical random-cluster model on the isoradial graph given by a rectangular grid (the criticality condition is $\frac{p_c^v}{1-p_c^v}\cdot \frac{p_c^h}{1-p_c^h} = q$). This allows us to use the following results from~\cite{DumLiMan18}: (i) for~$q\in [1,4]$, the phase transition is of the second order and one has Russo--Seymour--Welsh estimates on crossings~\cite[Theorem 1.1]{DumLiMan18}; (ii) for~$q> 4$, the phase transition is of the first order and the wired infinite-volume measure exhibits a unique infinite cluster and exponential decay of dual clusters~\cite[Theorem 1.2]{DumLiMan18}.

A modification of a more minor nature concerns the formulas involving the FK--Ising representation of the Ashkin--Teller model $\xi$ introduced in Section~\ref{sec:fk-ising}. As explained in Section~\ref{sec:overview for a=b=1} for the case $a=b$, we obtain $\xi$ by conditioning on $\sigma^\bullet$ and then letting $\xi$ be the FK-Ising bond configuration corresponding to $\sigma^\circ$ (which, after the conditioning, is a ferromagnetic Ising model). Exactly the same construction of $\xi$ is used in the case $a\ne b$ and thus the arguments involving $\xi$ are still applicable, though the precise formulas governing the distribution of $\xi$ are modified. Specifically, the construction of $\xi$ conditioned on both $\sigma^\bullet$ and $\sigma^\circ$ which is described in Figure~\ref{fig:fk-ising} is used with the following modification:
in vertices where both the diagonally-adjacent spins of $\sigma^\bullet$ agree and the diagonally-adjacent spins of $\sigma^\circ$ agree, the probability with which the edge of $\xi$ is present equals $\tfrac{c-a}{c}$ if the top-left face of the vertex is odd and equals $\tfrac{c-b}{c}$ if it is even.

The rest of the arguments used for deriving our results apply directly to the case $a\ne b$ with only notational changes. We emphasize in particular that the arguments involving $\mathbb{T}$-circuits do not rely on reflection symmetry. We also point out that the connection with the Ashkin--Teller model and the results of Section~\ref{sec:ashkin-teller} are specific to the case $a=b$ and are not extended here (the case $a\ne b$ of the six-vertex model corresponds to a staggered Ashkin--Teller model). Theorem~\ref{thm:q-b}, concerning the random-cluster model with modified boundary weights, and its proof directly extend to the critical random-cluster model with the edge probabilities~\eqref{eq:horizontal and vertical probabilities}.

\section{Coupling between the six-vertex and the random-cluster models}
\label{sec:coupling}

The correspondence between the six-vertex and the random-cluster models is known since the seminal paper by Temperley and Lieb~\cite{LieTem71} and was described geometrically by Baxter, Kelland, and Wu~\cite{BaxKelWu76} (BKW).
As outlined in Sections \ref{sec:monotonicity-in-c-b} and \ref{sec:overview of the proof}, we upgrade the correspondence to a coupling and show how to describe the boundary condition for the six-vertex model by introducing the boundary parameter $c_{b}$ --- instead of changing the set of configurations as in the work of BKW.
Additionally, we extend the statement to the setup where the parameters in the six-vertex model are the same inside the domain and on its boundary. 
Following our work, this extension was used in~\cite{RaySpi19} to provide a new proof of the first-order phase transition in the random-cluster model with~$q> 4$. 

We start by introducing the graphs where the random-cluster model will be defined. Let~$\calD$ be a domain. Recall that~$\partial\calD$ denotes the circuit formed by boundary edges of~$\calD$, and~$\partial_{\mathrm{ext}}\calD$ denotes the set of faces in~$\bbZ^2\setminus\calD$ that are adjacent to faces in~$\calD$. For every face~$u\in \partial_{\mathrm{ext}}\calD$ and every vertex~$z$ on~$\partial\calD$ belonging to~$u$, we call the pair~$(u,z)$ a \emph{corner} of~$\calD$. The corner~$(u,z)$ is called even or odd according to the parity of~$u$.

Define a graph on the set of all even faces and corners of~$\calD$ by drawing edges according to the rule:
\begin{itemize}
	\item any two even faces of~$\calD$ having a common vertex are linked by an edge;
	\item even corner~$(u,z)$ and even face~$v$ of~$\calD$ are linked by an edge if~$z\in v$;
	\item even corners~$(u,z)$ and~$(v,z')$ of~$\calD$ are linked by an edge if~$z=z'$.
\end{itemize}
In this graph, we identify every two corners~$(u,z)$ and~$(u,z')$ such that~$zz'$ is an edge of~$\calD$. The resulting graph is denoted by~$\calD^\bullet$. All interior faces of~$\calD^\bullet$ are of degree four and are in bijection with odd faces of~$\calD$. Also, if one merges together corners of~$\calD$ corresponding to the same face, one obtains from~$\calD^\bullet$ a subgraph of a square lattice. Graph~$\calD^\circ$ is defined in the same way on odd faces and corners of~$\calD$; see Figure~\ref{fig:domains}.

\begin{figure}
	\begin{center}
		\includegraphics[width=0.49\textwidth, page=1]{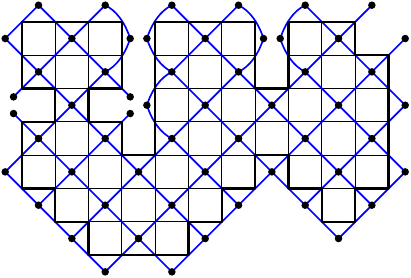}
		\includegraphics[width=0.49\textwidth, page=2]{domains.pdf}
		\caption{Domain~$\calD$ (in black) with the boundary~$\partial\calD$ (in bold). \emph{Left:} Graph~$\calD^\bullet$ on even faces (bullets) with edges marked in blue. \emph{Right:} Graph~$\calD^\circ$ on even faces (circles) with edges marked in red. \emph{Both:} Graphs~$\calD^\bullet$ and~$\calD^\circ$ are dual to each other. Corners of~$\calD$ are bullets and circles located outside of~$\partial\calD$. Inside of~$\partial\calD$, both~$\calD^\bullet$ and~$\calD^\circ$ are subgraphs of a square lattice. This is not the case outside of~$\partial\calD$, since some corners corresponding to the same face are not merged into one vertex.}
		\label{fig:domains}
	\end{center}
\end{figure}

By the definition, edges of~$\calD^\bullet$ and~$\calD^\circ$ are in bijection with vertices of~$\calD$. Given a vertex~$z$ of~$\calD$, denote the edges of~$\calD^\bullet$ and~$\calD^\circ$ corresponding to it by~$e_z$ and~$e_z^*$. For any edge configuration~$\eta\in\{0,1\}^{E(\calD^\bullet)}$, the dual configuration~$\eta^* \in \{0,1\}^{E(\calD^\circ)}$ is defined by:
\[
	\forall e\in E(\calD^\bullet) \quad \eta^*(e^*) = 1- \eta(e).
\]
Every height function~$h$ can be considered as a function on vertices of~$\calD^\bullet$ and~$\calD^\circ$ by setting~$h(u,z):=h(u)$, for every corner~$(u,z)$ of~$\calD$.

We say that a height function $h\in \calE_{\sf hf}^{0,1}(\calD)$ is \emph{compatible} with an edge-configuration $\eta\in\{0,1\}^{E(\calD^\bullet)}$ and write~$\eta\perp h$, if it has a constant value at every cluster of~$\eta$ and~$\eta^*$ (primal and dual clusters); see Figure~\ref{fig:coupling-height-rcm}. We say that cluster~$\calC$ of~$\eta$ and cluster~$\calC^*$ of~$\eta^*$ are adjacent, and denote this by~$\calC\sim\calC^*$, if there exist~$u\in\calC$ and~$u^*\in\calC^*$ that correspond to two adjacent faces of~$\calD$ or to two corners of~$\calD$ that share a vertex.

For two adjacent clusters~$\calC$ and~$\calC^*$, we write~$\calC\prec\calC^*$ if~$\calC$ is surrounded by~$\calC^*$.

Recall the height-function measures~${\sf HF}_{\calD,c}^{0,1}$ and~${\sf HF}_{\calD,c}^{0,1;c_b}$ defined in Sections~\ref{sec:results-heights} and~\ref{sec:monotonicity-in-c-b} (where we fix $a=b=1$), the random-cluster measure~${\sf RC}_{\calD^\bullet,q,p}^{q_b}$ defined in Section~\ref{sec:rcm-results}, and that~$p_c(q) :=\tfrac{\sqrt{q}}{\sqrt{q}+1}$.

\begin{theorem}
	\label{thm:coupling}
	1) Let~$\calD$ be a domain and~$\lambda\in\bbR$. Take
	\[
		c=e^{\lambda/2} + e^{-\lambda/2}, \quad q=[e^{\lambda} + e^{-\lambda}]^2, \quad q_b = e^{-\lambda}\sqrt{q}.
	\]
	Then, the measures~${\sf HF}_{\calD,c}^{0,1}$ and~${\sf RC}_{\calD^\bullet,q,p_c(q)}^{q_b}$ can be coupled in such a way that the joint law is supported on pairs of compatible configurations~$(h,\eta)$ and can be written in either of the two following ways:
	\begin{align}
		\label{eq:coupling-hf-form}	
		\bbP_{\sf cluster}(h,\eta)&\propto \exp{\left[ \lambda\sum_{\calC\sim\calC^*} {(h(\calC^*) - h(\calC))(-1)^{\mathbbm{1}_{\calC \prec \calC^*}}}\right]},\\
		\label{eq:coupling-edge-form}
		\bbP_{\sf edge}(h,\eta)&\propto\exp{\left[ \tfrac{\lambda}{4}\sum_{uv\in \calD^\bullet}{(h(u^*) + h(v^*) - h(u) - h(v))(-1)^{\mathbbm{1}_{uv\not\in \eta}}}\right]},
	\end{align}
	where in~\eqref{eq:coupling-edge-form}, vertices~$u^*$ and~$v^*$ are the endpoints of the edge~$(uv)^*$.

	2) If~$\calD$ is an even domain, then the same holds for~${\sf HF}_{\calD,c}^{0,1;c_b}$ and~${\sf RC}_{\calD^\bullet,q,p_c(q)}^1$ when
	\[
		c=e^{\lambda/2} + e^{-\lambda/2}, \quad c_b = e^{\lambda/2}, \quad q=[e^{\lambda} + e^{-\lambda}]^2.
	\]
\end{theorem}

\begin{proof}
	1) We write~$\bbP_{\sf hf}$ and~$\bbP_{\sf RC}$ instead of~${\sf HF}_{\calD,c}^{0,1}$ and~${\sf RC}_{\calD^\bullet,q,p_c(q)}^{q_b}$ for brevity. To prove the claim, it is enough to show that~$\bbP_{\sf edge}(h,\eta) = \bbP_{\sf cluster}(h,\eta)$ and that:
	\begin{align}
		\label{eq:marginal-hf}
		&\forall h\in \calE_{\sf hf}^{0,1}(\calD)&
		\sum_{\eta \perp h} \bbP_{\sf edge}(h,\eta)&= \bbP_{\sf hf}(h)\\
		\label{eq:marginal-rc}
		&\forall\eta\in \{0,1\}^{E(\calD^\bullet)}&
		\sum_{h \perp \eta} \bbP_{\sf cluster}(h,\eta) &= \bbP_{\sf RC}(\eta).
	\end{align}
	Relation~\eqref{eq:marginal-hf} follows immediately. Indeed, summing~$\bbP_{\sf edge}$ over all edge configurations compatible with~$h$, one obtains that every vertex of~$\calD$ contributes~$e^{\lambda/2}+e^{-\lambda/2} = c$ if the corresponding four heights agree on diagonals, and it contributes~$1$ otherwise. This coincides with the definition of~$\bbP_{\sf hf}(h)$.
	
	We now show~\eqref{eq:marginal-rc}. The height at the unique boundary cluster~$\calC^*$ of~$\eta^*$ equals~$1$ and the height at every boundary cluster~$\calC$ of~$\eta$ equals~$0$, whence the contribution of each such pair~$(\calC, \calC^*)$ to the LHS of~\eqref{eq:marginal-rc} equals~$e^{-\lambda} = q_b/\sqrt{q}$. All height functions~$h$ compatible with~$\eta$ can be obtained by exploring the adjacency graph of clusters of~$\eta$ and~$\eta^*$ starting from the boundary and at each step choosing independently whether the height is increasing or decreasing by one. Thus, every non-boundary cluster of~$\eta$ and~$\eta^*$ contributes~$e^\lambda + e^{-\lambda} = \sqrt{q}$ to the LHS of~\eqref{eq:marginal-rc}. Substituting this in~\eqref{eq:marginal-rc}, we get
	\begin{align*}
		\sum_{h \perp \eta} \bbP_{\sf cluster}(h,\eta)
		&\propto \sqrt{q}^{k_i(\eta) + k_i(\eta^*)}\left(\tfrac{q_b}{\sqrt{q}}\right)^{k_b(\eta)} = \sqrt{q}^{k(\eta^*)-k(\eta)-1}q^{k_i(\eta)}q_b^{k_b(\eta)} \\
		&= \left(\tfrac{p_c(q)}{1-p_c(q)}\right)^{o(\eta) -|V(\calD^\bullet)|} q^{k_i(\eta)}q_b^{k_b(\eta)} \propto \bbP_{\sf RC}(\eta),
	\end{align*}
	where in the first line we use that~$k_i(\eta) = k(\eta) - k_b(\eta)$ and~$k_i(\eta^*) = k(\eta^*) - 1$; in the second line we used the identity~$k(\eta^*) - k(\eta) -1 = o(\eta) -|V(\calD^\bullet)|$ (follows from Euler's formula and can be checked by induction in~$o(\eta)$) and the fact that~$\sqrt{q} = \tfrac{p_c(q)}{1-p_c(q)}$.
	
	\begin{figure}
		\begin{center}
			\includegraphics[scale=1.2, page=1]{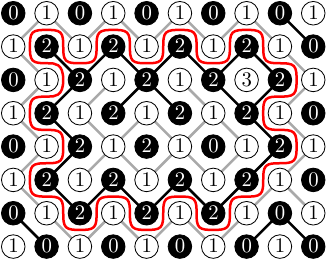} \quad\quad
			\includegraphics[scale=1.2, page=2]{coupling-height-flip.pdf}
			\caption{\emph{Left:} Edge configuration~$\eta$ (black primal edges, gray dual edges) and a height function~$h$ with~$0,1$ boundary conditions that is compatible with~$\eta$. The red cyclic path separates cluster~$\calC$ of~$\eta$ (inside) from cluster~$\calC^*$ of~$\eta^*$ (outside). \emph{Right:} Height function~$h_{\calC,\calC^*}$~--- all heights surrounded by the red loop are decreased by~$2$.}
			\label{fig:coupling-height-rcm}
		\end{center}
	\end{figure}
	
	It remains to show that~\eqref{eq:coupling-edge-form} and~\eqref{eq:coupling-hf-form} describe the same probability measure. For any pair of adjacent non-boundary clusters~$\calC$ and~$\calC^*$ that satisfy~$\calC\prec\calC^*$ and any height function~$h$ compatible with~$\eta$ and such that~$h(\calC) - h(\calC^*) = 1$, define~$h_{\calC,\calC^*}$ on the faces of~$\bbZ^2$ such that~$h_{\calC,\calC^*}(\cdot) = h(\cdot) - 2$ on~$\calC$ and its interior, and~$h_{\calC,\calC^*}(\cdot) = h(\cdot)$ outside of~$\calC$. It is immediate that
	\[
		\bbP_{\sf cluster}(h_{\calC,\calC^*},\eta) = e^{-2\lambda}\bbP_{\sf cluster}(h,\eta).
	\]
	We now prove that the same is true for~$\bbP_{\sf edge}$; see Figure~\ref{fig:coupling-height-rcm} for an illustration. The edges of~$\calD$ separating~$\calC$ from~$\calC^*$ form a cyclic path~$\ell$ of alternating vertical and horizontal edges that does not visit twice the same edge (but can visit twice the same vertex of~$\calD$). Consider the difference between the expression in~\eqref{eq:coupling-edge-form} computed for~$h_{\calC,\calC^*}$ and for~$h$:
	\begin{align}\label{eq:proof-coupling-diff-contr}
		\Delta_{\calC,\calC^*}:
		&= \tfrac{\lambda}{4}\sum_{uv\in \calD^\bullet}{(h_{\calC,\calC^*}(u^*) + h_{\calC,\calC^*}(v^*)- h_{\calC,\calC^*}(u) - h_{\calC,\calC^*}(v))(-1)^{\mathbbm{1}_{uv\not\in \eta}}} \\
		&- \tfrac{\lambda}{4}\sum_{uv\in \calD^\bullet}{(h(u^*) + h(v^*) - h(u) - h(v))(-1)^{\mathbbm{1}_{uv\not\in \eta}}}.\nonumber
	\end{align}
	Only edges of~$\calD^\bullet$ corresponding to vertices on~$\ell$ have a non-zero contribution to~$\Delta_{\calC,\calC^*}$:
	\begin{itemize}
		\item[--] if~$z$ has degree~$2$ in~$\ell$, then~$e_z$ contributes~$\lambda/2$ if~$e_z\in\eta$ and~$-\lambda/2$ if~$e_z^*\in\eta^*$;
		\item[--] if~$z$ has degree~$4$ in~$\ell$, then~$e_z$ contributes~$\lambda$ if~$e_z\in\eta$ and~$-\lambda$ if~$e_z^*\in\eta^*$.
	\end{itemize}
	Going along~$\ell$ in a clockwise direction, we obtain that every left turn occurs at an edge of~$\eta$ (and contributes~$\lambda/2$ to~$\Delta_{\calC,\calC^*}$) and every right turn occurs at an edge of~$\eta^*$ (and contributes~$-\lambda/2$ to~$\Delta_{\calC,\calC^*}$). Since~$\ell$ is a non-self-intersecting curve oriented clockwise, it has~$4$ more right turns than left turns, whence~$\Delta (h_{\calC,\calC^*}, h) = -2\lambda$ and
	\[
		\bbP_{\sf edge}(h_{\calC,\calC^*},\eta) = e^{-2\lambda}\bbP_{\sf edge}(h,\eta).
	\]
	The operation~$h\mapsto h_{\calC,\calC^*}$ can be described analogously when~$\calC^*$ is surrounded by~$\calC$. The combination of such operations can bring any height function~$h\in \calE_{\sf hf}^{0,1}(\calD)$ to the~$0,1$ height function that is equal to~$0$ at all even faces and to~$1$ at all odd faces. Since we showed that this operation has the same effect on~$\bbP_{\sf edge}$ and~$\bbP_{\sf cluster}$, it is enough to show that the two probability measures are equal when~$h$ is a~$0,1$ height function. In the latter case, we have:
	\begin{align*}
		\bbP_{\sf cluster}(h,\eta) &\propto \exp{\left[\lambda \sum_{\calC\sim \calC^*}(-1)^{\mathbbm{1}_{\calC \prec \calC^*}}\right]} = \exp{\left[\lambda(k(\eta^*) - k(\eta) - 1)\right]},\\
		\bbP_{\sf edge}(h,\eta) &\propto \exp{\left[ \tfrac{\lambda}{4}\sum_{uv\in \calD^\bullet}{2(-1)^{\mathbbm{1}_{uv\not\in \eta}}}\right]} = \exp{\left[ \tfrac{\lambda}{2}(o(\eta) - c(\eta))\right]}.
	\end{align*}
	By Euler's formula, the right-hand sides of the above equations are the same up to a constant and this finishes the proof.\\

	2) The second item is a straightforward consequence of the first item when one conditions all boundary edges to be open (where an edge of~$\calD^\bullet$ is a boundary edge if its endpoints are corners of~$\calD$). Indeed, this sets wired boundary conditions for the random-cluster model and the contribution of the boundary edges to~\eqref{eq:marginal-hf} equals~$e^{\lambda/2}=c_b$.
\end{proof}

Since height functions are in correspondence with spin configurations (see Section~\ref{sec:spin-rep}), the coupling with the random-cluster model can also be stated for the spin representation. Similarly to above, we say that a spin configuration~$\sigma\in \calE_{\sf spin}^{++}(\calD)$ and an edge-configuration~$\eta\in \{0,1\}^{E(\calD^\bullet)}$ are compatible if~$\sigma$ is constant at each cluster of~$\eta$ and~$\eta^*$.

\begin{corollary}\label{cor:coupling-spins}
	In the notation of Theorem~\ref{thm:coupling}, the measures~${\sf Spin}_{\calD,c}^{++}$ and~${\sf RC}_{\calD^\bullet,q,p_c(q)}^{q_b}$ (and the measures~${\sf Spin}_{\calD,c}^{+;e^{\lambda/2}}$ and~${\sf RC}_{\calD^\bullet,q,p_c(q)}^{1}$ if~$\calD$ is even) can be coupled in such a way that the joint law is supported on pairs of compatible configurations~$(\sigma,\eta)$ and can be written in either of the two following ways:
	\begin{align}
		\label{eq:coupling-spin-form}
		\bbP_{\sf cluster}(\sigma, \eta) &\propto \exp{\left[ \lambda\sum_{\calC\sim\calC^*}
{\sigma(\calC)\sigma(\calC^*)(-1)^{\mathbbm{1}_{\calC \prec \calC^*}}}\right]},\\
		\label{eq:coupling-edge-form-spin}
		\bbP_{\sf edge}(\sigma, \eta) &\propto \exp{\left[ \tfrac{\lambda}{4}\sum_{uv\in \calD^\bullet}{(\sigma(u)\sigma(u^*) + \sigma(v)\sigma(v^*))(-1)^{\mathbbm{1}_{uv\not\in \eta}}}\right]}.
	\end{align}
\end{corollary}

For~$c =2$, the coupling becomes a uniform measure.

\begin{corollary}
	\label{cor:coupling-c-2}
	1) The measures~${\sf HF}_{\calD,2}^{0,1}$ and~${\sf RC}_{\calD^\bullet,4,p_c(4)}^2$ can be coupled in such a way that the joint law is a uniform measure on pairs~$(h,\eta)$ of compatible configurations. In particular, the distribution of the height at a particular face~$u$ of~$\calD$ according to~${\sf HF}_{\calD,2}^{0,1;1}$ is that of a simple random walk that starts at~$0$, at each step goes up or down by~$1$ uniformly, and makes in total as many steps as there are clusters of~$\eta$ and~$\eta^*$ surrounding~$u$, where~$\eta$ is distributed according to~${\sf RC}_{\calD^\bullet,4,p_c(4)}^2$.
	
	If~$\calD$ is an even domain, the same holds for the measures~${\sf HF}_{\calD,2}^{0,1;2}$ and~${\sf RC}_{\calD^\bullet,4,p_c(4)}^1$.
	
	2) Similarly, the measures~${\sf Spin}_{\calD,2}^{++}$ and~${\sf RC}_{\calD^\bullet,4,p_c(4)}^{1}$ can be coupled in such a way that the joint law is a uniform measure on pairs~$(\sigma,\eta)$ of compatible configurations; and the same for~${\sf Spin}_{\calD,2}^{+;1}$ and~${\sf RC}_{\calD^\bullet,4,p_c(4)}^1$ if~$\calD$ is an even domain.
\end{corollary}

\section{Input from the random-cluster model}
\label{sec:rcm-input}

In this section, we discuss some fundamental properties of the random-cluster model~${\sf RC}_{\Omega,q,p}^{q_b}$ introduced in Section~\ref{sec:rcm-results}, with a priori different weights~$q_b$ and~$q$ for boundary and non-boundary clusters. These properties are derived in a straightforward manner from the known results on the standard random-cluster model~--- classical results are described in~\cite{Gri06}, and the relevant recent results were established in~\cite{BefDum12,DumSidTas17,DumGagHar16}.

Define a partial order on~$\{0,1\}^{E(\Omega)}$ as follows: $\eta \leq \eta'$ if~$\eta(e) \leq \eta'(e)$ for any~$e\in E(\Omega)$. An event~$A\subset \{0,1\}^{E(\Omega)}$ is called increasing if its indicator is an increasing function with the respect to this partial order.

It is well-known that when~$q \geq 1$, the standard random-cluster model is positively associated (\cite[Theorem (3.8)]{Gri06}): any two increasing events are positively correlated.
Below we show that when~$q_b \in [1,q]$, the measure~$\bbP_{{\sf RC},\Omega,q,p}^{q_b}$ satisfies the FKG lattice condition, which in particular implies positive association~\cite{ForKasGin71}.

\begin{proposition}[Positive association]
\label{prop:fkg-rc}
Let~$q\geq 1$, $p\in [0,1]$, $q_b\in [1,q]$ and~$\Omega$ be a finite subgraph of~$\bbZ^2$.  Then~$\bbP_{{\sf RC},q,p}^{q_b}$ satisfies the FKG lattice condition. In particular, for any two increasing functions~$f,g$ one has:
\[
	{\sf RC}_{\Omega,q,p}^{q_b}(f(\eta)g(\eta)) \geq {\sf RC}_{\Omega,q,p}^{q_b}(f(\eta))\cdot {\sf RC}_{\Omega,q,p}^{q_b}(g(\eta)).
\]
\end{proposition}

\begin{proof}
We write~$\bbP$ instead of~${\sf RC}_{\Omega,q,p}^{q_b}$ for brevity. By~\cite[Theorem (2.22)]{Gri06} it is sufficient to consider only pairs of configurations that differ on exactly two edges. In this case, the lattice conditions takes the form:
\begin{equation}
\label{eq:fkg-rc}
\bbP(\eta^{ef})\bbP(\eta_{ef}) \geq \bbP(\eta^e_f)\cdot\bbP(\eta^f_e),
\end{equation}
where~$e,f\in E(\Omega)$, $\eta\in\{0,1\}^{E(\Omega)}$, all four configurations~$\eta^{ef}$, $\eta_{ef}$, $\eta^e_f$, $\eta^f_e$ agree with~$\eta$ on~$E(\Omega)\setminus\{e,f\}$,
$e\in \eta^{ef}\cap \eta^e_f$, $f\in \eta^{ef}\cap \eta^f_e$, $e\not\in \eta_{ef}\cup \eta^f_e$, $f\not\in \eta_{ef}\cup \eta^e_f$.

The term counting the edges cancels out in~\eqref{eq:fkg-rc} and it remains to take care of the number of clusters. We are going to use the following notation:
\begin{align*}
\Delta_i(e) &= k_i(\eta_{ef}) - k_i(\eta^e_{f}),& \Delta_i(f) &= k_i(\eta_{ef}) - k_i(\eta^f_{e}),&  \Delta_i(e,f) &= k_i(\eta_{ef}) - k_i(\eta^{ef}); \\
\Delta_b(e) &= k_b(\eta_{ef}) - k_b(\eta^e_{f}),& \Delta_b(f) &= k_b(\eta_{ef}) - k_b(\eta^f_{e}),&  \Delta_b(e,f) &= k_b(\eta_{ef}) - k_b(\eta^{ef}).
\end{align*}
In this notation, we need to show that:
\[
\log q \cdot (\Delta_i(e) + \Delta_i(f) - \Delta_i(e,f)) + \log q_b \cdot (\Delta_b(e) + \Delta_b(f) - \Delta_b(e,f)) \geq 0.
\]
It is easy to see that~$\Delta_i(e) + \Delta_i(f) - \Delta_i(e,f)\geq 0$. Also, $\Delta_b(e) + \Delta_b(f) - \Delta_b(e,f)\geq -1$, because if~$\Delta_b(e,f) = 2$, then each of~$e$ and~$f$ connects two different boundary clusters. Using the inequalities~$q\geq 1$ and~$q_b\in [1,q]$, it is then enough to show that~$\Delta_b(e) + \Delta_b(f) - \Delta_b(e,f) = -1$ implies~$\Delta_i(e) + \Delta_i(f) - \Delta_i(e,f)\geq 1$.

Indeed, assume that~$\Delta_b(e) + \Delta_b(f) - \Delta_b(e,f) = -1$. Then~$\Delta_b(e,f) = 1$, $\Delta_b(e) = 0$, $\Delta_b(f) = 0$. This means that clusters in~$\eta_{ef}$ containing endpoints of~$e$ and~$f$ can be denoted by~$C_1$, $C_2$, $C_3$, so that: $e$ connects~$C_1$ and~$C_2$; $f$ connects~$C_2$ and~$C_3$; $C_1$ and~$C_3$ are boundary clusters; $C_2$ is an interior cluster. Clearly, in this case~$\Delta_i(e) + \Delta_i(f) - \Delta_i(e,f) = 1$ and the proof is finished.
\end{proof}

Denote by~${\sf RC}_{\Omega,q,p}^{\mathrm{wired}}$ and~${\sf RC}_{\Omega,q,p}^{\mathrm{free}}$ the standard random-cluster measures on~$\Omega\subset \bbZ^2$ with wired and free boundary conditions, respectively. As defined above,
	\[
		p_c(q) = \tfrac{\sqrt{q}}{\sqrt{q}+1}.
	\]
For two measures~$\mu$ and~$\nu$ on~$\{0,1\}^{E(\Omega)}$, one says that~$\mu$ stochastically dominates~$\nu$ and writes~$\mu\succeq \nu$, if~$\mu(A)\geq \nu(A)$, for any increasing event~$A$.

\begin{proposition}\label{prop:rcm-input}
	Let~$q\geq 1$, $p\in [0,1]$, $q_b\in [1,q]$, and~$\Omega$ be a finite subgraph of~$\bbZ^2$.\\
	
	i) One has~${\sf RC}_{\Omega,q,p}^{1} = {\sf RC}_{\Omega,q,p}^{\mathrm{wired}}$ and ${\sf RC}_{\Omega,q,p}^q = {\sf RC}_{\Omega,q,p}^{\mathrm{free}}$. In particular, as~$\Omega\nearrow \bbZ^2$, the infinite-volume limits~${\sf RC}_{q,p}^1$ and~${\sf RC}_{q,p}^q$ are well-defined and coincide with the wired and free Gibbs states for the random-cluster model.\\

	ii) Let $q'\geq 1$, $q_b'\in [1,q]$, $p'\in [0,1]$ satisfy~$q'\geq q$, $q_b'\geq q_b$ and~$p' \leq p$. Then,
	\[
		{\sf RC}_{\Omega,q',p'}^{q_b'} \preceq {\sf RC}_{\Omega,q,p}^{q_b}.
	\]
	
	iii) Let~$p\neq p_c(q)$ and~$\Omega_k$ be a sequence of domains increasing to~$\bbZ^2$. Then the infinite-volume limit of~${\sf RC}_{\Omega_k,q,p}^{q_b}$ exists, is independent of~$q_b$ and coincides with the unique Gibbs state for the random-cluster model with parameters~$q,p$. We denote it by~${\sf RC}_{q,p}$.\\
	
	iv) The statement of item iii) holds true also if~$q\in [1,4]$ and~$p=p_c(q)$. Also, the following Russo--Seymour--Welsh (RSW) type estimate holds for any vertex~$u\in\Omega$ and some constants~$c,C>0$ independent of~$\Omega$:
		\begin{equation}\label{eq:rsw}
			c\,\log  \dist(u, \bbZ^2\setminus\Omega) < \bbE_{{\sf RC},\Omega,q,p_c}^{q_b}(N_\Omega) < C\,\log  \dist(u, \bbZ^2\setminus\Omega),
		\end{equation}
		where~$N_\Omega$ is the number of connected components surrounding~$u$.\\
		
	v)  Let~$q>4$. Then, under~${\sf RC}_{\Omega,q,p_c}^1$, the size of any dual cluster has exponential tails. In particular, ${\sf RC}_{q,p_c}^1$-a.s. there exists a unique infinite cluster and, under~${\sf RC}_{q,p_c}^q$, the sizes of clusters exhibit exponential decay.
\end{proposition}

\begin{proof}
	$i)$ When~$q_b=q$, all clusters receive the same weight. There is no imposed connectivity on the boundary. Thus, this value of~$q_b$ corresponds to free boundary conditions.

When~$q_b=1$, the number of boundary clusters has no influence on the distribution. This is equivalent to counting all of them as one cluster. Thus, this value of~$q_b$ corresponds to wired boundary conditions.\\

	$ii)$ In the same way as for the standard random-cluster model (\cite[Theorem (3.21)]{Gri06}), the statement follows from the FKG inequality shown in Proposition~\ref{prop:fkg-rc}.\\
	
	$iii)$ By~\cite[Theorem (6.17)]{Gri06} and item i), for any~$q\geq 1$ and~$p\neq \tfrac{\sqrt{q}}{\sqrt{q}+1}$, the measures~${\sf RC}_{\Omega_k,q,p}^{1}$ and~${\sf RC}_{\Omega_k,q,p}^q$ have the same limit, as~$k$ tends to infinity. By item ii), for any~$q_b\in [1,q]$,
	\[
		{\sf RC}_{\Omega_k,q,p}^{1} \succeq {\sf RC}_{\Omega_k,q,p}^{q_b} \succeq {\sf RC}_{\Omega_k,q,p}^q,
	\]
	whence the claim follows.\\
	
	$iv)$ By~\cite{BefDum12}, when~$q\geq 1$, the random-cluster model exhibits a phase transition at $p=p_c(q)$ (see also~\cite{DumRauTas16,DumRauTas17} for alternative proofs). It was shown~\cite{DumSidTas17} that, when~$q\in [1,4]$, the phase transition is of the second order. In particular, this means that the Gibbs measure is unique. In the same way as in item iii), this implies that the limit of~${\sf RC}_{\Omega_k,q,p}^{q_b}$ is independent of~$q_b\in [1,q]$.
	
	The estimate~\eqref{eq:rsw} is a standard consequence of the RSW theory developed in~\cite{DumSidTas17}. We provide only a sketch of the proof. It is enough to consider only~$q_b=q$, since for $q_b = 1$ the proof is completely analogous and then the statement can be extended to any~$q_b\in (1,q)$ by monotonicity shown in Item ii. The following claim allows to bound~$N_\Omega$ from above and below by Bernoulli random variables.	
	
	Without loss of generality, we can assume that~$u=0$.
	
	\begin{claim}\label{cl:one-circuit}
		Let~$\calE_{\mathrm{open}}$ and~$\calE_{\mathrm{closed}}$ be the events that there exists a circuit of open (resp. closed) edges in~$\Omega\setminus\Lambda_{\mathrm{rad}(\Omega)/2}$ that goes around~$0$. Then there exists a constant~$c'>0$ not depending on~$\Omega$ such that we have
		\[
			c'<{\sf RC}_{\Omega,q,p_c}^{1}(\calE_{\mathrm{closed}}) < 1-c' \quad \text{ and }\quad  c'<{\sf RC}_{\Omega,q,p_c}^{q}(\calE_{\mathrm{open}}) < 1-c' .
		\]
	\end{claim}
	
	\begin{proof}
		Inequalities for~$\calE_{\mathrm{closed}}$ and~$\calE_{\mathrm{open}}$ are completely analogous, so we will show only the first one. The lower bound follows readily from the box-crossing property established in Theorems~$2$ and~$3$ of~\cite{DumSidTas17} for~$q\in [1,4]$ under any boundary conditions. The upper bound is also a rather straightforward consequence of the Russo--Seymour--Welsh theory but it is less standard so we prefer to give details below.
		
		Let~$r:= \dist (0, \Omega^c)$. Let~$\calF_1$ be the event that there exists a circuit of open edges contained in~$\Lambda_{r/2}\setminus \Lambda_{r/4}$ and going around~$0$. Let~$\calF_2$ be the event that there exists an open path linking two different points on the boundary of~$\Omega$ and passing through~$\Lambda_{r/4}$. Since~$\calF_1\cap \calF_2\cap \calE_{\mathrm{closed}} = \emptyset$, it is enough to show that there exists~$c'>0$ such that
		\begin{align*}
			\bbP_{{\sf RC},\Omega,q,p_c}^{1}(\calF_1\cap \calF_2) &>c'.
		\end{align*}
		
		Events~$\calF_1$ and~$\calF_2$ are increasing, thus it is enough to show the statement for each of them separately. By the definition of~$r$, there exists a vertex~$z\in \Lambda_{r + 1}$ that belongs to the boundary of~$\Omega$. Then~${\sf RC}_{\Omega,q,p_c}^{1}(\calF_2)$ is greater or equal than the probability to have an open circuit going around~$z$ and crossing~$\Lambda_{r/4}$ under~${\sf RC}_{q,p_c}$ (the unique infinite-volume measure). The latter, as well as~${\sf RC}_{\Omega,q,p_c}^{1}(\calF_1)$, can be bounded below as explained in the beginning of the proof.
	\end{proof}
	
	 To see how the estimate~\eqref{eq:rsw} follows from the claim, we refer the reader to the proof of~\cite[Theorem 1.2 (v)]{GlaMan18}. The only difference is that in our case one has two types of clusters~--- primal and dual. However, since Claim~\ref{cl:one-circuit} takes care of both of them, this does not have any impact on the proof.

	$v)$ This is shown in~\cite{DumGagHar16}.
\end{proof}

\section{FKG for heights, proof of Proposition~\ref{prop:monotone-c-b-heights}}
\label{sec:fkg-heights}

In this section we discuss the positive association properties of the height function measures. These are deduced from straightforward applications of the FKG inequality (as done in~\cite{BenHagMos00} for the uniform case $a=b=c$). We also deduce Proposition~\ref{prop:monotone-c-b-heights} as a corollary.

\begin{proposition}(FKG and monotonicity in boundary conditions for the height function)
\label{prop:fkg-homo}
	Let $a,b>0$ and let $c\ge \max\{a,b\}$. Let~$\calD$ be a domain and let $t$ be a height function. Then ${\sf HF}_{\calD,a,b,c}^{t}$ satisfies the FKG lattice condition. In particular for any increasing functions~$F,G:\bbZ^{F(\calD)}\to\R$, one has
	\[
		{\sf HF}_{\calD,a,b,c}^{t}(F(h)G(h))\geq {\sf HF}_{\calD,a,b,c}^{t}(F(h)){\sf HF}_{\calD,a,b,c}^{t}(G(h)).
	\]
In addition, the measures ${\sf HF}_{\calD,a,b,c}^{t}$ are stochastically increasing in $t$.

Moreover, the FKG lattice condition is satisfied also for the measure ${\sf HF}_{\calD,a,b,c}^{t;c_b}$ with $c\ge \max\{a,b\}, c_b\ge 0$.
If $c_{b}\geq 1$, then ${\sf HF}_{\calD,a,b,c}^{t;c_{b}}$ is stochastically increasing in $t$.
\end{proposition}

\begin{proof}[Proof of Proposition~\ref{prop:fkg-homo}]
By~\cite[Proposition 1]{ForKasGin71}, it is enough to check for any two height functions~$f,g$ on~$\calD$ with the given boundary conditions that the FKG lattice condition is satisfied:
\begin{equation}
\label{eq:fkg-lattice-homo}
{\sf HF}_{\calD,a,b,c}^{t}(f\vee g) \cdot {\sf HF}_{\calD,a,b,c}^{t}(f\wedge g) \geq {\sf HF}_{\calD,a,b,c}^{t}(f) \cdot {\sf HF}_{\calD,a,b,c}^{t}(g),
\end{equation}
where~$f\vee g$ and~$f\wedge g$ denote the point-wise maximum and minimum respectively. We start the proof by the following claim.\\

\emph{Claim.} If~$f(u)>g(u)$ for some~$u\in F(\calD)$, then on all four faces adjacent to~$u$ we have that~$f\vee g$ coincides with~$f$ and~$f\wedge g$ coincides with~$g$.

\begin{proof}
The functions~$f$ and~$g$ must have the same parity at~$u$. Thus, $f(u)>g(u)$ implies that~$f(u)-g(u)\geq 2$. Take any face~$v$ adjacent to~$u$. Since~$f,g$ are height functions, $|f(u)-f(v)| =1$ and~$|g(u)-g(v)|=1$. Thus, $f(v)\geq g(v)$.
\end{proof}

Note that the Claim implies that, on any two adjacent faces~$u,v$ in~$\calD$, each of the functions~$f\vee g$ and~$f\wedge g$ coincides either with~$f$ or with~$g$ (or with both of them). We know that~$|f(u)-f(v)| = 1$ and~$|g(u)-g(v)| = 1$. Thus, the same holds for~$f\vee g$ and~$f\wedge g$, and hence these two functions are also height functions.

It remains to show for any vertex~$z$ of~$\calD$ that its contribution to the LHS of~\eqref{eq:fkg-lattice-homo} is greater or equal than to the RHS of~\eqref{eq:fkg-lattice-homo}. Denote by~$(u_i)_{i=1,2,3,4}$ the four faces of~$\calD$ containing~$z$ (in this cyclic order). If either $f$ or $g$ is larger than the other on all of the $(u_i)$ then this statement is trivial. Otherwise, by the claim, there is a pair of diagonally adjacent faces $u_i, u_j$ such that $f(u_i)>g(u_i)$ and $f(u_j)<g(u_j)$. In this case, $z$ is necessarily a c-type vertex (see Figure~\ref{fig:6v-hom-config}) for both $f\vee g$ and $f\wedge g$ and the statement follows from the assumption that~$c\geq \max\{a,b\}$.

The same argument applies also to the analogue of~\eqref{eq:fkg-lattice-homo} for~${\sf HF}_{\calD,a,b,c}^{t;c_b}$, when $z\in \calD\setminus \partial_{V}\calD$. If $z\in\partial_{V}\calD$, then contribution to both sides of~\eqref{eq:fkg-lattice-homo} is trivially the same.

Monotonicity in $t$ for ${\sf HF}_{\calD,a,b,c}^{t}$ and ${\sf HF}_{\calD,a,b,c}^{c_{b},t}$ follows in the standard manner (by enlarging $\calD$).
\end{proof}

\begin{proof}[Proof of Proposition~\ref{prop:monotone-c-b-heights}]
	Recall that~$\calE_{\sf hf}^{0,1}(\calD)$ is the set of all height functions on~$\calD$ with~$0,1$ boundary conditions. Let~$A$ be any increasing event on~$\calE_{\sf hf}^{0,1}(\calD)$. We need to show that the derivative of~${\sf HF}_{\calD, c}^{0,1;c_b}(A)$ in~$c_b$ is non-negative. Define~$Z$ and~$Z(A)$ by
	\[
		Z:= Z_{{\sf hf},\calD,c}^{0,1;c_b} = \sum_{f\in \calE_{\sf hf}^{0,1}(\calD)} c^{n_{5}'(f)+n_{6}'(f)} c_b^{n_b(f)} \quad \text{and} \quad Z(A) : =\sum_{f\in A} c^{n_{5}'(f)+n_{6}'(f)} c_b^{n_b(f)}.
	\]
	Then~${\sf HF}_{\calD, c}^{0,1;c_b}(A) = Z(A) / Z$, and its derivative in~$c_b$ can be written as
	\[
		\tfrac{\partial}{\partial c_b} {\sf HF}_{\calD, c}^{0,1;c_b}(A) = \tfrac{1}{Z}\cdot \tfrac{\partial}{\partial c_b} Z(A)
		- \tfrac{Z(A)}{Z}\cdot \tfrac{1}{Z} \cdot \tfrac{\partial}{\partial c_b} Z = \tfrac{1}{c_b}\cdot \left[ \bbE(n_b\mathbbm{1}_{A}) -\bbE(n_b) \bbE(\mathbbm{1}_{A})\right],
	\]
	where~$\bbE$ denotes the expectation with respect to~${\sf HF}_{\calD, c}^{0,1;c_b}$.
	
	The random variable~$n_b$ is equal to the number of vertices~$z\in\partial_V\calD$, such that the unique face of~$\calD$ containing~$z$ has height~$1$. Since the height at these faces can be either~$1$ or~$-1$, the variable~$n_b$ is increasing. Thus, the RHS of the last equality is positive by the FKG inequality (Proposition~\ref{prop:fkg-homo}).
	
	This implies the second inequality of the claim and the rest follows, since
	\[
		{\sf HF}_{\calD\setminus\partial_V\calD, c}^{0,-1} = {\sf HF}_{\calD, c}^{0,1;0} \quad \text{and} \quad {\sf HF}_{\calD\setminus\partial_V\calD, c}^{0,1} = {\sf HF}_{\calD, c}^{0,1;\infty}.\qedhere
	\]
\end{proof}

\section{The behavior of the height function}
\label{sec:proofs-heights}

Throughout this section we assume that~$a=b=1$ and~$c\geq 2$. The proofs can be adapted to the general case~$a+b \leq c$ in a straightforward way (see Section~\ref{sec:non-symmetric-case}).

\subsection{Fluctuations}
\label{sec:finite-variance}
In this section we prove Theorem~\ref{thm:var}.
\begin{proof}
Let~$\calD$ be a domain. Define~$\calD^{\mathrm{even}}$ and~$\calD^{\mathrm{odd}}$ as the domains obtained from~$\calD$ by removing from~$\calD$ all its boundary faces that are even (resp. odd). It is easy to see that~$\calD^{\mathrm{even}}$ is an even domain and~$\calD^{\mathrm{odd}}$ is an odd domain (we assume here that~$\calD^{\mathrm{even}}$ and~$\calD^{\mathrm{odd}}$ are connected but this has no effect on the proofs).

Take the unique~$\lambda \geq 0$, such that~$e^{\lambda/2}+e^{-\lambda/2} = c$. The following comparison inequalities follow from Proposition~\ref{prop:monotone-c-b-heights} or can be obtained along the same lines:
\begin{equation}\label{eq:from-c-to-c-b}
	\mathrm{HF}_{\calD,c}^{0,-1}  \preceq \mathrm{HF}_{\calD^{\mathrm{even}},c}^{0,1;e^{\lambda/2}} \preceq \mathrm{HF}_{\calD,c}^{0,1} \preceq \mathrm{HF}_{\calD^{\mathrm{odd}},c}^{0,1;e^{\lambda/2}} \preceq \mathrm{HF}_{\calD,c}^{2,1}.
\end{equation}
Since~$\mathrm{HF}_{\calD,c}^{2,1}$ is the image of~$\mathrm{HF}_{\calD,c}^{0,-1}$ under the bijection~$h(\cdot)\mapsto h(\cdot) + 2$ between~$\calE_{\sf hf}^{0,-1}$ and~$\calE_{\sf hf}^{2,1}$, it is enough to prove Theorem~\ref{thm:var} for measures~${\sf HF}_{\calD^{\mathrm{even}},c}^{0,1;e^{\lambda/2}}$ and~${\sf HF}_{\calD^{\mathrm{odd}},c}^{0,1;e^{\lambda/2}}$.

We prove the statement only for~$\calD^{\mathrm{even}}$ (the case of~$\calD^{\mathrm{odd}}$ is analogous) and, to simplify the notation, we assume that~$\calD$ is an even domain, so that~$\calD^{\mathrm{even}}=\calD$. Clearly,
\[
	\mathrm{Var}_{\sf HF}(h^2(0,0)) = \bbE_{\sf HF} (h^2(0,0)) - \left[\bbE_{\sf HF} (h(0,0))\right]^2,
\]
where the variance and the expectation are with respect to the height-function measure~${\sf HF}_{\calD,c}^{0,1;e^{\lambda/2}}$. Since~$\bbE_{\sf HF} (h(0,0))\in [0,1]$ by~\eqref{eq:from-c-to-c-b}, it is enough to estimate~$\bbE_{\sf HF} (h^2(0,0))$.

Take~$q:= (e^{\lambda} + e^{-\lambda})^2$. For an edge-configuration~$\eta$ on~$\calD^\bullet$, denote the number of primal and dual clusters of~$\eta$ surrounding the origin~$(0,0)$ by~$N(\eta)$. By the coupling stated in Theorem~\ref{thm:coupling}, given~$\eta$ sampled according to~${\sf RC}_{\Omega_{\calD}^\bullet, q, p_c(q)}^1$, the height at the origin is distributed as a simple random walk on~$\bbZ$ starting at~$0$, making~$N(\eta)$ steps and at each step going up or down by~$1$ with probability~$e^{-\lambda}/\sqrt{q}$ (resp.~$e^{\lambda}/\sqrt{q}$). 
Then,
\begin{align*}
	\bbE_{\sf HF} (h^2(u))
	&= \bbE_{\sf RC} (\bbE_{\sf coupling}(h^2(u)\, | \, \eta)) = \bbE_{\sf RC}(\mathrm{Var}_{\sf coupling}(h(u)\, | \, \eta) + \bbE_{\sf coupling}(h(u)\, | \, \eta)^2)\\
	&= \bbE_{\sf RC}(N_u(\eta))\cdot \tfrac{4}{q} + \bbE_{\sf RC}(N_u(\eta)^2)\cdot \tfrac{(e^{\lambda} - e^{-\lambda})^2}{q}.
\end{align*}
If~$c=2$, then~$\lambda = 0$ and~$q=4$, whence the second term cancels out and the first term is treated in Item (iv) of Proposition~\ref{prop:rcm-input}. If~$c>2$, then~$q>4$ and by Item (v) of Proposition~\ref{prop:rcm-input}, the size of any dual cluster in~$\eta$ has expenontial tails. In particular, it means that~${\sf RC}_{\calD^\bullet,q,p_c(q)}^{1}(N_\eta > t) < e^{-\alpha t}$, for a certain constant~$\alpha>0$ depending only on~$q$, whence the statement follows.
\end{proof}

\subsection{Gibbs states}
\label{sec:proof-heights-gibbs}
In this section we prove Theorem~\ref{thm:heights-gibbs}.

The main step in the proof is Proposition~\ref{prop:therm-lim-heights}, which is proven by considering percolation on faces of particular heights. This is somewhat reminiscent to the approach used in~\cite{She05}. However, we emphasize that unlike in~\cite{She05}, here we consider percolation on a suitable {\em triangular} lattice ($\bbT^\bullet$ and~$\bbT^\circ$ defined below). The latter has a benefit of being self-dual and we hope that this approach will turn out to be useful in the future research.

\begin{proposition}\label{prop:therm-lim-heights}
	Fix~$c>2$ and~$n\in \bbZ$. Let~$\calD_k$ be an increasing sequence of domains exhausting~$\bbZ^2$. Then, the sequence of measures~${\sf HF}_{\calD_k,c}^{n,n+1}$ converges to a Gibbs state~${\sf HF}_c^{n,n+1}$, which is extremal, invariant under parity-preserving translations, and satisfies the exponential tail decay~\eqref{eq:exponential tail decay}.
\end{proposition}

The first step is to prove a similar statement under modified boundary conditions.

\begin{lemma}\label{lem:therm-lim-heights-c-b}
	Let~$c>2$. Take~$\lambda > 0$ such that~$e^{\lambda/2} + e^{-\lambda/2} = c$. Let~$\calD_k$ be an increasing sequence of even domains exhausting~$\bbZ^2$. Then, the sequence~${\sf HF}_{\calD_k,c}^{0,1; e^{\lambda/2}}$ converges to a Gibbs state~${\sf HF}_{\mathrm{even},c}^{0,1; e^{\lambda/2}}$, which is extremal and invariant under parity-preserving translations. Moreover, ${\sf HF}_{\mathrm{even},c}^{0,1; e^{\lambda/2}}$-a.s. faces of height~$0$ contain a unique infinite cluster (in the diagonal connectivity), and diameters of connected components of non-zero heights have exponential tails.
	
	Similarly, for odd domains, the limit~${\sf HF}_{\mathrm{odd},c}^{0,1; e^{\lambda/2}}$ exhibits a unique infinite cluster of height~$1$, and 
\end{lemma}

\begin{proof}
	We prove the statement only for even domains, since the case of odd domains is completely analogous. As was already mentioned above, the centres of even faces of~$\bbZ^2$ form another square lattice that we denote by~$(\bbZ^2)^\bullet$.
	
	Take~$q:=(c^2-2)^2$. Let~${\sf RC}_{q}^1$ be the wired infinite-volume random-cluster measure on~$(\bbZ^2)^\bullet$ with parameters~$q$ and~$p_c(q):=\tfrac{\sqrt{q}}{\sqrt{q}+1}$. By Item  (v) of Proposition~\ref{prop:rcm-input}, ${\sf RC}_{q}^1$-a.s. there is a unique infinite cluster and dual clusters are exponentially small, that is there exists~$\alpha>0$ such that, for any~$k\in\bbN$ and any $u^{*}\in (\bbZ^{2})^{\circ}$,
	\begin{equation}\label{eq:exp-decay}
		{\sf RC}_{q}^1( \text{ the dual cluster of $u^{*}$ has size } >k ) <e^{-\alpha k}.
	\end{equation}
	Define a random height function~$h$ in the following way: sample~$\eta\in\{0,1\}^{E((\bbZ^2)^\bullet)}$ according to~${\sf RC}_{q}^1$, set~$h$ to be~$0$ on the unique infinite cluster of~$\eta$, then sample~$h$ in the holes of this cluster according to the coupling~\eqref{eq:coupling-hf-form} in Theorem~\ref{thm:coupling} for~$c_b = e^{\lambda/2}$ and~$q_b=1$. Denote by~${\sf HF}_{\mathrm{even},c}^{0,1; e^{\lambda/2}}$ the distribution of~$h$.
	
	Note that measure~${\sf HF}_{\mathrm{even},c}^{0,1; e^{\lambda/2}}$ is well-defined since the values of~$h$ in different holes of~$\eta$ are independent (conditioned on~$\eta$) and the size of each hole has exponential tails. Properties of~${\sf HF}_{\mathrm{even},c}^{0,1; e^{\lambda/2}}$ (extremality, invariance under parity-preserving translations and existence of an infinite cluster of height~$0$) follow from the corresponding properties of~${\sf RC}_{q}^1$ (extremality, invariance under all translations and existence of an infinite cluster). It remains to show that~$\mathrm{HF}_{\calD_k,c}^{0,1; e^{\lambda/2}}$ tends to~$\mathrm{HF}_{\mathrm{even},c}^{0,1; e^{\lambda/2}}$.
	
	Fix any~$\varepsilon>0$ and take~$n$ big enough so that~$e^{-\alpha n}<\varepsilon$.
	
	Recall the definition of $\calD^\bullet$ for a domain~$\calD$ given in Section~\ref{sec:coupling}. Since~${\sf RC}_{\calD_k^\bullet,q}^{1}$ tends~${\sf RC}_{q}^1$, there exists~$K>0$ such that~$\calD_K\supset \Lambda_{4n}$ and, for all~$k>K$, the total variation distance between the restriction of~${\sf RC}_{q}^1$ and~${\sf RC}_{\calD_k^\bullet,q}^{1}$ to~$\Lambda_{2n}^\bullet$ is less than~$\varepsilon$. Fix any~$k>K$. Define~$C_{n,k}$ to be the exterior-most circuit of open edges in~$\eta$ contained in $\Lambda_{2n}^\bullet$ that goes around~$\Lambda_n^\bullet$ (take~$C_{n,k}:=\emptyset$ if no such circuit exists).
	
	Using the estimate on the total variation distance we get that
\begin{equation}\label{eq:inf-vol-tot-var}
	\sum_{C}|{\sf RC}_{q}^1(C_{n,k} = C) - {\sf RC}_{q,\calD_k^\bullet}^1(C_{n,k} = C)| < \varepsilon,
\end{equation}
where the sum is taken over all circuits~$C\subset\Lambda_{2n}^\bullet\setminus \Lambda_{n}^\bullet$.

Also note that by~\eqref{eq:exp-decay} we have
\begin{align}
	\label{eq:inf-vol-circuit-cond1}
	{\sf RC}_{q,\calD_k^\bullet}^1(C_{n,k} \subset \eta, \, C_{n,k} \xleftrightarrow{\eta}\partial\calD_k^\bullet) > 1 - 2e^{-\alpha n} &> 1-2\varepsilon, \\
	\label{eq:inf-vol-circuit-cond2}
	{\sf RC}_{q}^1(C_{n,k} \subset \eta, \, C_{n,k} \xleftrightarrow{\eta}\infty)  > 1 - 2e^{-\alpha n} &> 1-2\varepsilon.
\end{align}
	If~$\eta$ satisfies conditions in~\eqref{eq:inf-vol-circuit-cond1} and~\eqref{eq:inf-vol-circuit-cond2}, then given~$C_{n,k}$ the heights on~$\Lambda_n$ sampled according to~${\sf HF}_{\mathrm{even}, c}^{0,1;e^{\lambda/2}}$ and~${\sf HF}_{\calD_k,c}^{0,1;e^{\lambda/2}}$ have the same law. Putting this together with the estimates~\eqref{eq:inf-vol-tot-var}, \eqref{eq:inf-vol-circuit-cond1} and~\eqref{eq:inf-vol-circuit-cond2} we get that the total variation distance between restrictions of~${\sf HF}_{\mathrm{even}, c}^{0,1;e^{\lambda/2}}$ and~${\sf HF}_{\calD_k,c}^{0,1;e^{\lambda/2}}$ to~$\Lambda_n$ is less than~$5\varepsilon$, whence convergence follows.
\end{proof}

As we will show below, measures~${\sf HF}_{\mathrm{even},c}^{0,1; e^{\lambda/2}}$ and~${\sf HF}_{\mathrm{odd},c}^{0,1; e^{\lambda/2}}$ are in fact equal. The next step in the proof of Proposition~\ref{prop:therm-lim-heights} is to establish certain percolation statements for the faces of height~$1$ under~${\sf HF}_{\mathrm{even},c}^{0,1; e^{\lambda/2}}$ and for the faces of height~$0$ under~${\sf HF}_{\mathrm{odd},c}^{0,1; e^{\lambda/2}}$.

\begin{definition}\label{def-t-circuits}
	Denote by~$\bbT^\bullet$ (resp.~$\bbT^\circ$) the graph on the even (resp. odd) faces of~$\bbZ^2$, where a face~$(i,j)$ is linked by an edge to the faces~$(i\pm 1, j\pm 1)$ and~$(i\pm 2, j)$. It is easy to see that both~$\bbT^\bullet$ and~$\bbT^\circ$ are isomorphic to the standard triangular lattice. As usual, a circuit in $\bbT^\bullet$ (resp.~$\bbT^\circ$) is a sequence of vertices $v_1,\ldots, v_{n-1}, v_n=v_1$ with $v_i$ adjacent to $v_{i+1}$ in $\bbT^\bullet$ (resp.~$\bbT^\circ$) and $v_1,\ldots, v_{n-1}$ distinct. The circuit is said to surround a vertex $v\notin\{v_1,\ldots, v_{n-1}\}$ if the collection of edges $(v_i, v_{i+1})$, viewed as straight line segments in the ambient $\R^2$, disconnect $v$ from infinity.
\end{definition}

For~$K\in \bbN$ and~$(i,j)\in \bbZ^2$, denote by~$\Lambda_K(i,j)$ the ball of radius~$K$ around~$(i,j)$:
\[
	\Lambda_K(i,j):= \{(u,v)\in \bbZ^2 \colon |(u-i)\pm (v-j)| \leq K-1\}.
\]
Let~$\calC_K^0$ be the exterior-most $\bbT^{\bullet}$-circuit of height~$0$ in~$\Lambda_K(1,0)$ that surrounds the face~$(1,0)$ (take~$\calC_K^0:= \emptyset$ if there is no such circuit). Similarly, let~$\calC_K^1$ be the exterior-most $\bbT^{\circ}$-circuit of height~$1$ in~$\Lambda_K(0,0)$ that surrounds the face~$(0,0)$; see Figure~\ref{fig:height-function-flip}.

\begin{lemma}\label{lem:T-circuits}
	Let~$c>2$. Then, the distribution of~$\calC_K^0$ under~${\sf HF}_{\mathrm{odd},c}^{0,1; e^{\lambda/2}}$ coincides with the distribution of~$\calC_K^1$ under~${\sf HF}_{\mathrm{even},c}^{0,1; e^{\lambda/2}}$ shifted by~$1$ to the right. In addition, for any~$N \in \bbN$,
	\begin{align}
		{\sf HF}_{\mathrm{even},c}^{0,1; e^{\lambda/2}}(\calC_K^1 \text{ surrounds } \Lambda_N(0,0)) &\xrightarrow[K\to \infty]{} 1 \label{eq:many-T-circuits}.
	\end{align}
\end{lemma}

\begin{proof}
	For~$k\in\bbN$, let~$f_k$ be distributed according to~${\sf HF}_{\Lambda_{2k}(0,0),c}^{0,1; e^{\lambda/2}}$. Define~$g_k$ by:
	\begin{equation}
		g_k(i,j) = 1- f_k(i-1,j) \label{eq:operation-shift-one-minus}
	\end{equation}
		
	It is straightforward that~$g_k$ is supported on height functions and the image of~$0,1$ boundary conditions under this mapping are again~$0,1$ boundary conditions, though on a slightly different domain. More precisely, the domain is $1 + \Lambda_{2k}(0,0)$, which is the same as~$\Lambda_{2k}(1,0)$. In conclusion, height function~$g_k$ is distributed according to~${\sf HF}_{\Lambda_{2k}(1,0),c}^{0,1; e^{\lambda/2}}$.
	
	The domains~$\Lambda_{2k}(1,0)$ form a sequence of odd domains, whence by Lemma~\ref{lem:therm-lim-heights-c-b} the weak limit of~${\sf HF}_{\Lambda_{2k}(1,0),c}^{0,1; e^{\lambda/2}}$ is~${\sf HF}_{\mathrm{odd},c}^{0,1; e^{\lambda/2}}$. Also, by Lemma~\ref{lem:therm-lim-heights-c-b} the weak limit of~${\sf HF}_{\Lambda_{2k}(0,0),c}^{0,1; e^{\lambda/2}}$ is~${\sf HF}_{\mathrm{even},c}^{0,1; e^{\lambda/2}}$. Thus, measure~${\sf HF}_{\mathrm{odd},c}^{0,1; e^{\lambda/2}}$ is obtained from~${\sf HF}_{\mathrm{even},c}^{0,1; e^{\lambda/2}}$ by the operation described in~\eqref{eq:operation-shift-one-minus}. Finally, it is easy to see that under the operation~\eqref{eq:operation-shift-one-minus}, circuit~$\calC_K^1$ is mapped into circuit~$\calC_K^0$.
	
	It remains to show~\eqref{eq:many-T-circuits}, which is equivalent to showing that~${\sf HF}_{\mathrm{even},c}^{0,1; e^{\lambda/2}}$-a.s. there are infinitely many disjoint~$\bbT^\circ$-circuits of height~$1$ surrounding the origin. 
	Lemma~\ref{lem:therm-lim-heights-c-b} implies that, for all $n$, measure~${\sf HF}_{\mathrm{even},c}^{n,n+1; e^{\lambda/2}}$ is extremal. Thus, the ${\sf HF}_{\mathrm{even},c}^{0,1; e^{\lambda/2}}$-probability to have infinitely many disjoint~$\bbT^\circ$-circuits of height~$1$ surrounding the origin is either zero or one. Assume that this probability is zero. Then, by the self-duality of~$\bbT^\circ$ and the extremality of~${\sf HF}_{\mathrm{even},c}^{0,1; e^{\lambda/2}}$, at least one of the following occurs:
	\begin{align}
		{\sf HF}_{\mathrm{even},c}^{0,1; e^{\lambda/2}}(F_{\leq -1}^\circ \text{ contains an infinite } \bbT^\circ\text{-cluster}) &= 1, \label{eq:inf-T-cluster-minus-1} \quad \text{or}\\
		{\sf HF}_{\mathrm{even},c}^{0,1; e^{\lambda/2}}(F_{\geq 3}^\circ \text{ contains an infinite } \bbT^\circ\text{-cluster}) &= 1, \label{eq:inf-T-cluster-3}
	\end{align}
	where~$F_{\leq n}^\circ$ (resp. $F_{\geq n}^\circ$) denotes the set of odd faces of height smaller (resp. greater) or equal to~$n$. The bijection mapping~$h\to 2-h$ implies that~\eqref{eq:inf-T-cluster-3} is equivalent to
	\[
		{\sf HF}_{\mathrm{even},c}^{2,1; e^{\lambda/2}}(F_{\leq -1}^\circ \text{ contains an infinite } \bbT^\circ\text{-cluster}) = 1,
	\]
	which in turn implies~\eqref{eq:inf-T-cluster-minus-1}, since the height function measures are monotone with respect to boundary conditions (Proposition \ref{prop:fkg-homo}). Thus, we can assume~\eqref{eq:inf-T-cluster-minus-1} occurs. 
	
	Through the bijection mapping~$h\to -h$, \eqref{eq:inf-T-cluster-minus-1} implies
	\[
		{\sf HF}_{\mathrm{even},c}^{0,-1; e^{\lambda/2}}(F_{\geq 1}^\circ \text{ contains an infinite } \bbT^\circ\text{-cluster}) = 1,
	\]
	whence, by the monotonicity in boundary conditions, also
	\[
		{\sf HF}_{\mathrm{even},c}^{0,1; e^{\lambda/2}}(F_{\geq 1}^\circ \text{ contains an infinite } \bbT^\circ\text{-cluster}) = 1.
	\]
	This means that, ${\sf HF}_{\mathrm{even},c}^{0,1; e^{\lambda/2}}$-a.s. each of $F_{\leq -1}^\circ$ and $F_{\geq 1}^\circ$ contains an infinite $\bbT^{\circ}$-cluster. These clusters are unique. Indeed, the measure ${\sf HF}_{\mathrm{even},c}^{0,1; e^{\lambda/2}}$ is invariant under parity-preserving translations by Lemma~\ref{lem:therm-lim-heights-c-b}.
	Then, the general argument of Burton and Keane~\cite{BurKea89} can be applied in the same way as in~\cite[Theorem 4.9]{ChaPelSheTas18} to yield uniqueness of the infinite clusters.
			
	Since~${\sf HF}_{\mathrm{even},c}^{0,1; e^{\lambda/2}}$ is positively associated and~$F_{\geq 1}^\circ$ is increasing, the marginal of~${\sf HF}_{\mathrm{even},c}^{0,1; e^{\lambda/2}}$ on~$F_{\geq 1}^\circ$ is also positively associated. Also, the compliment to~$F_{\geq 1}^\circ$ in~$\bbT^\circ$ is the set~$F_{\leq -1}^\circ$. In conclusion, the distribution of~$F_{\geq 1}^\circ$ is invariant under all translations of~$\bbT^\circ$, is positively associated, and in both~$F_{\geq 1}^\circ$ and its compliment there is almost surely a unique infinite cluster. This contradicts~\cite[Theorem 1.5]{DumRauTas17}. We note that the latter theorem was established for pairs of dual edge-configurations but it adapts in a straightforward manner to the setting of pairs of site-configurations on the triangular lattice used here.
\end{proof}

We are now ready to finish the proof of Proposition~\ref{prop:therm-lim-heights}.

\begin{figure}
	\begin{center}
		\includegraphics[scale=1.2,page=3]{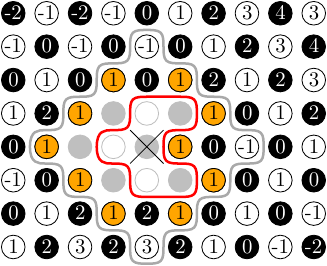} \quad \quad
		\includegraphics[scale=1.2,page=4]{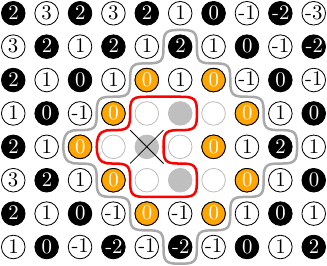}
		\caption{The origin is marked with a cross. \emph{Left:} Height function is sampled from the outside up to $\bbT^\circ$-circuit~$C_4^1$ (orange)~--- the exterior-most $\bbT^\circ$-circuit of height~$1$ contained inside~$\Lambda_4(0,0)$ (bounded by the gray contour). \emph{Right:}  Height function is obtained from the one on the left applying a composition of the shift by~$1$ to the right and operation~$h\mapsto 1-h$. This maps~$C_4^1$ to~$C_4^0$ (orange)~--- the exterior-most $\bbT^\bullet$-circuit of height~$0$ contained inside~$\Lambda_4 (1,0)$ (in gray). \emph{Both:} The red contour surrounds domain~$\calD_C$~--- the unique connected component in~$\bbZ^2\setminus (C_4^1\cup C_4^0)$ that contains the origin. The left and the right height functions are sampled on~$\calD_C$ independently. Then, the FKG inequality implies that, on~$\calD_C$, the left function stochastically dominates the right function.}
		\label{fig:height-function-flip}
	\end{center}
\end{figure}

\begin{proof}[Proof of Proposition~\ref{prop:therm-lim-heights}]
	Without loss of generality, one can assume that~$n=0$. The main step of the proof is to show that~${\sf HF}_{\mathrm{even},c}^{0,1; e^{\lambda/2}} = {\sf HF}_{\mathrm{odd},c}^{0,1; e^{\lambda/2}}$.
	
	Take any~$N\in \bbN$ and~$\varepsilon>0$. By Lemma~\ref{lem:T-circuits}, there exists~$K\in\bbN$ such that
	\begin{equation}\label{eq:t-circuit-large}
		{\sf HF}_{\mathrm{even},c}^{0,1; e^{\lambda/2}}(\calC_K^1 \text{ surrounds } \Lambda_{N}(0,0)) > 1-\varepsilon.
	\end{equation}
	Fix such $K$. 
	Consider any $\bbT^\circ$-circuit~$C^1$ which surrounds~$\Lambda_{N}(0,0)$ and for which ${\sf HF}_{\mathrm{even},c}^{0,1; e^{\lambda/2}}(\calC_K^1  = \calC^1) >0$.
	Define~$C^0$ as the $\bbT^\bullet$-circuit obtained from~$C^1$ after a shift by~$1$ to the right. 
	Let~$\calD_K$ be the set of faces in the connected component of~$\bbZ^2\setminus(C^1\cup C^0)$ containing~$\Lambda_{N}(0,0)$.  Conditioned on the event that~$\calC_K^1  = C^1$, the heights on the circuit~$C^0$ are at least~$0$, whence
	\begin{equation}\label{eq:t-circuits-1}
		{\sf HF}_{\mathrm{even},c}^{0,1; e^{\lambda/2}}( \cdot \, | \, \calC_K^1 = C^1)_{|\Lambda_{N}(0,0)} \succeq {\sf HF}_{\calD_K,c}^{0,1}( \cdot )_{|\Lambda_{N}(0,0)}.
	\end{equation}
	Similarly, conditioned on the event that~$\calC_K^0  = \calC^0$, the heights at the faces of the circuit~$\calC^1$ are at most~$1$, whence
	\begin{equation}\label{eq:t-circuits-0}
		{\sf HF}_{\mathrm{odd},c}^{0,1; e^{\lambda/2}}( \cdot \, | \, \calC_K^0 = C^0)_{|\Lambda_{N}(0,0)} \preceq {\sf HF}_{\calD_K,c}^{0,1}( \cdot )_{|\Lambda_{N}(0,0)}.
	\end{equation}
	We now couple~$h^{1}\sim {\sf HF}_{\mathrm{even},c}^{0,1; e^{\lambda/2}}$ and~$h^{0}\sim {\sf HF}_{\mathrm{odd},c}^{0,1; e^{\lambda/2}}$ in the following way; see Figure~\ref{fig:height-function-flip}. First, we explore $\calC_{K}^{1}(h^{1})$ from outside, for all~$(x,y)\in \bbZ^2$ in the exterior of~$\calC_{K}^{1}(h^{1})$, we set
	\[
		h^0((x,y))= 1-h^1((x-1,y)).
	\]
	This, in particular, implies that~$\calC_{K}^{0}(h^{0})$ is obtained from~$\calC_{K}^{1}(h^{1})$ by shift by one to the right.
	Also, conditioned on $\calC_{K}^{1}(h^{1}) = C^{1}$, the distribution of $h^{1}$ and $h^{0}$ on $\calD_{K}$ is given by ${\sf HF}_{\mathrm{even},c}^{0,1; e^{\lambda/2}}( \cdot \, | \, \calC_K^1 = C^1)$ and ${\sf HF}_{\mathrm{odd},c}^{0,1; e^{\lambda/2}}( \cdot \, | \, \calC_K^0 = C^0)$, respectively.
	By \eqref{eq:t-circuits-1} and \eqref{eq:t-circuits-0}, these two measures can be coupled in such a way that, for all $(x,y)\in \calD_{K}$,
	\[
		h^0((x,y)) \leq h^1((x,y)).
	\]
	Summing over all circuits~$C^1$, by \eqref{eq:t-circuit-large}, we get that $h^{0}\leq h^{1}$ on the box $\Lambda_{N}(0,0)$ with probability at least $1-\eps$. Sending~$\varepsilon$ to zero and taking arbitrary~$N$ gives
	\[
		{\sf HF}_{\mathrm{even},c}^{0,1; e^{\lambda/2}} \succeq {\sf HF}_{\mathrm{odd},c}^{0,1; e^{\lambda/2}}.
	\]
	
	The opposite inequality follows from monotonicity in boundary conditions. Indeed, let~$\calD_k$ be any increasing sequence of even domains exhausting~$\bbZ^2$. Define the sequence~$\calD_k'$ of odd domains obtained from~$\calD_k$ after a shift by one to the right. 
	By Proposition~\ref{prop:monotone-c-b-heights},
	\begin{align*}
		{\sf HF}_{\calD_k,c}^{0,1; e^{\lambda/2}} &\preceq {\sf HF}_{\calD_k,c}^{0,1; c}  = {\sf HF}_{\calD_k,c}^{0,1},\\
		{\sf HF}_{\calD_k',c}^{0,1; e^{\lambda/2}} &\succeq {\sf HF}_{\calD_k',c}^{0,1; c}  = {\sf HF}_{\calD_k',c}^{0,1}.
	\end{align*}
	Also, by the monotonicity in boundary conditions (Proposition \ref{prop:fkg-homo}),
	\[
		{\sf HF}_{\calD_k,c}^{0,1} \preceq {\sf HF}_{\calD_k\cap\calD_k',c}^{0,1} \preceq {\sf HF}_{\calD_k',c}^{0,1}.
	\]
	Altogether, this gives
	\[
		{\sf HF}_{\calD_k,c}^{0,1; e^{\lambda/2}} \preceq {\sf HF}_{\calD_k,c}^{0,1} \preceq  {\sf HF}_{\calD_k',c}^{0,1} \preceq {\sf HF}_{\calD_k',c}^{0,1; e^{\lambda/2}}.
	\]	
	By taking limits on both sides with we get the desired in equality.	
	In conclusion,
	\[
		{\sf HF}_{\mathrm{even},c}^{0,1; e^{\lambda/2}} = {\sf HF}_{\mathrm{odd},c}^{0,1; e^{\lambda/2}}.
	\]
	We denote this measure by~${\sf HF}_c^{0,1}$.
	Then, by~\eqref{eq:from-c-to-c-b}, for any sequence of domains $\calD_{k}''$, the limit of ${\sf HF}_{\calD_{k}'',c}^{0,1; e^{\lambda/2}}$ exists and is given by ${\sf HF}_c^{0,1}$. Properties of~${\sf HF}_c^{0,1}$ follow immediately from Lemma~\ref{lem:therm-lim-heights-c-b}.
\end{proof}

The following corollary is a straightforward consequence of the convergence proven in Lemma~\ref{lem:therm-lim-heights-c-b}, the equality~${\sf HF}_{\mathrm{even},c}^{0,1; e^{\lambda/2}} = {\sf HF}_{\mathrm{odd},c}^{0,1; e^{\lambda/2}}$ proven in Proposition~\ref{prop:therm-lim-heights} and the monotonicity in~$c_b$ established in Proposition~\ref{prop:monotone-c-b-heights}.

\begin{corollary}\label{cor:therm-lim-heights-c-b}
	Let~$c>2$ and~$\calD_k$ be an increasing sequence of even domains that exhausts~$\bbZ^2$. Then, ${\sf HF}_{\calD_k,c}^{0,1;c_b}$ converges to~${\sf HF}_{c}^{0,1}$ when~$c_b\geq e^{\lambda/2}$, and~${\sf HF}_{\calD_k,c}^{0,1;c_b}$ converges to~${\sf HF}_{c}^{0,-1}$ when~$c_b\leq e^{-\lambda/2}$.
\end{corollary}

We are now ready to finish the proof of Theorem~\ref{thm:heights-gibbs}.

\begin{proof}[Proof of Theorem~\ref{thm:heights-gibbs}]
	Fix~$c>2$. Existence, invariance under parity-preserving translations and extremality of Gibbs measures~$\{{\sf HF}_c^{n,n+1}\}_{n\in\bbZ}$ is established in Proposition~\ref{prop:therm-lim-heights}.
	Also, this proposition gives existence of a unique infinite cluster on heights $n$ and $n+1$ with exponentially small holes whose diameters satisfy \eqref{eq:exponential tail decay}.
	Invariance under transformation $h(i,j) \mapsto 1-h(i-1,j)$ and stochastic ordering in $n$ follow readily from the construction of $\{{\sf HF}_c^{n,n+1}\}$ as the infinite-volume limit under $(n,n+1)$-boundary conditions. 
	
	Now take any extremal Gibbs state~${\sf HF}$ for the height-function measure with parameter~$c$ that is invariant under parity-preserving translations. It remains to show that~${\sf HF} = {\sf HF}_c^{n,n+1}$, for some~$n\in\bbZ$.
	
	For any~$n\in\bbZ$, define the following events:
	\begin{align*}
		\calF_{2n}&:=\{\exists \text{ infinitely many disjoint } \bbT^\bullet\text{-circuits of height } 2n \text{ surrounding } (0,0)\},\\
		\calF_{2n+1}&:=\{\exists \text{ infinitely many  disjoint } \bbT^\circ\text{-circuits of height } 2n+1 \text{ surrounding } (0,0)\}.
	\end{align*}
	Since~${\sf HF}$ is an extremal measure, we have that~${\sf HF}(\calF_n) \in \{0,1\}$.
	
	Assume there exist~$m, n\in\bbZ$, such that~$m< n$ and~${\sf HF}(\calF_m) = {\sf HF}(\calF_n) = 1$. All faces that are adjacent to a face of height~$m$ have height at most~$m+1$. Thus~${\sf HF}$-a.s., there exist infinitely many circuits (in the usual $\bbZ^2$ connectivity, even and odd faces are alternating) of height at most~$m+1$. By positive association (Proposition~\ref{prop:fkg-homo}), this implies that~${\sf HF} \preceq {\sf HF}_c^{m,m+1}$. Similarly, ${\sf HF} \succeq {\sf HF}_c^{n-1,n}$. Then necessarily~$n=m+1$ and~${\sf HF} = {\sf HF}_c^{m,m+1}$, which would finish the argument.
	
	Thus, we can assume that either ${\sf HF}(\calF_m) = 0$ for all even values of $m$ or ${\sf HF}(\calF_m) = 0$ for all odd values of $m$.
	Without loss of generality, below we assume that for any~$n\in\bbZ$,
	\begin{equation}
		{\sf HF}(\calF_{2n}) = 0. \label{eq:thm-gibbs-no-T-circuits}
	\end{equation}
	Define the following events:
	\begin{align*}
		A_{2n}&:= \{\exists \text{ infinite } \bbT^\bullet \text{-cluster of height } \geq 2n\}, \\
		B_{2n}&:= \{\exists \text{ infinite } \bbT^\bullet \text{-cluster of height } \leq 2n\}.
	\end{align*}
	By~\eqref{eq:thm-gibbs-no-T-circuits}, for any~$n\in\bbZ$,
	\[
		{\sf HF}(A_{2n+2} \cup B_{2n-2}) = 1.
	\]
	By extremality of~${\sf HF}$, each of the events~$A_{2n+2}$ and~$B_{2n-2}$ occurs with probability~$0$ or~$1$. Thus, for any~$n\in \bbZ$
	\begin{equation}
		\text{either }{\sf HF} (A_{2n+2}) = 1 \text{ or } {\sf HF}(B_{2n-2}) = 1.\label{eq:thm-gibbs-alternative}
	\end{equation}
	Without loss of generality, assume that, for~$n=0$, the first alternative in~\eqref{eq:thm-gibbs-alternative} occurs, that is~${\sf HF}(A_2) = 1$ (the case~${\sf HF}(B_{-2}) = 1$ is completely analogous). Then, by~\cite[Theorem 1.5]{DumRauTas17}, applied here in the same way as in the proof of Lemma~\ref{lem:T-circuits}, 
	\[
		{\sf HF}(B_0) = 0.
	\]
	Applying~\eqref{eq:thm-gibbs-alternative} for~$n=1$, we obtain that~${\sf HF} (A_4) = 1$. Continuing in the same way, we obtain~${\sf HF}(B_{2n}) = 0$, for all~$n\in \bbN$ --- and hence, by monotonicity, for all~$n\in \bbZ$. This implies that,  for all~$n\in \bbZ$,
	\[
		{\sf HF} \succeq {\sf HF}_c^{2n-1,2n}.
	\]
	This leads to a contradiction, since then, for any~$L\in\bbZ$,
	\[
		{\sf HF} (h(0,0) < L) \leq {\sf HF}_c^{2n-1,2n} (h(0,0) < L)  \xrightarrow[L\to \infty]{} 0.
	\]
	This finishes the proof in the case~$c>2$. 
	
	Now assume there exists an extremal translation-invariant Gibbs state~$\mathsf{HF}$ for the height-function measure with parameter~$c=2$. What we showed above implies that, for some $n\in \bbZ$,
	\[
		\mathsf{HF}(\calF_{2n})=1.
	\]
	Then, by positive association, for every~$K,N$,
	\[
		\mathsf{HF}(h(0)\geq N) \geq \inf_{\calD \supset \Lambda_K} \mathsf{HF}_{\calD,2}^{2n,2n-1}(h(0)\geq N),
	\]
	where the infimum is taken over all domains~$\calD$ containing~$\Lambda_K$. It follows from the proof of Theorem~\ref{thm:var} that the limit in~$K$ of the right-hand side of the last inequality is bounded below by a positive constant that is independent of~$N$. Contradiction.
\end{proof}

\subsection{Consequences for the FK model with modified boundary-cluster weight}
In this section we prove Theorem~\ref{thm:q-b}.
\begin{proof}
	Item (ii) of Theorem~\ref{thm:q-b} follows from Item (iv) of Proposition~\ref{prop:rcm-input} which is in turn implied by known results for the random-cluster model. It remains to show Item (i).
	
	Let~$q>4$. By Theorem~\ref{thm:coupling}, the measure~${\sf RC}_{\calD_k^\bullet,q,p_c(q)}^{e^{-\lambda}q}$ can be coupled with~${\sf HF}_{\calD_k,c}^{0,1}$, and~${\sf RC}_{\calD_k^\bullet,q,p_c(q)}^1$ can be coupled with~${\sf HF}_{\calD_k,c}^{0,1;e^{\lambda/2}}$. As shown in the proof of Proposition~\ref{prop:therm-lim-heights}, sequences of measures~${\sf HF}_{\calD_k,c}^{0,1}$ and~${\sf HF}_{\calD_k,c}^{0,1;e^{\lambda/2}}$ have the same limit. Since the coupling rule~\eqref{eq:coupling-edge-form} is completely local, this implies that sequences of measures~${\sf RC}_{\calD_k^\bullet,q,p_c(q)}^{e^{-\lambda}q}$ and~${\sf RC}_{\calD_k^\bullet,q,p_c(q)}^1$ also have the same limit. By monotonicity in~$q_b$ established in Proposition~\ref{prop:rcm-input}, we get that, for all~$q_b\in [1,e^{-\lambda}\sqrt{q}]$, the limit of~${\sf RC}_{\calD_k^\bullet,q,p_c(q)}^{q_{b}}$ also exists and is equal to~${\sf RC}_{q,p_c(q)}^1$.
\end{proof}

\section{The behavior of the spin representation and the six-vertex model}
\label{sec:proofs-spins}

Throughout this section we assume that~$a=b=1$ (see Section~\ref{sec:non-symmetric-case} for~$a\neq b$).

The main tools in the proof are the FK--Ising-type representation~$\xi$ introduced in Section~\ref{sec:fk-ising} and the height representation of the six-vertex model.

Let~$\calD$ be a domain on~$\bbZ^2$. Recall a pair of dual graphs~$\calD^\bullet$ and~$\calD^\circ$ defined on even and odd faces of~$\calD$ (Section~\ref{sec:coupling}). Recall that, for a spin configuration~$\sigma$ on~$\bbZ^2$, the restrictions of~$\sigma$ to even and odd faces are denoted by~$\sigma^\bullet$ and~$\sigma^\circ$.

Define~$\theta(\sigma^\circ) \subset E(\calD^\circ)$, such that~$uv\in \theta(\sigma^\circ)$ iff~$\sigma^\circ(u) \neq \sigma^\circ(v)$. Similarly, define~$\theta(\sigma^\bullet) \subset E(\calD^\bullet)$. Define~$\omega(\sigma^\bullet)\subset E(\calD^\circ)$, such that~$e\in \omega(\sigma^\bullet)$ iff~$e^*\in \theta(\sigma^\bullet)$.

\subsection{FK--Ising-type representation: Item 1 of Theorem~\ref{thm:spins-gibbs}, Corollary~\ref{cor:6v-gibbs}}
\label{sec:fk-ising}

The FK--Ising representation of the six-vertex model that we discuss in this section is directly related (see the remark after Proposition~\ref{prop:coupling-AT-6V}) to the random-cluster representation of the Ashkin--Teller model introduced by Pfister and Velenik~\cite{PfiVel97}  and is used in Section~\ref{sec:proof-ashkin-teller} to describe the coupling between the two models. For the Ashkin--Teller model, this representation allowed to derive the Lebowitz inequality~\cite{ChaSht00}. Here we choose to define this representation in terms of the six-vertex model in order to avoid confusion between different models and restrict the appearance of the Ashkin--Teller model to Section~\ref{sec:proof-ashkin-teller}. We refer the reader to the works of Ray and Spinka~\cite{RaySpi19b} and Lis~\cite{Lis19} where the representations on the primal and the dual lattices are considered simultaneously.

Given~$\sigma^\bullet$ and~$\sigma^\circ$, define a random edge-configuration $\xi\in\{0,1\}^{E(\calD^\circ)}$: if~$e\in\theta(\sigma^\circ)$, then~$\xi(e) = 0$; if~$e\in\omega(\sigma^\bullet)$, then~$\xi(e) = 1$; if~$e\in E(\calD^\circ)\setminus(\omega(\sigma^\bullet)\cup \theta(\sigma^\circ))$, then~$\xi(e) = 1$ with probability~$\tfrac{c-1}{c}$ and~$\xi(e) = 0$ with probability~$\tfrac{1}{c}$; see Figure~\ref{fig:fk-ising}. We call~$\xi$ the \emph{FK--Ising representation of the six-vertex model}. The measure~${\sf FKIs}_{\calD^\circ,c}$ is defined as the distribution of~$\xi$ when~$\sigma$ is distributed according to~${\sf Spin}_{\calD,c}^{++}$.

It is easy to see that~$\sigma$ has constant value on clusters of~$\xi$ and~$\xi^*$ (in the terminology of Section~\ref{sec:coupling}, $\sigma$ and~$\xi$ are compatible). In the next lemma, we state further properties of this coupling.

\begin{lemma}\label{lem:sigma-given-eta-o}
	$i)$ The joint law of~$\sigma$ and~$\xi$ can be written as:
	\begin{equation}\label{eq:fk-ising-coupling}
		(\sigma, \xi) \propto  (c-1)^{|\xi| - |\omega(\sigma^\bullet)|}\mathbbm{1}_{\sigma^\circ\perp \xi}\mathbbm{1}_{\sigma^\bullet\perp \xi}.
	\end{equation}
	$ii)$ The measure~${\sf FKIs}_{\calD^\circ,c}$ can be written in the following way:
	\begin{equation}\label{eq:marginal-fk-is}
		{\sf FKIs}_{\calD^\circ,c} (\xi) = \tfrac{1}{2Z} \cdot (c-1)^{|\xi|}2^{k(\xi^1)}\sum_{\sigma^\bullet:\,  \omega(\sigma^\bullet)\subset \xi } {\left(\tfrac{1}{c-1}\right)^{ |\omega(\sigma^\bullet)|}}.
	\end{equation}
	$iii)$ Let~$\xi$ be distributed according to~${\sf FKIs}_{\calD^\circ,c}$. Assign plus to all boundary clusters of~$\xi$ and plus or minus with probability~$1/2$ independently to all other clusters of~$\xi$. Then, the obtained spin configuration has the same distribution as the marginal distribution of~${\sf Spin}_{\calD,c}^{++}$ on~$\sigma^\circ$.
\end{lemma}

\begin{proof}
	$i)$ By the definition of~$\xi$, one has
	\[
		(\sigma, \xi) \propto {\sf Spin}_{\calD,c}^{++}(\sigma) \cdot \mathbbm{1}_{\omega(\sigma^\bullet) \subset \xi } \cdot \mathbbm{1}_{\theta(\sigma^\circ) \cap \xi = \emptyset}\cdot \left(\tfrac{c-1}{c}\right)^{|\xi \setminus \omega(\sigma^\bullet)|} \cdot \left(\tfrac{1}{c}\right)^{|E(\calD^\circ)\setminus (\xi \cup\theta(\sigma^\circ))|}.
	\]
	The indicators in the formula are equivalent to saying that~$\sigma$ and~$\xi$ are compatible. Since~${\sf Spin}_{\calD,c}^{++}(\sigma)$ is proportional to~$c^{|E(\calD^\circ)|- |\omega(\sigma^\bullet)| - |\theta(\sigma^\circ)|}$, the formula above turns into~\eqref{eq:fk-ising-coupling}.
	
	$ii)$ To show~\eqref{eq:marginal-fk-is}, we sum~\eqref{eq:fk-ising-coupling} over all~$\sigma$ compatible with~$\xi$ in two steps. First, we sum over all~$\sigma^\circ$ that are pluses at all boundary clusters of~$\xi$ and have a constant value at all other clusters of~$\xi$~--- this results in multiplication by~$2^{k(\xi^1)-1}$. Then, we sum over all~$\sigma^\bullet$ that are pluses at all boundary clusters of~$\xi^*$ and have a constant value at all other clusters of~$\xi^*$. This gives~\eqref{eq:marginal-fk-is}.
	
	$iii)$ By~\eqref{eq:fk-ising-coupling}, the marginal distribution of~$\sigma^\circ$ according to~${\sf Spin}_{\calD,c}^{++}$ conditioned on~$\xi$ is uniform on all spin configurations on~$\calD^\circ$ that are pluses at all boundary clusters of~$\xi$ and have a constant value at all other clusters of~$\xi$. This proves the claim.
\end{proof}

\begin{figure}
	\begin{center}
		\includegraphics[width=\textwidth]{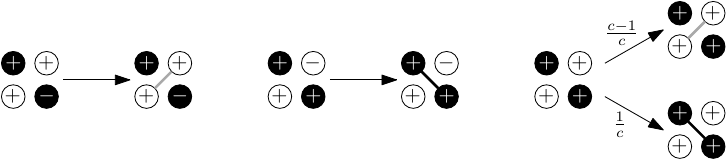}
		\caption{The FK--Ising representation~$\xi$ of the six-vertex model: if the spins of~$\sigma^\bullet$ disagree, then the edge is open in~$\xi$ (gray);  if the spins of~$\sigma^\circ$ disagree, then the edge is closed in~$\xi$ (the dual edge is shown in black); if the spins of~$\sigma^\bullet$ agree and the spins of~$\sigma^\circ$ agree, then the edge is open w.p.~$\tfrac{c-1}{c}$ and closed w.p.~$\tfrac{1}{c}$.}
		\label{fig:fk-ising}
	\end{center}
\end{figure}

In order to discuss the properties of the FK--Ising representation in the infinite volume for~$c>2$, we first state a straightforward consequence of Theorem~\ref{thm:heights-gibbs} for the thermodynamic limits of spin measures.

\begin{lemma}\label{lem:limits-c-greater-than-2}
	Let~$c>2$. Then, the measure~${\sf Spin}_{\calD,c}^{++}$, as~$\calD\nearrow \bbZ^2$, converges to a limiting measure~${\sf Spin}_c^{++}$ that is extremal, translation-invariant and exhibits a unique infinite cluster of pluses at even faces and a unique infinite cluster of pluses at odd faces; clusters of minuses (even and odd together) are exponentially small. 
	
	Similarly, for the limits~${\sf Spin}_c^{+-}$, ${\sf Spin}_c^{-+}$, ${\sf Spin}_c^{--}$ under the corresponding boundary conditions. In particular, the four limiting measures are distinct.
\end{lemma}

\begin{proof}
	Follows immediately from the properties of measures~${\sf HF}_c^{n,n+1}$ stated in Theorem~\ref{thm:heights-gibbs} and the mapping between heights and spins (Section \ref{sec:spin-rep}).
\end{proof}

The next corollary follows readily from Lemma~\ref{lem:sigma-given-eta-o} and extremality of~${\sf Spin}_c^{++}$.

\begin{corollary}\label{cor:sigma-given-fk-ising-infinite}
	Let $c>2$. The weak limit of~${\sf FKIs}_{\calD^\circ,c}$ as~$\calD\nearrow \bbZ^2$ exists. Denote it by~${\sf FKIs}_c$. Then, ${\sf FKIs}_c$ can be coupled with~${\sf Spin}_c^{++}$ in such a way that the joint law is supported on pairs of compatible configurations~$(\sigma,\xi)$ and satisfies the following properties:
	\begin{itemize}
		\item given~$\sigma$, if $\sigma(u)=\sigma(v)$ and~$\sigma(u^*)=\sigma(v^*)$ for a pair of dual edges $uv$ of $(\bbZ^{2})^{\circ}$ and $u^{*}v^{*}$ o f$(\bbZ^{2})^{\bullet}$, then~$\xi(uv)=1$ with probability~$(c-1)/c$;
		\item given~$\xi$, if~$\calC$ is a finite cluster of~$\xi$, the value of~$\sigma$ on~$\calC$ is constant plus or minus with probability~$1/2$, and the value of~$\sigma$ is fixed to be plus on infinite clusters of~$\xi$.
	\end{itemize}
	Moreover, the measure ${\sf FKIs}_c$ is extremal, translation-invariant and there exists a unique infinite cluster in~$\xi$ almost surely. Lastly, the following relation holds for any two odd faces~$u,v$ of~$\bbZ^2$:
	\begin{align}
	 	&{\sf Spin}_c^{++}(\sigma^\circ(u)) = {\sf FKIs}_c(u\xleftrightarrow{\xi} \infty)> 0,\label{eq:bias of spin}\\
		&{\sf Spin}_c^{++} (\sigma^\circ(u)\sigma^\circ(v)) = {\sf FKIs}_c(u\xleftrightarrow{\xi} v).\label{eq:correlations-sigma-as-xi}
	\end{align}
\end{corollary}

\begin{proof}
  Existence of the limiting measure ${\sf FKIs}_c$, existence of the coupling and the itemized statements follow immediately from the finite-volume coupling of Lemma~\ref{lem:sigma-given-eta-o} and the limiting properties of Lemma~\ref{lem:limits-c-greater-than-2}. The coupling satisfies that, conditioned on $\sigma$, the edges of $\xi$ are independent. Thus, extremality and translation invariance of ${\sf FKIs}_c$ follow from the corresponding properties of ${\sf Spin}_c^{++}$.

  We proceed to show the uniqueness of the infinite cluster. It is easy to see that~${\sf FKIs}_c$ satisfies the finite energy property. Thus, the argument of Burton and Keane~\cite{BurKea89} can be applied to show that ${\sf FKIs}_c$-a.s. there is at most one infinite cluster in~$\xi$. Assume in order to get a contradiction that there is no infinite cluster. Then the spin measure ${\sf Spin}_c^{++}$ is invariant to flipping the sign on all odd faces (by the finite-volume coupling of Lemma~\ref{lem:sigma-given-eta-o} and the limiting procedure). However, ${\sf Spin}_c^{++}$ is the push-forward of~${\sf HF}_{c}^{0,1}$, under the modulo $4$ mapping. The above contradicts Theorem~\ref{thm:heights-gibbs}, which shows that samples from ${\sf HF}_{c}^{0,1}$ have an infinite cluster in diagonal connectivity of height $1$ but do not have such an infinite cluster of any other odd height.

  Formula~\eqref{eq:bias of spin} follows from the finite-volume coupling of Lemma~\ref{lem:sigma-given-eta-o} and the existence of the infinite cluster. To deduce the formula~\eqref{eq:correlations-sigma-as-xi}, observe first that in finite volume, by Lemma~\ref{lem:sigma-given-eta-o}, the expectation of $\sigma^\circ(u)\sigma^\circ(v)$ is the probability that $u$ and $v$ are either connected by $\xi$ or are both in boundary clusters of $\xi$. Formula~\eqref{eq:correlations-sigma-as-xi} thus follows in infinite volume by the uniqueness of the infinite cluster.
\end{proof}

\begin{proof}[Proof of Item 1 of Theorem~\ref{thm:spins-gibbs}.] 
	Convergence to the infinite-volume limit, as well as existence and uniqueness of the infinite clusters, follow from Lem\-ma~\ref{lem:limits-c-greater-than-2}. The property in the second bullet follows from~\eqref{eq:bias of spin} and the fact that the spins at faces of the same parity are non-negatively correlated by Theorem~\ref{thm:fkg} (the theorem is stated in finite volume but extends to the infinite volume limiting measure).
\end{proof}

We point out that an alternative proof of the non-negative correlation of the spins at faces of the same parity could be obtained using positive association of ${\sf FKIs}_{\calD^\circ,c}$ (Proposition~\ref{prop:fkg-fk-ising} below), together with~\eqref{eq:bias of spin} and~\eqref{eq:correlations-sigma-as-xi}. However, such a proof would not apply in the entire regime $a+b<c$ since we establish Proposition~\ref{prop:fkg-fk-ising} only in the restricted range $c\ge\max\{2a,2b\}$. This is in contrast to the FKG property of the spins (Theorem~\ref{thm:fkg}) which applies whenever $c\ge\max\{a,b\}$.

We proceed to prove Corollary~\ref{cor:6v-gibbs}. We will view each configuration~$\vec{\omega}$ of the six-vertex model as an element in~$\{1,-1\}^{E(\bbZ^2)}$, where~$\vec{\omega}(e) = 1$ iff when following~$e$ in the direction that it is assigned in~$\vec{\omega}$, the even face bordering~$e$ is on the left.

\begin{proof}[Proof of Corollary~\ref{cor:6v-gibbs}]
	Let $c \geq 2$. Define~${\sf SixV}_{c}^{\circlearrowleft}$ as the push-forward of~${\sf Spin}_c^{++}$ to the six-vertex configurations under the mapping depicted on Figure~\ref{fig:6v-hom-config}. Similarly, ${\sf SixV}_{c}^{\circlearrowright}$ is defined as the push-forward of ${\sf Spin}_c^{+-}$. That measures ${\sf SixV}_{c}^{\circlearrowleft}$ and ${\sf SixV}_{c}^{\circlearrowright}$ are extremal and invariant under parity-preserving translations follows directly from the same properties of~${\sf Spin}_c^{++}$ and ${\sf Spin}_c^{+-}$ established in Theorem~\ref{thm:spins-gibbs}.
		
	It is also a straightforward consequence that ${\sf SixV}_{c}^{\circlearrowleft}$  and ${\sf SixV}_{c}^{\circlearrowright}$ are the only extremal flat Gibbs states.
	Indeed, let ${\sf SixV}_{c}$ be an extremal flat Gibbs state on arrow configurations.
	Associate with ${\sf SixV}_{c}$ a spin measure ${\sf Spin}_{c}$ in the following way: spin at $(0,0)$ is $+$ or $-$ with probability 1/2 and all other spins are obtained as the pushforward of ${\sf SixV}_{c}$.
	Then ${\sf Spin}_{c}$ is a Gibbs state for the spin representation of the six-vertex model.
	Since, ${\sf SixV}_{c}$ is flat, we know that ${\sf Spin}_{c}$ exhibits infinitely many circuits of alternating vertical and horizontal edges surrounded by even faces of constant spin and odd faces of constant spin.
	Then ${\sf Spin}_{c}$ is a mixture of the measures ${\sf Spin}_{c}^{++}$,  ${\sf Spin}_{c}^{+-}$,  ${\sf Spin}_{c}^{-+}$,  ${\sf Spin}_{c}^{--}$.
	The push-forward of these measures to arrow configurations gives ${\sf SixV}_{c}^{\circlearrowleft}$ and ${\sf SixV}_{c}^{\circlearrowright}$, which yields the statement.
	
	In particular, when $c=2$,  by Theorem \ref{thm:spins-gibbs}, one has  ${\sf Spin}_{c}^{++} = {\sf Spin}_{c}^{+-}$ --- and hence ${\sf SixV}_{c}^{\circlearrowleft} = {\sf SixV}_{c}^{\circlearrowright}$ is the unique flat Gibbs state.
	This finishes the proof for $c=2$.
	
	Below we assume that $c>2$.
	We will now prove that ${\sf SixV}_{c}^{\circlearrowleft}(A(e))> 1/2$, for every edge~$e\in E(\bbZ^2)$. 
	This is equivalent to showing that, for any pair of adjacent faces~$u$ (even) and $u^{*}$ (odd),
	\begin{equation}\label{eq:arrow-up-bias}
		{\sf Spin}_c^{++}(\sigma(u)\sigma(u^*)) > 0.
	\end{equation}
	
	Consider an even face~$v$ adjacent to~$u^*$ and the odd face~$v^*$ adjacent to~$u$ and~$v$.
	
	Take~$\lambda  > 0$ such that~$e^{\lambda/2}+e^{-\lambda/2} = c$. It follows from Lemma~\ref{lem:therm-lim-heights-c-b}, that the weak limit of~${\sf Spin}_{\calD,c}^{++;e^{\lambda/2}}$ over even domains exists and is equal to~${\sf Spin}_c^{++}$. 
	
	Take~$q:=(e^\lambda + e^{-\lambda})^2$, $p_c(q):=\tfrac{\sqrt{q}}{\sqrt{q}+1}$. By Corollary~\ref{cor:coupling-spins}, when~$\calD$ is an even domain, measures~${\sf Spin}_{\calD,c}^{++;e^{\lambda/2}}$ and~${\sf RC}_{\calD^\bullet,q, p_c(q)}^1$ can be coupled in such a way that the joint law is given by~\eqref{eq:coupling-edge-form-spin} and thus:
	\begin{itemize}
		\item if~$\sigma(u) \neq \sigma(v)$, then edge~$uv$ is closed;
		\item if~$\sigma(u^*) \neq \sigma(v^*)$, then edge~$uv$ is open;
		\item if~$\sigma(u) = \sigma(v) = \sigma(u^*) = \sigma(v^*)$, then edge~$uv$ is open with probability~$e^{\lambda/2}/c$ and closed with probability~$e^{-\lambda/2}/c$;
		\item if~$\sigma(u) = \sigma(v) \neq \sigma(u^*) = \sigma(v^*)$, then edge~$uv$ is open with probability~$e^{-\lambda/2}/c$ and closed with probability~$e^{\lambda/2}/c$.
	\end{itemize}
	Taking a thermodynamic limit, one obtains
	\begin{align}
		{\sf RC}_{q, p_c(q)}^1 (uv \text{ open})
		&= {\sf Spin}_c^{++}(\sigma(u^*)\neq \sigma(v^*)) \nonumber
		+ \tfrac{1}{c}e^{\lambda/2}\cdot {\sf Spin}_c^{++}(\sigma(u) \sigma(u^*) = \sigma(v) \sigma(v^*) =1) \\
		&+ \tfrac{1}{c}e^{-\lambda/2}\cdot {\sf Spin}_c^{++}(\sigma(u) \sigma(u^*) = \sigma(v) \sigma(v^*) = -1).\label{eq:6v-order-1}
	\end{align}
	Similarly, when~$\calD$ is an odd domain, measures~${\sf Spin}_{\calD,c}^{++;e^{\lambda/2}}$ are~${\sf RC}_{\calD^\circ,q, p_c(q)}^1$ are coupled and
	\begin{align}
		{\sf RC}_{q, p_c(q)}^{*,1} (u^*v^* \text{ closed})
		&= {\sf Spin}_c^{++}(\sigma(u^*)\neq \sigma(v^*)) \nonumber
		+ \tfrac{1}{c}e^{-\lambda/2}\cdot {\sf Spin}_c^{++}(\sigma(u) \sigma(u^*) = \sigma(v) \sigma(v^*) =1) \\
		&+ \tfrac{1}{c}e^{\lambda/2}\cdot {\sf Spin}_c^{++}(\sigma(u) \sigma(u^*) = \sigma(v) \sigma(v^*) = -1),\label{eq:6v-order-2}
	\end{align}
	where~${\sf RC}_{q, p_c(q)}^{*,1}$ stands for the wired random-cluster measure on odd faces. By duality,
	\[
		{\sf RC}_{q, p_c(q)}^{*,1} (u^*v^* \text{ closed}) = {\sf RC}_{q, p_c(q)}^0 (uv \text{ open}).
	\]
	Subtracting~\eqref{eq:6v-order-2} from~\eqref{eq:6v-order-1}, we obtain
	\begin{align}
		{\sf RC}_{q, p_c(q)}^1 (uv \text{ open}) 
		&- {\sf RC}_{q, p_c(q)}^0 (uv \text{ open})   \nonumber \\
		&= \tfrac{1}{c}(e^{\lambda/2} - e^{-\lambda/2}){\sf Spin}_c^{++}(\sigma(u) \sigma(u^*) + \sigma(v) \sigma(v^*)).\label{eq:difference-wried-free-bc}
	\end{align}
	It was proven in~\cite{DumGagHar16} that ${\sf RC}_{q, p_c(q)}^1 \neq {\sf RC}_{q, p_c(q)}^0$ when~$q>4$ (see Proposition~\ref{prop:rcm-input}). Positive association then implies that that the LHS of \eqref{eq:difference-wried-free-bc} is strictly positive. Also, by translation invariance of~${\sf Spin}_c^{++}$, we have~${\sf Spin}_c^{++}(\sigma(u) \sigma(u^*))= {\sf Spin}_c^{++}( \sigma(v) \sigma(v^*))$. Substituting this in \eqref{eq:difference-wried-free-bc} proves \ref{eq:arrow-up-bias}, since~$\lambda > 0$.
	
	To finish the proof for $c>2$, it remains to prove the inequality \eqref{eq:arrows-correlation}.
	We use that ${\sf SixV}_{c}^{\circlearrowleft}(A(e)) = {\sf SixV}_{c}^{\circlearrowleft}(A(f))$ by invariance, approximate the inequality by its finite volume analogue on even domains and rewrite it in terms of spins: 
	\begin{equation}\label{eq:arrows-correlation-spins}
		{\sf Spin}_{\calD,c}^{++;e^{\lambda/2}}(\sigma(u)\sigma(u^*) = 1 \, | \, \sigma(v)\sigma(v^*) = 1 ) \stackrel{?}{\geq}  {\sf Spin}_c^{++;e^{\lambda/2}}(\sigma(u)\sigma(u^*) = 1).
	\end{equation}
	When ${\sf Spin}_{\calD,c}^{++;e^{\lambda/2}}$ is coupled to ${\sf RC}_{\calD^\bullet,q, p_c(q)}^1$, the spins are assigned to different clusters independently, whence the inequality follows readily.	
\end{proof}

The next two propositions describe properties of the FK--Ising representation of the six-vertex model required for the proof of Theorem~\ref{thm:ashkin-teller} for the Ashkin--Teller model (Section~\ref{sec:proof-ashkin-teller}). We want to emphasize that, unlike other proofs in this section, the proofs of the propositions below do not extend to the full parameter regime $c\ge a+b$ but rather only to its subset $c\ge\max\{2a,2b\}$ (we do not know whether positive association holds whenever~$c\ge a+b$, however, the FKG lattice condition may fail when either~$2a > c $ or~$2b > c$).

\begin{proposition}[Positive association of~$\xi$]\label{prop:fkg-fk-ising}
	Let~$c\geq 2$ and~$\calD$ be a domain. Then, the measure~${\sf FKIs}_{\calD^\circ,c}$ satisfies the FKG lattice condition and hence is positively associated.
\end{proposition}

\begin{proof}
	By~\cite[Thm~4.11]{Gri06}, it is enough that for any~$e,f\in E(\calD^\circ)$ and any edge-configuration~$\xi$ where both~$e$ and~$f$ are closed, the following holds:
	\[
		{\sf FKIs}_{\calD^\circ,c}^{++}(\xi^{e,f}) {\sf FKIs}_{\calD^\circ,c}^{++}(\xi_{e,f}) \geq {\sf FKIs}_{\calD^\circ,c}^{++}(\xi^e_f) {\sf FKIs}_{\calD^\circ,c}^{++}(\xi^f_e),
	\]
	where configurations~$\xi^{ef}$, $\xi_{ef}$, $\xi^e_f$, and~$\xi^f_e$ coincide on~$E(\calD^\circ)\setminus \{e,f\}$; edge~$e$ is open in~$\xi^{e,f}$ and~$\xi^e_f$ and closed in~$\xi_{ef}$ and~$\xi^f_e$; edge~$f$ is open in~$\xi^{ef}$ and~$\xi^f_e$ and closed in~$\xi_{e,f}$ and~$\xi^e_f$.
	
	Substitute~\eqref{eq:marginal-fk-is} in this inequality. The term~$(c-1)^{|\xi|}2^{k(\xi^1)}$ is just the usual FK--Ising measure (random-cluster with $q=2$) and it is standard (and follows from Proposition \ref{prop:fkg-rc}) that it satisfies the FKG lattice condition. Thus, it is enough to show that
	\[
		\frac{\sum_{\sigma^\bullet:\,  \omega(\sigma^\bullet)\subset \xi_{ef} } {\left(\tfrac{1}{c-1}\right)^{ |\omega(\sigma^\bullet)|}}}
		{\sum_{\sigma^\bullet:\,  \omega(\sigma^\bullet)\subset \xi^{ef}} {\left(\tfrac{1}{c-1}\right)^{ |\omega(\sigma^\bullet)|}} }
		\stackrel{?}{\geq}
		\frac{\sum_{\sigma^\bullet:\,  \omega(\sigma^\bullet)\subset\xi^e_f  } {\left(\tfrac{1}{c-1}\right)^{ |\omega(\sigma^\bullet)|}}}
		{\sum_{\sigma^\bullet:\,  \omega(\sigma^\bullet)\subset \xi^{ef}  } {\left(\tfrac{1}{c-1}\right)^{ |\omega(\sigma^\bullet)|}} }
		\cdot
		\frac{ \sum_{\sigma^\bullet:\,  \omega(\sigma^\bullet)\subset\xi^f_e  } {\left(\tfrac{1}{c-1}\right)^{ |\omega(\sigma^\bullet)|}}}
		{\sum_{\sigma^\bullet:\,  \omega(\sigma^\bullet)\subset \xi^{ef}  } {\left(\tfrac{1}{c-1}\right)^{ |\omega(\sigma^\bullet)|}} }.
	\]
	Let~$\bbP$ denote the Ising measure with parameter~$\beta = \tfrac12 \log (c-1)$ and plus boundary conditions on the graph obtained from~$\calD^\bullet$ after identifying all vertices belonging to the same connected component of~$(\xi^{e,f})^*$. Also, denote the endpoints of~$e^*$ (resp.~$f^*$) by~$u_e$ and $v_e$ (resp.~$u_f$ and~$v_f$). Then, the ratios can be written in terms of~$\bbP$ as follows:
	\[
		\bbP(\sigma^\bullet(u_e) = \sigma^\bullet(v_e),\sigma^\bullet(u_f) = \sigma^\bullet(v_f)) \stackrel{?}{\geq} \bbP(\sigma^\bullet(u_e) = \sigma^\bullet(v_e))\cdot  \bbP(\sigma^\bullet(u_f) = \sigma^\bullet(v_f)).
	\]
	The last inequality follows from the second Griffiths' inequality in the Ising model~\cite[Theorem 2]{Gri67} (see also~\cite[Theorem 2.3]{PelSpi17}) and thus holds whenever~$\beta \geq  0$, that is~$c \geq 2$.
\end{proof}

\begin{proposition}\label{prop:exp-decay-eta-o}
	Let~$c>2$. Then~${\sf FKIs}_c$-a.s. there exists an infinite cluster, while dual clusters are exponentially small~--- there exist~$M,\alpha>0$ such that, for any even faces~$u,v$,
	\[
		{\sf FKIs}_c(u\xleftrightarrow{\xi^*} v) \leq M e^{-\alpha |u-v|}.
	\]
\end{proposition}

\begin{proof}
	Since~${\sf FKIs}_c$ is obtained as a limit of finite-volume measures, Lemma~\ref{prop:fkg-fk-ising} implies that it satisfies positive association inequality for any increasing events of finite support. By Corollary~\ref{cor:sigma-given-fk-ising-infinite} measure~${\sf FKIs}_c$ is extremal. Hence, approximating any increasing events with increasing events of finite support and using the martingale convergence theorem, we get that~${\sf FKIs}_c$ is positive associated.

	Corollary~\ref{cor:sigma-given-fk-ising-infinite} showed that $\xi$ exhibits a unique infinite cluster almost surely. Denote it by~$\calC_\infty$. Applying the argument of Burton and Keane~\cite{BurKea89} to~$\xi^*$, we obtain that either ${\sf FKIs}_c$-a.s. there is no infinite cluster in~$\xi^*$ or ${\sf FKIs}_c$-a.s. there exists a unique infinite cluster in~$\xi^*$. The latter option, together with the translational invariance, positive association and existence of~$\calC_\infty$, contradicts~\cite[Theorem 1.5]{DumRauTas17}. Hence, ${\sf FKIs}_c$ exhibits no infinite dual cluster.
	
	Recall that an $\ell^1$ box of size~$n$ centered at the origin is denoted by~$\Lambda_n$. We now show
	\begin{equation}\label{eq:AT-exp-decay-interior-radius}
		{\sf FKIs}_c(\calC_\infty\cap \Lambda_n = \emptyset) \leq M'e^{-\alpha'n},
	\end{equation}
	for some~$M',\alpha'>0$. Indeed, assume the event in \eqref{eq:AT-exp-decay-interior-radius}, so that all clusters of~$\xi$ that intersect~$\Lambda_n$ are finite. Then, by Corollary~\ref{cor:sigma-given-fk-ising-infinite}, the distribution of~$\sigma^\circ$ on~$\Lambda_n$ conditioned on this realisation of~$\xi$ is invariant under a global sign flip. By duality then, two opposite sides of~$\Lambda_n$ are linked by a $\bbT^\circ$-crossing  (recall definitions above Lemma~\ref{lem:T-circuits}) of minus spins with probability at least 1/2.
	However, averaging over all $\xi$, we get ${\sf Spin}_{c}^{++}$, where this event has exponentially small probability --- by the pushforward of \eqref{eq:exponential tail decay} in Theorem \ref{thm:heights-gibbs} to spins.
	This gives \eqref{eq:AT-exp-decay-interior-radius}.
	
It remains to show that~\eqref{eq:AT-exp-decay-interior-radius} implies exponential decay of connectivities in~$\xi^*$. For any~$u\in \partial \Lambda_n$, let~$A_u$ be an event that~$u$ is connected to the origin by a path in~$\xi^*\cap \Lambda_n$. Define, $A_u':= A_u + u$ and~$A_u'':= A_u + 2u$. Combining the crossings and using the positive association (Proposition \ref{prop:fkg-fk-ising}) and translational invariance of~${\sf FKIs}_c$, we obtain
	\[
		{\sf FKIs}_c(0 \xleftrightarrow{\xi^*} 3u \text{ in }\Lambda_n\cup\Lambda_n(u)\cup\Lambda_n(2u)) \geq {\sf FKIs}_c(A_u)^3.
	\]
	Note that the crossing described above does not intersect~$\Lambda_n + \tfrac32\cdot n(1+i)$. Since~${\sf FKIs}_c$ is invariant under rotation by~$\tfrac{\pi}{2}$, we obtain bounds on existence of crossings~$3u \leftrightarrow 3u+3iu$, $3u + 3iu \leftrightarrow 3iu$, and~$3iu \leftrightarrow 0$ none of which intersects~$\Lambda_n + \tfrac32\cdot n(1+i)$. Combining these crossings and using the FKG inequality once again, we get
	\[
		{\sf FKIs}_c(\exists \text{ circuit } \Gamma \subset \xi^* \text{ surrounding } \Lambda_n + \tfrac32\cdot n(1+i)) \geq {\sf FKIs}_c(A_u)^{12}.
	\]
	The LHS of the above inequality is exponentially small by~\eqref{eq:AT-exp-decay-interior-radius}, then so is~${\sf FKIs}_c(A_u)$ and the proof is finished.
\end{proof}

\subsection{FKG for spins: proof of Theorem~\ref{thm:fkg}, Proposition~\ref{prop:monotone-c-b-spins}}
\label{sec:proof-fkg-spin-rep}

Theorem~\ref{thm:fkg} will follow from the next proposition.

\begin{proposition}[FKG lattice condition]\label{prop:FKG}
	Let~$\calD$ be a domain, $c\geq 1$ and~$\tau\in\{1,-1\}^{F(\bbZ^2)}$ be such that~$\tau$ is a plus at all odd faces. Then, for every~$\sigma_e,\sigma_e' \in\{1,-1\}^{F^\bullet(\bbZ^2)}$,
	\begin{equation}\label{eq:FKG}
		{\sf Spin}_{\calD,c}^\tau[\sigma_1^\bullet\vee {\sigma_2^\bullet}] \cdot {\sf Spin}_{\calD,c}^\tau[\sigma_1^\bullet\wedge \sigma_2^\bullet]\ge{\sf Spin}_{\calD,c}^\tau[\sigma_1^\bullet] \cdot {\sf Spin}_{\calD,c}^\tau[\sigma_2^\bullet].
	\end{equation}
\end{proposition}

	Recall the pair of dual graphs~$\calD^\bullet$ and~$\calD^\circ$ introduced in Section~\ref{sec:coupling}. The proof goes through the FK--Ising and the dual FK--Ising representations on these graphs~--- we use that each of the terms in~\eqref{eq:FKG2} can be interpreted as the partition function of an FK model with free boundary conditions on the set of all pluses of~$\sigma^\bullet$ times the same on the set of all minuses of~$\sigma^\bullet$, and we derive the claim from the FKG inequality for the FK model applied separately to these partition functions.
	
	Define~$\calE(\sigma^\bullet)$ as the set of all spin configurations~$\sigma^\circ$ on~$\calD^\circ$, for which~${\sf Spin}_{\calD,c}^\tau(\sigma^\bullet, \sigma^\circ)>0$ (in other words,~$\omega(\sigma^\bullet)\cap \theta(\sigma^\circ) = \emptyset$ and~$\sigma^\circ$ is a plus at all corners of~$\calD$).

\begin{proof}[Proof of Proposition~\ref{prop:FKG}]
	By~\cite[Theorem (2.22)]{Gri06}, it is enough to show~\eqref{eq:FKG} for any two configurations which differ in exactly two places i.e., that for any $\sigma^\bullet\in\{-1,1\}^{F^\bullet(\bbZ^2)}$, that coincides with~$\tau$ outside of~$\calD$, and for any $u, v\in F^\bullet(\calD)$,
	\begin{align}\label{eq:FKG2}
		{\sf Spin}_{\calD,c}^\tau[\sigma^\bullet_{++}] \cdot {\sf Spin}_{\calD,c}^\tau[\sigma^\bullet_{-\, -}] \ge {\sf Spin}_{\calD,c}^\tau[\sigma^\bullet_{+-}] \cdot {\sf Spin}_{\calD,c}^\tau[\sigma^\bullet_{-+}] ,
	\end{align}
	where $\sigma^\bullet_{\varepsilon\varepsilon'}$ is the configuration coinciding with $\sigma^\bullet$ except (possibly) at $u$ and $v$, and such that $\sigma^\bullet_{\varepsilon\varepsilon'}(u)= \varepsilon$ and $\sigma^\bullet_{\varepsilon\varepsilon'}(v)= \varepsilon'$.
	
	By definition, the marginal of~${\sf Spin}_{\calD,c}^\tau$ on even spins can be written as:
	\begin{align}
		{\sf Spin}_{\calD,c}^\tau(\sigma^\bullet)
		&= \frac{1}{Z}\sum_{\sigma^\circ\in \calE(\sigma^\bullet)} c^{| E(\calD^\circ)\setminus(\omega(\sigma^\bullet)\cup \theta(\sigma^\circ))|}
		= \frac{1}{Z} \sum_{\sigma^\circ\in \calE(\sigma^\bullet)}\sum_{\substack{\xi\subset E(\calD^\circ)\setminus \theta(\sigma^\circ)\\ \xi\supset \omega(\sigma^\bullet)}} (c-1)^{|\xi \setminus \omega(\sigma^\bullet)|} \label{eq:fkg-spins-marginal-1}\\
		&= \frac{1}{Z} \sum_{\xi \supset \omega(\sigma^\bullet)}\sum_{\substack{\sigma^\circ\in \calE(\sigma^\bullet):\\
		\theta(\sigma^\circ) \cap \xi = \emptyset}}(c-1)^{|\xi| - |\omega(\sigma^\bullet)|} = \frac{1}{Z} \left(\tfrac{1}{c-1}\right)^{\omega(\sigma^\bullet)}\sum_{\xi \supset \omega(\sigma^\bullet)}(c-1)^{|\xi|} 2^{k(\xi^1)-1}, \nonumber		
	\end{align}
	where the 1st equality holds since edges not belonging to~$\omega(\sigma^\circ)\cup \theta(\sigma^\circ)$ are exactly those contributing~$c$ to the probability of a configuration; to obtain the 2nd equality, we develop~$c=(c-1)+1$ (and the resulting~$\xi$ is in fact the FK--Ising representation defined in Section~\ref{sec:fk-ising}); the 3rd equality is obtained by exchanging the order of summation; the 4th equality uses the fact that every non-boundary cluster of~$\xi$ receives in~$\sigma^\circ$ a constant spin plus or minus independently. Configuration $\xi^1$ is obtained from~$\xi$ by \emph{wiring} (i.e., merging) all vertices corresponding to corners of~$\calD$.
	
	The sum on the RHS of \eqref{eq:fkg-spins-marginal-1} is a partition function of the FK--Ising model on~$\calD^\circ$ conditioned on all edges in~$\omega(\sigma^\bullet)$ being open. Performing the usual duality transformation and using that
	\[
		k(\xi^1)-1 = k(\xi^*) + |\xi^*| - |E(\calD^\bullet)|- |V(\calD^\bullet)|
	\]
	(follows from Euler's formula, can be proven by induction), we obtain
	\begin{align}
		{\sf Spin}_{\calD,c}^\tau(\sigma^\bullet)
		&= \frac{2^{-|V(\calD^\bullet)|}}{2Z}\cdot \left(\tfrac{c-1}{2}\right)^{|E(\calD^\bullet)|} \left(\tfrac{1}{c-1}\right)^{\omega(\sigma^\bullet)}\sum_{\xi^* \subset E(\calD^\bullet)\setminus \theta(\sigma^\bullet)}\left(\tfrac{2}{c-1}\right)^{|\xi^*|} 2^{k(\xi^*)} \nonumber \\
		&= \frac{1}{Z'} \left(\tfrac{1}{c+1}\right)^{\omega(\sigma^\bullet)}\sum_{\xi^* \subset E(\calD^\bullet)\setminus \theta(\sigma^\bullet)}\left(\tfrac{2}{c+1}\right)^{|\xi^*|}\left(\tfrac{c-1}{c+1}\right)^{|E(\calD^\bullet)\setminus \theta(\sigma^\bullet)| - |\xi^*|} 2^{k(\xi^*)},
		\label{eq:fkg-spins-marginal-2}
	\end{align}
	where~$Z'$ is the normalizing constant independent of~$\sigma^\bullet$.
	
	The sum on the RHS of \eqref{eq:fkg-spins-marginal-2} is the partition function of the FK--Ising model on~$\calD^\bullet\setminus\theta(\sigma^\bullet)$ with free boundary conditions and parameter~$p=\tfrac{2}{c+1}$. 
	We denote it by~$Z_{\sf FK}(\calD^\bullet\setminus\theta(\sigma^\bullet))$.
	It is clear that
	\[
		Z_{\sf FK}(\calD^\bullet\setminus\theta(\sigma^\bullet)) = Z_{\sf FK}(P(\sigma^\bullet)) \cdot Z_{\sf FK}(M(\sigma^\bullet)),
	\]
	where~$P(\sigma^\bullet)$ and~$M(\sigma^\bullet)$ are the subgraphs of~$\calD^\bullet$ spanned on the vertices having $\sigma^\bullet$-spin plus or minus, respectively. Then \eqref{eq:fkg-spins-marginal-2} takes form:
	\begin{align}
		{\sf Spin}_{\calD,c}^\tau(\sigma^\bullet) = \tfrac{1}{Z'}\cdot \left(\tfrac{1}{c+1}\right)^{|\theta(\sigma^\bullet)|}\cdot Z_{\sf FK}(P(\sigma^\bullet)) \cdot Z_{\sf FK}(M(\sigma^\bullet)),
		\label{eq:marginal-final}
	\end{align}
	Before inserting this into~\eqref{eq:FKG2}, note that
	\[
		|\theta(\sigma_{++}^\bullet)|+|\theta(\sigma_{-\, -}^\bullet)|-|\theta(\sigma_{+-}^\bullet)|-|\theta(\sigma_{-+}^\bullet)| = -2\cdot \mathbbm{1}_{u\sim v},
	\]
	where by~$u\sim v$ we mean that~$u$ and~$v$ are adjacent in~$\calD^\bullet$.
	Thus, it is enough to show the following two inequalities:
	\begin{align}\label{eq:FKG3}
		(c+1)^{\mathbbm{1}_{u\sim v}}\cdot \frac{Z_{\sf FK}(P(\sigma_{-\, -}^\bullet))}{Z_{\sf FK}(P(\sigma_{++}^\bullet))}
		&\stackrel{?}{\geq} \frac{Z_{\sf FK}(P(\sigma_{+-}^\bullet))}{Z_{\sf FK}(P(\sigma_{++}^\bullet))}\cdot \frac{Z_{\sf FK}(P(\sigma_{-+}^\bullet))}{Z_{\sf FK}(P(\sigma_{++}^\bullet))}, \\
		(c+1)^{\mathbbm{1}_{u\sim v}}\cdot \frac{Z_{\sf FK}(M(\sigma_{++}^\bullet))}{Z_{\sf FK}(M(\sigma_{-\, -}^\bullet))}
		&\stackrel{?}{\geq} \frac{Z_{\sf FK}(M(\sigma_{+-}^\bullet))}{Z_{\sf FK}(M(\sigma_{-\, -}^\bullet))}\cdot \frac{Z_{\sf FK}(M(\sigma_{-+}^\bullet))}{Z_{\sf FK}(M(\sigma_{-\, -}^\bullet))}.\nonumber
	\end{align}	
	We will show only the first inequality, as the second one is analogous.
	Each ratio in \eqref{eq:FKG3} can be linked to the probability of some event under FK--Ising measure on~$P(\sigma_{++}^\bullet)$. Indeed, graph~$P(\sigma_{-+}^\bullet)$ is obtained from~$P(\sigma_{++}^\bullet)$ by removing vertex~$u$ together with all edges that are incident to it. Let~$E_u$ (resp.~$E_v$) be the set of edges in~$P(\sigma_{++}^\bullet)$ incident to~$u$ (resp.~$v$). Then,
	\begin{align*}
		\frac{Z_{\sf FK}(P(\sigma_{-+}^\bullet))}{Z_{\sf FK}(P(\sigma_{++}^\bullet))}
		&= \frac{1}{Z_{\sf FK}(P(\sigma_{++}^\bullet))} \cdot \sum_{\xi^*\subset P(\sigma_{-+}^\bullet)} \left(\tfrac{2}{c+1}\right)^{|\xi^*|}\left(\tfrac{c-1}{c+1}\right)^{|E(P(\sigma_{-+}^\bullet))| - |\xi^*|} 2^{k(\xi^*)}\\
		&= \frac{1}{Z_{\sf FK}(P(\sigma_{++}^\bullet))} \cdot \sum_{\xi^*\subset P(\sigma_{++}^\bullet)} \left(\tfrac{2}{c+1}\right)^{|\xi^*|}\left(\tfrac{c-1}{c+1}\right)^{|E(P(\sigma_{-+}^\bullet))| - |\xi^*|} 2^{k(\xi^*)-1}\mathbbm{1}_{\xi^*\cap E_u=\emptyset},
	\end{align*}
	where we used that, when~$\xi^*\cap E_u=\emptyset$, the number of clusters in~$\xi^*$ gets increased by one when it is viewed as a spanning subgraph of~$P(\sigma_{++}^\bullet)$ instead of a spanning subgraph of~$P(\sigma_{-+}^\bullet)$ (since we need to count a singleton~$u$). Let~$\bbP_{\sf FK}$ denote the FK--Ising measure on~$P(\sigma_{++}^\bullet)$ with parameter~$\tfrac2{c+1}$. In order to write the RHS of the last equation as~$\bbP_{\sf FK}$-probability, it remains to substitute~$|E(P(\sigma_{-+}^\bullet))|$ with~$|E(P(\sigma_{++}^\bullet))|$. Since~$|E(P(\sigma_{++}^\bullet))| - |E(P(\sigma_{-+}^\bullet))| = |E_u|$, we obtain
	\[
		\frac{Z_{\sf FK}(P(\sigma_{-+}^\bullet))}{Z_{\sf FK}(P(\sigma_{++}^\bullet))}  = \tfrac12 \left(\tfrac{c+1}{c-1}\right)^{|E_u|} \bbP_{\sf FK}(\xi^*\cap E_u=\emptyset)
	\]
	Similarly,
	\begin{align*}
		\frac{Z_{\sf FK}(P(\sigma_{+-}^\bullet))}{Z_{\sf FK}(P(\sigma_{++}^\bullet))}  &= \tfrac12 \left(\tfrac{c+1}{c-1}\right)^{|E_v|}\bbP_{\sf FK}(\xi^*\cap E_v=\emptyset) ,\\
		\frac{Z_{\sf FK}(P(\sigma_{-\, -}^\bullet))}{Z_{\sf FK}(P(\sigma_{++}^\bullet))}  &= \tfrac14 \left(\tfrac{c+1}{c-1}\right)^{|E_u\cup E_v|}\bbP_{\sf FK}(\xi^*\cap (E_u\cup E_v) =\emptyset).
	\end{align*}
	Substituting this in~\eqref{eq:FKG3} and using that~$|E_u| + |E_v| - |E_u\cup E_v| = \mathbbm{1}_{u\sim v}$, we get that it is enough to show
	\begin{equation}\label{eq:FKG4}
		(c-1)^{\mathbbm{1}_{u\sim v}}\bbP_{\sf FK}(\xi^*\cap (E_u \cup E_v) =\emptyset)\stackrel{?}{\geq} \bbP_{\sf FK}(\xi^*\cap E_u=\emptyset) \cdot \bbP_{\sf FK}(\xi^*\cap E_v=\emptyset).
	\end{equation}
	If~$u\not\sim v$, then the last inequality takes from of positive association inequality for the FK--Ising model, which is well known (see eg.~\cite[Thm. (3.8)]{Gri06} and also Section~\ref{sec:rcm-input} above).
	
	Assume that~$u\sim v$. Dividing all probabilities in~\eqref{eq:FKG4} by~$\bbP_{\sf FK}(uv\not\in\xi^*)$ and rewriting them as conditional probabilities, we obtain that it is enough to show that
	\begin{align*}
		\frac{c-1}{\bbP_{\sf FK}(uv\not\in\xi^*)}\cdot &\bbP_{\sf FK}(\xi^*\cap (E_u \cup E_v) =\emptyset \, | \, uv\not\in\xi^*)\\
		&\stackrel{?}{\geq} \bbP_{\sf FK}(\xi^*\cap E_u=\emptyset\, | \, uv\not\in\xi^*) \cdot \bbP_{\sf FK}(\xi^*\cap E_v=\emptyset\, | \, uv\not\in\xi^*).
	\end{align*}
	Since the conditional probability~$\bbP_{\sf FK}(\cdot \, | \, uv\not\in\xi^*)$ is equal to the FK--Ising measure on~$P(\sigma_{++}^\bullet)\setminus \{uv\}$, it is also positively associated. Thus, in order to finish the proof, it remains to show
	\[
		c-1 \stackrel{?}{\geq} \bbP_{\sf FK}(uv\not\in\xi^*).
	\]
	Recall that the parameter of the FK--Ising measure is equal to~$\tfrac{2}{c+1}$. Pairing up edge-configurations on~$P(\sigma_{++}^\bullet)$ that coincide everywhere except at~$uv$, one obtains
	\[
		2\cdot \tfrac{c-1}{c+1} \cdot \bbP_{\sf FK}(uv\in\xi^*) \geq \tfrac{2}{c+1}\cdot \bbP_{\sf FK}(uv\not\in\xi^*),
	\]
	whence the claim follows readily.
\end{proof}

\begin{corollary}\label{cor:fkg-spins}
	Let~$\calD$ be an even domain, $c\geq 1$ and~$c_b \geq  0$. Then the marginal distribution of~${\sf Spin}_{\calD}^{++;c_b}$ on~$\sigma^\bullet$ satisfies the FKG lattice condition.
\end{corollary}

This readily implies Proposition~\ref{prop:monotone-c-b-spins}.

\begin{proof}[Proof of Proposition~\ref{prop:monotone-c-b-spins}]
	The proof of Proposition~\ref{prop:monotone-c-b-heights} given in Section~\ref{sec:fkg-heights} can be adapted mutatis mutandis using the FKG inequality stated in Corollary~\ref{cor:fkg-spins}.
\end{proof}

\subsection{Proof of Item 2 of Theorem~\ref{thm:spins-gibbs}}
\label{sec:same-measure-c-2}

\begin{proof}
Let~$\tau\in\calE_{\sf spin}(\bbZ^2)$ be a constant plus at all odd faces. By Theorem~\ref{thm:fkg},
\begin{equation}\label{eq:spins-gibbs-st-dom}
	{\sf Spin}_{\calD,2}^{+-} \stackrel[{\sf even}]{}{\preceq} {\sf Spin}_{\calD,2}^{\tau} \stackrel[{\sf even}]{}{\preceq} {\sf Spin}_{\calD,2}^{++},
\end{equation}
where by~$\stackrel[{\sf even}]{}{\preceq}$ we mean the stochastic domination of the marginals on the spin configurations at even faces. We start by proving that~${\sf Spin}_{\calD,2}^{+-}$ and~${\sf Spin}_{\calD,2}^{++}$ converge to the same limit and then we show how this implies that the limit of~${\sf Spin}_{\calD,2}^{\tau}$ is also the same.

By Corollary~\ref{cor:coupling-c-2}, measure~${\sf Spin}_{\calD,2}^{++}$ can be obtained from the random-cluster measure ${\sf RC}_{\calD^\bullet,4,p_c(4)}^2$ by assigning plus to all boundary clusters and assigning plus or minus independently with probability~$1/2$ to all other clusters. By Proposition~\ref{prop:rcm-input}, the limit of~${\sf RC}_{\calD^\bullet,4,p_c(4)}^2$, as~$\calD\nearrow\bbZ^2$, exists, is the unique random-cluster Gibbs measure~${\sf RC}_{4,p_c(4)}$ with parameters~$q=4$, $p=p_c(4)$ and exhibits infinitely many primal and dual clusters surrounding the origin. Then, the infinite-volume limit of~${\sf Spin}_{\calD,2}^{++}$ also exists, can be obtained from~${\sf RC}_{4,p_c(4)}$ by assigning plus or minus independently with probability~$1/2$ to every cluster and thus exhibits infinitely many circuits of even (or odd) faces having constant spin plus (or minus). Denote this measure by~${\sf Spin}_{2}$. 

Similarly, the infinite-volume limit of~${\sf Spin}_{\calD,2}^{+-}$ is also equal to measure~${\sf Spin}_{2}$.
Extremality of measure~${\sf Spin}_{2}$ and invariance under all translations follow from the same properties of the random-cluster measure~${\sf RC}_{4,p_c(4)}$.

By~\eqref{eq:spins-gibbs-st-dom}, the above immediately implies that the limit of the marginal distribution of~${\sf Spin}_{\calD,c}^{\tau}$ on the spin configurations at even faces exists and is equal to the corresponding marginal of~${\sf Spin}_{2}$. In particular, for any~$\varepsilon,N>0$, when~$\calD$ is large enough,
\[
	{\sf Spin}_{\calD,2}^{\tau}(\exists \text{ circuits } C_+ \text{ and } C_- \text{ surrounding } \Lambda_N \text{ s.t. } \sigma^\bullet(C_+)= 1, \, \sigma^\bullet(C_-)= -1) > 1-\varepsilon,
\]
where by~$\sigma^\bullet(C_+)= 1$ and~$\sigma^\bullet(C_-)= -1$ we mean that~$\sigma^\bullet$ is constant plus at~$C_+$ and constant minus at~$C_-$. When this occurs, the ice-rule implies existence of a circuit~$C$ of odd faces between~$C_+$ and~$C_-$, on which~$\sigma$ is constant plus or minus. Then there exists a simple cyclic path~$\gamma$ on~$\bbZ^2$ between~$C$ and~$C_+$ such that all even faces bordering~$\gamma$ have constant spin plus and all odd faces bordering~$\gamma$ have constant spin (plus or minus). This implies that for any fixed~$n>0$, when~$N$ is large enough, the restriction of~${\sf Spin}_{\calD,c}^{\tau}$ to the box~$\Lambda_n$ is~$\varepsilon$-close to the restriction of the  measure~${\sf Spin}_{2}$ to the same box. Letting~$\varepsilon$ tend to zero finishes the proof.
\end{proof}

\section{The Ashkin--Teller model on the self-dual curve}
\label{sec:proof-ashkin-teller}

In this section we prove Theorem~\ref{thm:ashkin-teller}.

The main tool in the proof is the FK--Ising-type representation~${\sf FKIs}_c$ introduced in Section~\ref{sec:fk-ising} that allows to transfer to the Ashkin--Teller model the results established in Theorem~\ref{thm:spins-gibbs} for the spin representation of the six-vertex model.

In the next proposition we describe a coupling between the six-vertex and the Ashkin--Teller models; see Figure~\ref{fig:coupling-AT-SixV}. Recall graphs~$\calD^\bullet$ and~$\calD^\circ$ dual to each other and the notion of compatible spin and edge configurations (Section~\ref{sec:coupling}).

\begin{proposition}\label{prop:coupling-AT-6V}
	Let~$\calD$ be a domain on~$\bbZ^2$. Let~$J,U\in \bbR$ be such that~$\sinh 2J = e^{-2U}$. Take~$c = \coth 2J$. Let~$(\tau,\tau')$, $\xi$, and~$(\sigma^\bullet,\sigma^\circ)$ be random variables distributed according to measures~${\sf AT}_{\calD^\bullet,J,U}^{\mathrm{free},+}$, ${\sf FKIs}_{\calD^\circ,c}$, and~${\sf Spin}_{\calD,c}^{++}$, respectively. Then, these random variables can be coupled in such a way that their joint law takes the form:
	\begin{equation}\label{eq:coupling-AT-6V}
		(\tau,\tau',\xi,\sigma^\bullet,\sigma^\circ) \propto 2^{-k(\xi^*)} (c-1)^{|\xi| - |\omega(\sigma^\bullet)|}\mathbbm{1}_{\tau\tau'=\sigma^\bullet}\mathbbm{1}_{\tau\perp \xi^*}\mathbbm{1}_{\sigma^\bullet\perp \xi^*}\mathbbm{1}_{\sigma^\circ\perp \xi},
	\end{equation}
	where by~$\tau\perp \xi^*$ we mean that~$\tau$ has a constant value on every cluster of~$\xi^*$, and similarly for other spin and edge configurations. As before, $k(\xi^{*})$ is the number of connected components in $\xi^{*}$ and $\omega(\sigma^{\bullet})$ is he set of edges of $\calD^{\circ}$ separating opposite spins in $\sigma^{\bullet}$.
\end{proposition}
  
\begin{figure}
	\begin{center}
		\includegraphics[width=\textwidth]{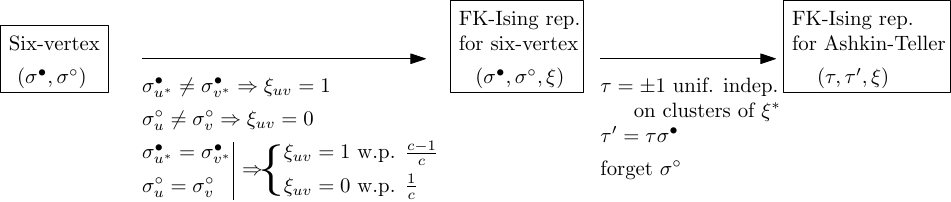}
		\caption{Coupling between the six-vertex and the Ashkin--Teller models through the FK--Ising representation.}
		\label{fig:coupling-AT-SixV}
	\end{center}
\end{figure}

\begin{proof}
	First, note that, given~$\sigma^\bullet$ and~$\tau$, the value of~$\tau'$ is uniquely determined by~$\tau\tau'=\sigma^\bullet$. Also, there are exactly~$2^{k(\xi^*)}$ different spin configurations~$\tau$ that are constant on every cluster of~$\xi^*$ (two possibilities per cluster). Thus, summing~\eqref{eq:coupling-AT-6V} over all possible values of~$(\tau,\tau')$, we obtain~\eqref{eq:fk-ising-coupling}. By Lemma~\ref{lem:sigma-given-eta-o}, the latter gives~${\sf FKIs}_{\calD^\circ,c}$ and~${\sf Spin}_{\calD,c}^{++}$ as marginals.
	
	It remains to show that the marginal of the distribution defined by~\eqref{eq:coupling-AT-6V} on~$(\tau,\tau')$ coincides with~${\sf AT}_{\calD^\bullet,J,U}^{\mathrm{free},+}$. Substituting~$\sigma^\bullet$ with~$\tau\tau'$ and summing~\eqref{eq:coupling-AT-6V} over all spin configurations~$\sigma^\circ$ on~$\calD^\circ$ that have a constant value at each cluster of~$\xi$, which is fixed to be a plus at boundary clusters, we obtain
	\[
		(\tau,\tau',\xi) \propto 2^{k(\xi^1)-k(\xi^*)} (c-1)^{|\xi| - |\omega(\tau\tau')|}\mathbbm{1}_{\tau\perp \xi^*}\mathbbm{1}_{\tau\tau'\perp \xi^*}.
	\]
	Recall that~$k(\xi^1)-k(\xi^*) = |V(\calD^\circ)| - |\xi|$ by Euler's formula, whence the joint distribution takes form
	\begin{equation}\label{eq:coupling-AT-fk-ising}
		(\tau,\tau',\xi) \propto \left(\tfrac{c-1}{2}\right)^{|\xi|}  \left(\tfrac1{c-1}\right)^{|\omega(\tau\tau')|}\mathbbm{1}_{\tau\perp \xi^*}\mathbbm{1}_{\tau\tau'\perp \xi^*}.
	\end{equation}
	Note that, since~$\tau$ and~$\tau\tau'$ have constant values on the clusters of~$\xi^*$, then an edge~$e\in E(\calD^\circ)$ can be closed in~$\xi$ only if both~$\tau$ and~$\tau'$ have constant value at the endpoints of~$e^*$, and otherwise the edge~$e$ must be open in~$\xi$. Thus, summing over edge configurations~$\xi\in\{0,1\}^{E(\calD^\circ)}$, edges belonging to~$\omega(\tau)\cup \omega(\tau')$ contribute~$\tfrac{c-1}{2}$ and other edges contribute~$\tfrac{c+1}{2}$. The distribution then takes form
	\[
		(\tau,\tau') \propto \left(\tfrac{c+1}{2}\right)^{|E(\calD^\circ)|-|\omega(\tau)\cup \omega(\tau')|} \left(\tfrac{c-1}{2}\right)^{\omega(\tau)\cup \omega(\tau')}\left(\tfrac1{c-1}\right)^{\omega(\tau\tau')}.
	\]
	Replacing~$\omega(\cdot)$ with~$\theta(\cdot)$ and multiplying the RHS by~$2^{|E(\calD^\circ)|}$, we obtain
	\[
		(\tau,\tau') \propto (c+1)^{|E(\calD^\bullet)|-|\theta(\tau)\cup \theta(\tau')|} (c-1)^{\theta(\tau)\cup \theta(\tau')}\left(\tfrac1{c-1}\right)^{\theta(\tau\tau')}.
	\]
	Consider any~$u,v\in V(\calD^\bullet)$. If~$\tau(u)=\tau(v)$ and~$\tau'(u)=\tau'(v)$, then~$uv$ contributes~$c+1$ to the RHS. If~$\tau(u)\neq \tau(v)$ and~$\tau'(u)\neq \tau'(v)$, then~$uv$ contributes~$c-1$ to the RHS. Otherwise, $uv$ contributes~$(c-1)\cdot \tfrac1{c-1} = 1$ to the RHS. It remains to check that these contributions are the same in case of the Ashkin--Teller measure~${\sf AT}_{\calD^\bullet}^{\mathrm{free},+}$.
	
	Indeed, if~$\tau(u)=\tau(v)$ and~$\tau'(u)=\tau'(v)$, then~$uv$ contributes~$e^{2J+U}$ to~${\sf AT}_{\calD^\bullet}^{\mathrm{free},+}$. If~$\tau(u)\neq \tau(v)$ and~$\tau'(u)\neq \tau'(v)$, then~$uv$ contributes~$e^{-2J+U}$ to~${\sf AT}_{\calD^\bullet}^{\mathrm{free},+}$. Otherwise, $uv$ contributes~$e^{-U}$ to~${\sf AT}_{\calD^\bullet}^{\mathrm{free},+}$. Multiplying all these contributions by~$e^U$ and comparing to the above, we get that it is enough to check that
	\[
		e^{2J+2U} \stackrel{?}{=} c+1 \quad \text{and} \quad e^{-2J+2U} \stackrel{?}{=} c-1.
	\]	
	Since~$e^{-2U} = \sinh 2J$ and~$c=\coth 2J$, we have
	\begin{align*}
		e^{2J+2U} &= \tfrac{2e^{2J}}{e^{2J}-e^{-2J}}= \tfrac{e^{2J}+e^{-2J}}{e^{2J}-e^{-2J}}+1 = c+1,\\
		e^{-2J+2U} &= \tfrac{2e^{-2J}}{e^{2J}-e^{-2J}} = \tfrac{e^{2J}+e^{-2J}}{e^{2J}-e^{-2J}}-1 = c-1.\qedhere
	\end{align*}
\end{proof}

\begin{remark}\label{rem:fk-ising-rep}
	By~\eqref{eq:coupling-AT-fk-ising}, given~$\tau$ and~$\tau'$, the configuration~$\xi^*$ can be sampled at every edge~$uv$ independently in the following way: if~$\tau(u)\neq \tau(v)$ or~$\tau'(u)\neq \tau'(v)$, then~$uv\not\in \xi^*$; if~$\tau(u)=\tau(v)$ and~$\tau'(u) =\tau'(v)$, then~$uv\in\xi^*$ w.p.~$1-e^{-4J}$ and~$uv\not\in\xi^*$ w.p.~$e^{-4J}$. Thus, $\xi^*$ coincides with~$\left[\underline{n}^{-1}(0,0)\right]^c$ in~\cite{PfiVel97}.
\end{remark}

\begin{corollary}\label{cor:AT-given-eta-o}
	In the notation of Proposition~\ref{prop:coupling-AT-6V}, assign~$1$ or~$-1$ uniformly at random independently to every cluster of~$\xi^*$. Then, the obtained random spin configuration on~$\calD^\bullet$ has the same distribution as the marginal of~${\sf AT}_{\calD^\bullet}^{\mathrm{free},+}$ on~$\tau$ (or, equivalently, $\tau'$).
	
	In particular, the following holds:
	\begin{equation}\label{eq:corr-tau-eta-o}
		{\sf AT}_{\calD^\bullet}^{\mathrm{free},+}(\tau(u)\tau(v)) = {\sf FKIs}_{\calD^\circ,c}(u \xleftrightarrow{\xi^*} v).
	\end{equation}
\end{corollary}

\begin{proof}
	This follows from Proposition~\ref{prop:coupling-AT-6V} since, for a fixed value of~$\tau\tau'$, the coupling~\eqref{eq:coupling-AT-6V} (or~\eqref{eq:coupling-AT-fk-ising}) does not depend on the value of~$\tau$ at clusters of~$\xi^*$.
\end{proof}

We now have all the ingredients to finish the proof of Theorem~\ref{thm:ashkin-teller}.

\begin{proof}[Proof of Theorem~\ref{thm:ashkin-teller}]
	We start by showing that a limit of~${\sf AT}_{\calD^\bullet,J,U}^{\mathrm{free},+}$ as~$\calD\nearrow\bbZ^2$ exists. Consider the coupling~$(\tau,\tau',\xi,\sigma^\bullet,\sigma^\circ)$ introduced in~Proposition~\ref{prop:coupling-AT-6V}. Then, the joint distribution of~$(\tau\tau', \xi)$ coincides with the joint distribution of~$(\sigma^\bullet,\xi)$. By Corollary~\ref{cor:sigma-given-fk-ising-infinite}, the distribution of $\xi$ converges to an extremal measure that exhibits a unique infinite cluster. Also, the limit of the joint distribution of~$(\sigma^\bullet,\xi)$ is obtained by assigning $1$ to the infinite cluster and assigning $1$ or $-1$ uniformly at random independently to every finite cluster. Then, the limiting joint distribution is also extremal.
	
	By Corollary~\ref{cor:AT-given-eta-o}, given~$\tau\tau'$ and~$\xi$, the distribution of~$\tau$ according to the coupling~\eqref{eq:coupling-AT-fk-ising} is obtained by assigning spins~$1$ and~$-1$ uniformly at random independently to different clusters of~$\xi^*$. By Proposition~\ref{prop:exp-decay-eta-o}, the clusters of~$\xi^*$ are finite, and thus, the joint law of~$(\tau,\tau',\xi)$ described by~\eqref{eq:coupling-AT-fk-ising} also converges to an extremal measure. Taking the marginal on~$(\tau,\tau')$, we obtain convergence of the Ashkin--Teller measures to an extremal limit denoted by~${\sf AT}_{J,U}^{\mathrm{free},+}$, which is also translation-invariant.
	
	It is easy to see that~$\Omega_k$ defined in Section~\ref{sec:ashkin-teller} coincides with~$\calD^\bullet$, when~$\calD$ is the square~$[-k-\tfrac12,k+\tfrac12]\times [-k-\tfrac12,k+\tfrac12]$. Thus, the above implies that the sequence of measures~${\sf AT}_{\Lambda_k^\bullet,J,U}^{\mathrm{free},+}$ converges to~${\sf AT}_{J,U}^{\mathrm{free},+}$.
	
	By our construction, the marginal of~${\sf AT}_{J,U}^{\mathrm{free},+}$ on~$\tau\tau'$ coincides with the marginal of~${\sf Spin}_c^{++}$ on~$\sigma^\bullet$. The distribution of~$\sigma^\bullet$ coincides with that of~$\sigma^\circ$ shifted by one. Thus, by~\eqref{eq:correlations-sigma-as-xi}, for any two even faces~$u,v$ of~$\bbZ^2$,
	\[
		{\sf AT}_{J,U}^{\mathrm{free},+}(\tau(u)\tau'(u)\tau(v)\tau'(v)) = {\sf FKIs}_c ((u+(0,1))\xleftrightarrow{\xi} (v+(0,1))) \geq {\sf FKIs}_c ((0,1)\xleftrightarrow{\xi} \infty)^2  >0,
	\]
	where we used translation-invariance of~${\sf FKIs}_c$ and that, by Proposition~\ref{prop:exp-decay-eta-o}, ${\sf FKIs}_c$ exhibits a unique infinite cluster.
	
	This proves the bound~\eqref{eq:AT-corr-bounded}. Similarly, by~\eqref{eq:corr-tau-eta-o}, the exponential decay of connectivities in~$\xi^*$ established in Proposition~\ref{prop:exp-decay-eta-o} implies~\eqref{eq:AT-corr-exp-decay}.
\end{proof}

\section{Discussion and open questions}\label{sec:open questions}

{\bf Ashkin--Teller model.} The segment of the self-dual curve~$\sinh 2J = e^{-2U}$ where $J\ge U$ can be shown to be critical using~\cite{DumRauTas17}\footnote{See a recent work~\cite{AouDobGla23} for the details.}. It is natural to try to extend the statement of Theorem~\ref{thm:ashkin-teller} to prove that no transition takes place on the self-dual curve when $J<U$, so that the critical line indeed splits into two critical curves immediately after~$J=U=\tfrac14\log 3$ (both transitions can be shown to be sharp using~\cite{DumRauTas17}).

\begin{question}\footnote{During the revision process, this question has been answered positively in~\cite{AouDobGla23}.}
	Fix~$\alpha < 1$. Show that the line~$J = \alpha U$ contains a non-trivial interval of parameters for which, under~${\sf AT}_{J,U}^{\mathrm{free},+}$, correlations of spins~$\tau$ and~$\tau'$ exhibit exponential decay while the product~$\tau\tau'$ is ferromagnetically ordered.
\end{question}
 
{\bf Height function with modified boundary weight.}
Fix~$a=b=1$ and~$c>2$. By Corollary~\ref{cor:therm-lim-heights-c-b}, the sequence of measures~${\sf HF}_{\calD_k,c}^{0,1;c_b}$, as~$\calD_k\nearrow\bbZ^2$, converges to~${\sf HF}_c^{0,1}$ when~$c_b\geq e^{\lambda/2}$ and to~${\sf HF}_c^{0,-1}$ when~$c_b \leq e^{-\lambda/2}$. What happens when~$c_b\in (e^{-\lambda/2},e^{\lambda/2})$?

\begin{question}
	Show that, as~$\calD_k\nearrow\bbZ^2$, the measure~${\sf HF}_{\calD_k,c}^{0,1;c_b}$ converges to~${\sf HF}_c^{0,1}$ when~$c_b>1$, to~${\sf HF}_c^{0,-1}$ when~$c_b<1$, and to the mixture~$\tfrac12({\sf HF}_c^{0,1}+ {\sf HF}_c^{0,-1})$ at~$c_b=1$.
\end{question}

{\bf FK model with modified boundary-cluster weight.}
Fix~$q>4$. By~\cite{DumGagHar16} (see~\cite{RaySpi19} for a recent short proof), the free and the wired Gibbs states for the critical random-cluster model with cluster-weight~$q$ are distinct (and by~\cite{GlaMan21} there are no other extremal Gibbs states). By Theorem~\ref{thm:q-b}, the measure~${\sf RC}_{\Omega,p_c(q),q}^{q_b}$ converges to the free measure when~$q_b\in [e^{\lambda}\sqrt{q},q]$ and to the wired measure when~$q_b\in [1,e^{-\lambda}\sqrt{q}]$. What happens when~$q_b\in (e^{-\lambda}\sqrt{q},e^{\lambda}\sqrt{q})$, or~$q_b <1$, or~$q_b>q$? It is easy to see that the measures with parameters~$q_b$ and~$q/q_b$ are dual to each other.

\begin{question}\label{ques:rc-qb}
	Show that~${\sf RC}_{\Omega,p_c(q),q}^{q_b}$ converges to the free measure when~$q_b > \sqrt{q}$, to the wired measure when~$q_b < \sqrt{q}$, and to the mixture~$\tfrac12({\sf RC}_{p_c(q),q}^{\mathrm{free}}+ {\sf RC}_{p_c(q),q}^{\mathrm{wired}})$ when~$q_b=\sqrt{q}$. The self-dual point~$q_b=\sqrt{q}$ is then critical.
\end{question}

{\bf FKG inequality.}
Fix~$a=b=1$. By Theorem~\ref{thm:fkg}, the FKG inequality for the marginal distribution on spin configurations on one of the two sublattices holds for any~$c\geq 1$. However, the FKG inequality for the random-cluster model holds only when~$q\geq 1$ which via the BKW coupling corresponds to~$c\geq \sqrt{3}$. Since the absence of an FKG inequality for the random-cluster model when~$q<1$ is the main reason why this regime of parameters is less studied, this motivates the following question.

\begin{question}
	Find an image of the six-vertex FKG inequality under the Baxter--Kelland--Wu coupling (see Theorem~\ref{thm:coupling}).
\end{question} 

{\bf Lipschitz clock model.}
Our work introduces a coupling of the self-dual Ashkin--Teller model with the six-vertex model with parameters $a=b=1$ and $c>1$; see Section~\ref{sec:ashkin-teller}. The coupling extends to the limiting case $c=1$, where the Ashkin--Teller model should be understood as the limit of the measures~\eqref{eq:Ashkin-Teller measure def} (on suitable domains and with suitable boundary conditions, as in the measure~${\sf AT}_{\Omega_k,J,U}^{\mathrm{free},+}$) as $J\to\infty$ and $\sinh 2J = e^{-2U}$. This limiting measure, which is coupled to the uniform six-vertex model (square-ice) studied in~\cite{She05,ChaPelSheTas18,DumHarLasRauRay18,ray2020proper}, may be described explicitly as the case $x=\frac{1}{2}$ of the following \emph{$4$-state Lipschitz clock model}.

The $4$-state Lipschitz clock model with parameter $x>0$ on a graph $G = (V,E)$ is supported on pairs of spin configurations~$(\tau,\tau')\in \{-1,1\}^V\times  \{-1,1\}^V$ which satisfy the \emph{Lipschitz condition}: either $\tau_u = \tau_v$ or $\tau'_u = \tau'_v$ for every edge $\{u,v\}\in E$. It assigns probability proportional to $x^{N(\tau,\tau')}$ to each configuration in its support, where $N(\tau,\tau') = |\{\{u,v\}\in E\colon (\tau_u, \tau'_u) \neq (\tau_v, \tau'_v)\}|$.

The case $x=\frac{1}{2}$ of the $4$-state Lipschitz clock model on $\Z^2$ is further shown in~\cite[Section 4]{ikhlef2012spin} to be in correspondence with an integrable 19-vertex (19V) model and a dilute Brauer model with loop weight $n=2$.

The $4$-state Lipschitz clock model at $x=1$ on a bipartite graph is equivalent to the proper $4$-coloring model (i.e., the zero-temperature $4$-state antiferromagnetic Potts model) through the mapping which flips the signs of both $\tau$ and $\tau'$ on one of the bipartition classes. This model is predicted to exhibit exponential decay of correlations on $\Z^2$~\cite{ferreira1995antiferromagnetic}.

On the triangular lattice, the $4$-state Lipschitz clock model is equivalent to integer-valued Lipschitz height functions (the loop $O(2)$ model; see, e.g., \cite{PelSpi17, GlaMan18}) by taking the modulo $4$ of the height function. However, this equivalence does not extend to the square lattice due to the appearance of vortices in the clock model (cycles of length $4$ in $\Z^2$ on which all four spin values appear).

As the Ashkin--Teller model with $J\ge U$ has a phase transition at the self-dual line it is natural to conjecture that this is the case also for the limiting model.
\begin{question}
	Prove that the $4$-state Lipschitz clock model on $\Z^2$ undergoes a phase transition at $x=\frac{1}{2}$.
\end{question}

{\bf $\mathbb{T}$-connectivity and delocalization of the height function of square-ice.} As discussed in Section~\ref{sec:overview of the proof}, the use of the triangular lattice connectivity in proving the equality of the height function measures ${\sf HF}_{\mathrm{even}, c}^{0,1}$ and ${\sf HF}_{\mathrm{odd},c}^{0,1}$ is one of the main novelties of this work. We hope that this argument would be of further use and outline here an additional consequence that may be obtained with it.

The height function of the square-ice model (the case $a=b=c$ of the six-vertex model) has been shown to delocalize~\cite{ChaPelSheTas18} with an argument relying on the seminal results of Sheffield~\cite{She05}. The method that we use for showing that ${\sf HF}_{\mathrm{even}, c}^{0,1}={\sf HF}_{\mathrm{odd},c}^{0,1}$ can give an alternative proof of delocalization, at least with zero boundary conditions, which does not rely on~\cite{She05}.

Consider the height function of the square-ice model in a sequence of \emph{odd} domains with zero boundary conditions which increases to $\Z^2$. Let $\mathcal{L}_{\text{even}}$ be the set of parity-preserving translations of $\Z^2$. A simple consequence of the FKG inequality for the absolute value of the model~\cite{BenHagMos00} is the following dichotomy~\cite[Theorem 1.1]{ChaPelSheTas18}: either the height at the origin is not tight in this limit (that is, the height delocalizes) or a limiting Gibbs measure, denoted ${\sf HF}^{0}$, exists and is $\mathcal{L}_{\text{even}}$-invariant. Assume, in order to obtain a contradiction, that the latter alternative occurs. By symmetry, this also implies the existence of the limiting Gibbs measures ${\sf HF}^{k}$ for each integer $k$ (as the thermodynamic limit of odd/even domains with boundary value $k$ according to whether $k$ is even/odd). We will show that
\begin{equation}\label{eq:measure equality square ice}
{\sf HF}^{-1}={\sf HF}^{1}
\end{equation}
which yields a contradiction since samples of ${\sf HF}^{1}$ can be obtained by adding $2$ to samples of ${\sf HF}^{-1}$.

Positive association of the heights (in finite volume) implies that ${\sf HF}^{1}\succeq{\sf HF}^{0}$ (writing $\succeq$ for stochastic domination). We will show that in fact
\begin{equation}\label{eq:measure equality square ice2}
{\sf HF}^{1}={\sf HF}^{0}.
\end{equation}
This implies~\eqref{eq:measure equality square ice} (thus yielding the contradiction) since ${\sf HF}^{-1}$ is obtained from ${\sf HF}^{1}$ by negating the sign of its samples while ${\sf HF}^{0}$ is invariant to this operation.

The measure ${\sf HF}^{0}$ is $\mathcal{L}_{\text{even}}$-invariant and has positive association (since it is obtained in the thermodynamic limit). To highlight the main parts of the argument first, we now assume that ${\sf HF}^{0}$ is $\mathcal{L}_{\text{even}}$-ergodic. This will allow us to conclude, via the method of $\mathbb{T}$-connectivity, after which we will revisit the ergodicity assumption.

We let $h$ be sampled from ${\sf HF}^{0}$ and study the set of odd faces $u$ where $h(u)\ge 1$ and the set where $h(u)\le -1$ where we endow both sets with the $\mathbb{T}^{\circ}$-connectivity (Definition \ref{def-t-circuits}). As in the proof of Theorem~\ref{thm:heights-gibbs} for the case $a+b<c$, each of the sets may have at most one infinite cluster ({\em Burton--Keane argument}) and, moreover, the two sets may not have an infinite cluster simultaneously ({\em non-coexistence}). By the symmetry of ${\sf HF}^{0}$ under sign flip, we conclude that, in fact, neither has an infinite cluster. This allows to apply the $\mathbb{T}$-circuit argument (as overviewed in Section~\ref{sec:overview of the proof}) to deduce that ${\sf HF}^{0}\succeq {\sf HF}^{1}$, yielding~\eqref{eq:measure equality square ice2}, and thus finishing the proof of delocalization.

We now prove that ${\sf HF}^{0}$ is $\mathcal{L}_{\text{even}}$-ergodic. We first observe that if there exists an $\mathcal{L}_{\text{even}}$-ergodic measure $\mu$ which is also extremal, then $\mu={\sf HF}^{k}$ for some $k$, whence all the ${\sf HF}^{k}$ are $\mathcal{L}_{\text{even}}$-ergodic and extremal. To see this, let $h$ be sampled from $\mu$. For each integer $k$, let $I_{\ge k}$ be the event that $h\ge k$ has an infinite cluster (in the standard connectivity on all faces of $\Z^2$) and define $I_{<k}$ analogously with $h<k$. The standard percolation arguments imply that for each integer $k$, either $\mu(I_{\ge k}) = 0$ or $\mu(I_{<k}) = 0$. In the first case, there are infinitely many circuits surrounding the origin where $h<k$, from which it follows that $\mu\preceq {\sf HF}^{k-1}$. Thus it is impossible that $\mu(I_{\ge k})=0$ for all $k$. Similarly, by the second case, it is impossible that $\mu(I_{<k})=0$ for all $k$. The above facts imply that there is a $k_0$ such that $\mu(I_{\ge k_0}) = 1$ and $\mu(I_{\ge k_0+1})=\mu(I_{<k_0})=0$, whence there are infinitely many circuits surrounding the origin where $h = k_0$. It follows that $\mu = {\sf HF}^{k_0}$.

It remains to show that there exists an $\mathcal{L}_{\text{even}}$-ergodic and extremal measure $\mu$. This follows from the general results of~\cite{She05} (which show that every $\mathcal{L}_{\text{even}}$-ergodic measure is extremal) but may also be argued directly for the square-ice model as we now explain. Let $h$ be sampled from ${\sf HF}^{0}$. There is at most one infinite cluster where $h>0$ and at most one infinite cluster where $h<0$ (the argument of Burton--Keane may be invoked for this as in the proof of Theorem~\ref{thm:heights-gibbs}). This further implies that there is zero probability that both $h>0$ and $h<0$ have infinite clusters since the construction of ${\sf HF}^{0}$ implies that the signs of distinct infinite clusters of $h\neq 0$ are uniform and independent. It follows that $|h|$ is a Gibbs measure for the absolute value specification. Moreover, using the FKG for absolute values, the distribution of $|h|$ is extremal as it is minimal in the absolute value specification. Since the signs of $h$ on clusters where $h$ is non-zero are uniform and independent, it follows that the extremal components of ${\sf HF}^{0}$ differ only in the sign assigned to the \emph{infinite} cluster where $h$ is non-zero. In particular, if $h\neq 0$ has no infinite cluster with probability one, then $ {\sf HF}^{0}$ is extremal, and otherwise $h>0$ has an infinite cluster with positive probability and conditioning on this event yields an $\mathcal{L}_{\text{even}}$-ergodic and extremal Gibbs measure, as required for the argument.




\bibliographystyle{amsalpha} 
\bibliography{bib-6v}       

\newcommand{\etalchar}[1]{$^{#1}$}
\providecommand{\bysame}{\leavevmode\hbox to3em{\hrulefill}\thinspace}
\providecommand{\MR}{\relax\ifhmode\unskip\space\fi MR }
\providecommand{\MRhref}[2]{%
  \href{http://www.ams.org/mathscinet-getitem?mr=#1}{#2}
}
\providecommand{\href}[2]{#2}
\begin{thebibliography}{DCKMO20}

\bibitem[ADG23]{AouDobGla23}
Yacine Aoun, Moritz Dober, and Alexander Glazman, \emph{Phase diagram of the
  {A}shkin--{T}eller model}, arXiv preprint arXiv:2301.10609 (2023).

\bibitem[Agg18]{aggarwal2018current}
Amol Aggarwal, \emph{Current fluctuations of the stationary {ASEP} and
  six-vertex model}, Duke Mathematical Journal \textbf{167} (2018), no.~2,
  269--384.

\bibitem[AT43]{AshTel43}
Julius Ashkin and Edward Teller, \emph{Statistics of two-dimensional lattices
  with four components}, Physical Review \textbf{64} (1943), no.~5-6, 178.

\bibitem[Bax82]{Bax82}
Rodney~J Baxter, \emph{Exactly solved models in statistical mechanics},
  Academic Press Inc., London, 1982.

\bibitem[BCG16]{BorCorGor16}
Alexei Borodin, Ivan Corwin, and Vadim Gorin, \emph{Stochastic six-vertex
  model}, Duke Mathematical Journal \textbf{165} (2016), no.~3, 563--624.

\bibitem[BDC12]{BefDum12}
Vincent Beffara and Hugo Duminil-Copin, \emph{The self-dual point of the
  two-dimensional random-cluster model is critical for {$q\geq 1$}}, Probab.
  Theory Related Fields \textbf{153} (2012), no.~3-4, 511--542.

\bibitem[Bet31]{Bet31}
Hans Bethe, \emph{Zur {T}heorie der {M}etalle}, Zeitschrift f{\"u}r Physik
  \textbf{71} (1931), no.~3-4, 205--226.

\bibitem[BHM00]{BenHagMos00}
Itai Benjamini, Olle H{\"a}ggstr{\"o}m, and Elchanan Mossel, \emph{On random
  graph homomorphisms into $\mathbb{Z}$}, Journal of Combinatorial Theory,
  Series B \textbf{78} (2000), no.~1, 86--114.

\bibitem[BK89]{BurKea89}
R.~M. Burton and M.~Keane, \emph{Density and uniqueness in percolation}, Comm.
  Math. Phys. \textbf{121} (1989), no.~3, 501--505.

\bibitem[BKW76]{BaxKelWu76}
Rodney~J Baxter, Stewart~B Kelland, and Frank~Y Wu, \emph{Equivalence of the
  {P}otts model or {W}hitney polynomial with an ice-type model}, Journal of
  Physics A: Mathematical and General \textbf{9} (1976), no.~3, 397.

\bibitem[CGST20]{CorGhoSheTsa20}
Ivan Corwin, Promit Ghosal, Hao Shen, and Li-Cheng Tsai, \emph{Stochastic {PDE}
  limit of the six vertex model}, Communications in Mathematical Physics
  \textbf{375} (2020), no.~3, 1945--2038.

\bibitem[Cha98]{Cha98}
L~Chayes, \emph{Discontinuity of the spin-wave stiffness in the two-dimensional
  {XY} model}, Communications in mathematical physics \textbf{197} (1998),
  no.~3, 623--640.

\bibitem[CHI15]{CheHonIzy15}
Dmitry Chelkak, Cl{{\'e}}ment Hongler, and Konstantin Izyurov, \emph{Conformal
  invariance of spin correlations in the planar {I}sing model}, Ann. of Math.
  (2) \textbf{181} (2015), no.~3, 1087--1138.

\bibitem[CM98]{ChaMac98}
L.~Chayes and J.~Machta, \emph{Graphical representations and cluster algorithms
  {II}}, Physica A: Statistical Mechanics and its Applications \textbf{254}
  (1998), no.~3, 477 -- 516,
  \href{https://doi.org/10.1016/S0378\%2D4371(97)00637\%2D7}{doi:10.1016/S0378--4371(97)00637--7}.

\bibitem[CMW98]{ChaMcWin98}
Lincoln Chayes, D~McKellar, and B~Winn, \emph{Percolation and {G}ibbs states
  multiplicity for ferromagnetic {A}shkin--{T}eller models on~$\mathbb{Z}^2$},
  Journal of Physics A: Mathematical and General \textbf{31} (1998), no.~45,
  9055.

\bibitem[CPST21]{ChaPelSheTas18}
Nishant Chandgotia, Ron Peled, Scott Sheffield, and Martin Tassy,
  \emph{Delocalization of uniform graph homomorphisms from $\mathbb{Z}^2$ to
  $\mathbb{Z}$}, Communications in Mathematical Physics \textbf{387} (2021),
  no.~2, 621--647.

\bibitem[CS00]{ChaSht00}
L~Chayes and K~Shtengel, \emph{Lebowitz inequalities for {A}shkin--{T}eller
  systems}, Physica A: Statistical Mechanics and its Applications \textbf{279}
  (2000), no.~1-4, 312--323.

\bibitem[DBGK80]{DitBanGreKad80}
Ruth~V Ditzian, Jayanth~R Banavar, GS~Grest, and Leo~P Kadanoff, \emph{Phase
  diagram for the {A}shkin--{T}eller model in three dimensions}, Physical
  Review B \textbf{22} (1980), no.~5, 2542.

\bibitem[DCGH{\etalchar{+}}18]{DumGagHar16b}
Hugo Duminil-Copin, Maxime Gagnebin, Matan Harel, Ioan Manolescu, and Vincent
  Tassion, \emph{The {B}ethe ansatz for the six-vertex and {XXZ} models: {A}n
  exposition}, Probability Surveys \textbf{15} (2018), 102--130.

\bibitem[DCGH{\etalchar{+}}21]{DumGagHar16}
\bysame, \emph{Discontinuity of the phase transition for the planar
  random-cluster and {P}otts models with $q> 4$}, Annales Scientifiques de
  l'Ecole Normale Sup{\'e}rieure \textbf{54} (2021), no.~6, 1363--1413.

\bibitem[DCGPS21]{DumGlaPel17}
Hugo Duminil-Copin, Alexander Glazman, Ron Peled, and Yinon Spinka,
  \emph{Macroscopic loops in the loop ${O} (n)$ model at {N}ienhuis' critical
  point.}, Journal of the European Mathematical Society (EMS Publishing)
  \textbf{23} (2021), no.~1.

\bibitem[DCHL{\etalchar{+}}22]{DumHarLasRauRay18}
Hugo Duminil-Copin, Matan Harel, Benoit Laslier, Aran Raoufi, and Gourab Ray,
  \emph{Logarithmic variance for the height function of square-ice},
  Communications in Mathematical Physics \textbf{396} (2022), no.~2, 867--902.

\bibitem[DCKMO20]{DumKarManOul20b}
Hugo Duminil-Copin, Alex Karrila, Ioan Manolescu, and Mendes Oulamara, \emph{On
  delocalization in the six-vertex model}, arXiv:2012.13750, 2020.

\bibitem[DCLM18]{DumLiMan18}
Hugo Duminil-Copin, Jhih-Huang Li, and Ioan Manolescu, \emph{Universality for
  the random-cluster model on isoradial graphs}, Electronic Journal of
  Probability \textbf{23} (2018).

\bibitem[DCRT18]{DumRauTas16}
Hugo Duminil-Copin, Aran Raoufi, and Vincent Tassion, \emph{A new computation
  of the critical point for the planar random-cluster model with $q\geq 1$},
  Annales de l’Institut Henri Poincar{\'e}-Probabilit{\'e}s et Statistiques
  \textbf{54} (2018), no.~1, 422--436.

\bibitem[DCRT19]{DumRauTas17}
\bysame, \emph{Sharp phase transition for the random-cluster and {P}otts models
  via decision trees}, Annals of Mathematics \textbf{189} (2019), no.~1,
  75--99.

\bibitem[DCST17]{DumSidTas17}
Hugo Duminil-Copin, Vladas Sidoravicius, and Vincent Tassion, \emph{Continuity
  of the phase transition for planar random-cluster and {P}otts models with
  $1\leq q \leq 4$}, Communications in Mathematical Physics \textbf{349}
  (2017), no.~1, 47--107.

\bibitem[Dub11]{Dub11}
Julien Dubédat, \emph{Exact bosonization of the {I}sing model},
  arXiv:1112.4399, 2011.

\bibitem[Fan72]{Fan72}
C~Fan, \emph{On critical properties of the {A}shkin--{T}eller model}, Physics
  Letters A \textbf{39} (1972), no.~2, 136.

\bibitem[FK72]{ForKas72}
Cornelius~Marius Fortuin and Piet~W Kasteleyn, \emph{On the random-cluster
  model: {I}. {I}ntroduction and relation to other models}, Physica \textbf{57}
  (1972), no.~4, 536--564.

\bibitem[FKG71]{ForKasGin71}
Cees~M Fortuin, Pieter~W Kasteleyn, and Jean Ginibre, \emph{Correlation
  inequalities on some partially ordered sets}, Communications in Mathematical
  Physics \textbf{22} (1971), no.~2, 89--103.

\bibitem[FS95]{ferreira1995antiferromagnetic}
Sabino~Jose Ferreira and Alan~D Sokal, \emph{Antiferromagnetic {P}otts models
  on the square lattice}, Physical Review B \textbf{51} (1995), no.~10,
  6727--30.

\bibitem[Gaa79]{Gaa79}
Aris Gaaff, \emph{Symmetry properties of the sixteen-vertex model},
  \url{http://inis.iaea.org/search/search.aspx?orig_q=RN:11500584}, 1979, p.~1.

\bibitem[GM21]{GlaMan18}
Alexander Glazman and Ioan Manolescu, \emph{Uniform {L}ipschitz functions on
  the triangular lattice have logarithmic variations}, Communications in
  mathematical physics \textbf{381} (2021), no.~3, 1153--1221.

\bibitem[GM23]{GlaMan21}
\bysame, \emph{Structure of {G}ibbs measure for planar {FK}-percolation and
  {P}otts models}, Probability and Mathematical Physics \textbf{4} (2023),
  no.~2, 209--256.

\bibitem[GMT17]{GiuMasTon17}
Alessandro Giuliani, Vieri Mastropietro, and Fabio~Lucio Toninelli,
  \emph{Height fluctuations in interacting dimers}, Annales de l'Institut Henri
  Poincar{\'e}, Probabilit{\'e}s et Statistiques, vol.~53, Institut Henri
  Poincar{\'e}, 2017, pp.~98--168.

\bibitem[Gri67]{Gri67}
Robert~B Griffiths, \emph{Correlations in ising ferromagnets. {I,II}}, Journal
  of Mathematical Physics \textbf{8} (1967), no.~3, 478--489.

\bibitem[Gri06]{Gri06}
G.~Grimmett, \emph{The random-cluster model}, Grundlehren der Mathematischen
  Wissenschaften [Fundamental Principles of Mathematical Sciences], vol. 333,
  Springer-Verlag, Berlin, 2006.

\bibitem[H{\"a}g98]{Hag96}
O.~H{\"a}ggstr\"{o}m, \emph{Random-cluster representations in the study of
  phase transitions}, Markov Process. Related Fields \textbf{4} (1998), no.~3,
  275--321.

\bibitem[HDJS13]{HuaDenJacSal13}
Yuan Huang, Youjin Deng, Jesper~Lykke Jacobsen, and Jes{\'u}s Salas, \emph{The
  {H}intermann--{M}erlini--{B}axter--{W}u and the infinite-coupling-limit
  {A}shkin--{T}eller models}, Nuclear Physics B \textbf{868} (2013), no.~2,
  492--538.

\bibitem[IR12]{ikhlef2012spin}
Yacine Ikhlef and Mohammad~Ali Rajabpour, \emph{Spin interfaces in the
  {A}shkin--{T}eller model and {S}{L}{E}}, Journal of Statistical Mechanics:
  Theory and Experiment \textbf{2012} (2012), no.~01, 1--12.

\bibitem[Ken00]{Ken00}
Richard Kenyon, \emph{Conformal invariance of domino tiling}, Ann. Probab.
  \textbf{28} (2000), no.~2, 759--795.

\bibitem[Kno75]{Kno75}
H~J~F Knoops, \emph{A branch point in the critical surface of the
  {A}shkin--{T}eller model in the renormalization group theory}, Journal of
  Physics A: Mathematical and General \textbf{8} (1975), no.~9, 1508.

\bibitem[Kot85]{Kot85}
Roman Koteck{\`y}, \emph{Long-range order for antiferromagnetic {P}otts
  models}, Physical Review B \textbf{31} (1985), no.~5, 3088.

\bibitem[Lie67a]{Lie67b}
Elliott~H Lieb, \emph{Exact solution of the {F} model of an antiferroelectric},
  Phys. Rev. Lett. \textbf{18} (1967), 1046.

\bibitem[Lie67b]{Lie67a}
\bysame, \emph{Exact solution of the problem of the entropy of two-dimensional
  ice}, Physical Review Letters \textbf{18} (1967), no.~17, 692--694.

\bibitem[Lie67c]{lieb1967exact}
\bysame, \emph{Exact solution of the two-dimensional {S}later {KDP} model of a
  ferroelectric}, Physical Review Letters \textbf{19} (1967), no.~3, 108.

\bibitem[Lie67d]{Lie67c}
\bysame, \emph{The residual entropy of square ice}, Physical Review
  \textbf{162} (1967), 162--172.

\bibitem[Lis21]{Lis20}
Marcin Lis, \emph{On delocalization in the six-vertex model}, Communications in
  Mathematical Physics \textbf{383} (2021), 1181--1205.

\bibitem[Lis22]{Lis19}
\bysame, \emph{Spins, percolation and height functions}, Electronic Journal of
  Probability \textbf{27} (2022), 1--21.

\bibitem[LW72]{LieWu72}
EH~Lieb and FY~Wu, \emph{Two-dimensional ferroelectric models}, Phase
  Transitions and Critical Phenomena 1, ed. C. Domb and MS Green, 1972.

\bibitem[MS71]{MitSte71}
L~Mittag and MJ~Stephen, \emph{Dual transformations in many-component {I}sing
  models}, Journal of Mathematical Physics \textbf{12} (1971), no.~3, 441--450.

\bibitem[MW14]{mccoy2014two}
Barry~M McCoy and Tai~Tsun Wu, \emph{The two-dimensional {I}sing model},
  Courier Corporation, 2014.

\bibitem[OB87]{OwcBax87}
AL~Owczarek and Rodney~J Baxter, \emph{A class of interaction-round-a-face
  models and its equivalence with an ice-type model}, Journal of statistical
  physics \textbf{49} (1987), no.~5-6, 1093--1115.

\bibitem[Ons44]{onsager1944crystal}
Lars Onsager, \emph{Crystal statistics. {I}. {A} two-dimensional model with an
  order-disorder transition}, Physical Review \textbf{65} (1944), no.~3-4, 117.

\bibitem[OSSS05]{OdoSakSchSer05}
Ryan O'Donnell, Michael Saks, Oded Schramm, and Rocco~A Servedio, \emph{Every
  decision tree has an influential variable}, 46th Annual IEEE Symposium on
  Foundations of Computer Science (FOCS'05), IEEE, 2005, pp.~31--39.

\bibitem[Pau35]{Pau35}
Linus Pauling, \emph{The structure and entropy of ice and of other crystals
  with some randomness of atomic arrangement}, Journal of the American Chemical
  Society \textbf{57} (1935), no.~12, 2680--2684.

\bibitem[Pel17]{Pel17}
Ron Peled, \emph{High-dimensional {L}ipschitz functions are typically flat},
  Ann. Probab. \textbf{45} (2017), no.~3, 1351--1447.

\bibitem[Pet08]{Pet08}
G{\'a}bor Pete, \emph{Corner percolation on $\mathbb{Z}^2$ and the square root
  of 17}, The Annals of Probability \textbf{36} (2008), no.~5, 1711--1747.

\bibitem[Pfi82]{Pfi82}
C.~E. Pfister, \emph{Phase transitions in the {A}shkin--{T}eller model},
  Journal of Statistical Physics \textbf{29} (1982), no.~1, 113–116.

\bibitem[PS19]{PelSpi17}
Ron Peled and Yinon Spinka, \emph{Lectures on the spin and loop $o(n)$ models},
  Sojourns in Probability Theory and Statistical Physics-I, Springer, 2019,
  pp.~246--320.

\bibitem[PV97]{PfiVel97}
C-E Pfister and Yvan Velenik, \emph{Random-cluster representation of the
  {A}shkin--{T}eller model}, Journal of Statistical Physics \textbf{88} (1997),
  no.~5-6, 1295--1331.

\bibitem[Res10]{Res10}
N~Reshetikhin, \emph{Lectures on the integrability of the six-vertex model},
  Exact methods in low-dimensional statistical physics and quantum computing
  (2010), 197--266.

\bibitem[RS20]{RaySpi19}
Gourab Ray and Yinon Spinka, \emph{A short proof of the discontinuity of phase
  transition in the planar random-cluster model with $q>4$}, Communications in
  Mathematical Physics \textbf{378} (2020), no.~3, 1977--1988.

\bibitem[RS22]{RaySpi19b}
\bysame, \emph{Finitary codings for gradient models and a new graphical
  representation for the six-vertex model}, Random Structures \& Algorithms
  \textbf{61} (2022), no.~1, 193--232.

\bibitem[RS23]{ray2020proper}
\bysame, \emph{Proper $3$-colorings of $\mathbb {Z}^{2}$ are {B}ernoulli},
  Ergodic Theory and Dynamical Systems \textbf{43} (2023), no.~6, 2002–2027.

\bibitem[Rys63]{Rys63}
Franz Rys, \emph{{\"U}ber ein zweidimensionales klassisches
  {K}onfigurationsmodell}, Helvetica Physica Acta, vol.~36, Birkhauser Verlag
  AG Viadukstrasse 40-44, PO Box 133, CH-4010 Basel, Switzerland, 1963, p.~537.

\bibitem[She05]{She05}
Scott Sheffield, \emph{Random surfaces}, Ast{\'e}risque (2005), no.~304,
  vi+175.

\bibitem[Sim93]{simon2014statistical}
Barry Simon, \emph{The statistical mechanics of lattice gases}, vol.~1,
  Princeton University Press, 1993.

\bibitem[Sla41]{Sla41}
John~C Slater, \emph{Theory of the transition in {KH2PO4}}, The Journal of
  Chemical Physics \textbf{9} (1941), no.~1, 16--33.

\bibitem[Smi10]{Smi10}
Stanislav Smirnov, \emph{Conformal invariance in random cluster models. {I}.
  {H}olomorphic fermions in the {I}sing model}, Ann. of Math. (2) \textbf{172}
  (2010), no.~2, 1435--1467.

\bibitem[Sut67]{sutherland1967exact}
Bill Sutherland, \emph{Exact solution of a two-dimensional model for
  hydrogen-bonded crystals}, Physical Review Letters \textbf{19} (1967), no.~3,
  103.

\bibitem[TL71]{LieTem71}
Harold~NV Temperley and Elliott~H Lieb, \emph{Relations between the
  ‘percolation’ and ‘colouring’ problem and other graph-theoretical
  problems associated with regular planar lattices: some exact results for the
  ‘percolation’ problem}, Proc. R. Soc. Lond. A \textbf{322} (1971),
  no.~1549, 251--280.

\bibitem[Weg72]{Weg72}
FJ~Wegner, \emph{Duality relation between the {A}shkin--{T}eller and the
  eight-vertex model}, Journal of Physics C: Solid State Physics \textbf{5}
  (1972), no.~11, L131.

\bibitem[YY66]{YanYan66}
Chen-Ning Yang and Chen-Ping Yang, \emph{One-dimensional chain of anisotropic
  spin-spin interactions. {I}. {P}roof of {B}ethe's hypothesis for ground state
  in a finite system}, Physical Review \textbf{150} (1966), no.~1, 321.

\end{thebibliography}

\end{document}